%% file: main_conser.tex
\documentclass[11pt]{article}

\usepackage{graphicx}
\usepackage{latexsym,amsmath,amsfonts,amscd, amsthm, dsfont}
\usepackage{bm,color}
\usepackage{epsfig,verbatim,epstopdf,graphics}
\usepackage{subfigure}
\usepackage{changebar}
\usepackage{multirow}

\usepackage[ruled,vlined]{algorithm2e}

\usepackage{yhmath}
 \usepackage{booktabs} 
 \usepackage{tikz}
\usepackage{verbatim}
\usetikzlibrary{arrows,backgrounds,snakes,shapes}

 \numberwithin{equation}{section}

\graphicspath{{./}{./figure/}}
\allowdisplaybreaks

\newtheoremstyle{plainNoItalics}{}{}{\normalfont}{}{\bfseries}{.}{ }{}

\theoremstyle{plain}
\newtheorem{thm}{Theorem}[section]

\theoremstyle{plainNoItalics}

\newtheorem{rem}[thm]{Remark}
\newtheorem{prop}[thm]{Proposition}
\newtheorem{exa}[thm]{Example}

\newcommand{\bx}{{\bf x}}

\newcommand{\bv}{{\bf v}}
\newcommand{\bw}{{\bf w}}
\newcommand{\bB}{{\bf B}}
\newcommand{\bE}{{\bf E}}
\newcommand{\bJ}{{\bf J}}

\newcommand{\ba}{{\bf a}}

\newcommand{\bU}{{\bf U}}

\newcommand{\bV}{{\bf V}}

\newcommand{\bS}{{\bf S}}
\newcommand{\mT}{{\mathcal T}}

\newcommand{\mM}{{\mathcal M}}

\newcommand{\br}{{\bf r}}

\newcommand{\innerw}[1]{\langle {#1} \rangle_{\bf w}}

\newcommand{\beq}{\begin{equation}}
\newcommand{\eeq}{\end{equation}}
\newcommand{\bit}{\begin{itemize}}
\newcommand{\eit}{\end{itemize}}
\newcommand{\be}{\begin{eqnarray}}
\newcommand{\ee}{\end{eqnarray}}
\newcommand{\beno}{\begin{eqnarray*}}
\newcommand{\eeno}{\end{eqnarray*}}

\newcommand{\Rmnum}[1]{\expandafter\@slowromancap\romannumeral #1@}
\makeatother

\begin{document}

\baselineskip=1.8pc


\input{title}

\input{introduction}
\input{algorithm}

\input{numerical}

\input{conclusion}


\bibliographystyle{abbrv}
\bibliography{refer,ref_guo,ref_cheng,ref_cheng_2}

\end{document}

%% file: title.tex
\begin{center}
{\bf
A conservative low rank tensor method for the Vlasov dynamics 
}
\end{center}

\vspace{.2in}
\centerline{
 Wei Guo\footnote{
Department of Mathematics and Statistics, Texas Tech University, Lubbock, TX, 70409. E-mail:
weimath.guo@ttu.edu. Research is supported by NSF grant NSF-DMS-1830838 and NSF-DMS-2111383, Air Force Office of Scientific Research FA9550-18-1-0257.
} and 
Jing-Mei Qiu\footnote{Department of Mathematical Sciences, University of Delaware, Newark, DE, 19716. E-mail: jingqiu@udel.edu. Research supported by NSF grant NSF-DMS-1818924 and 2111253, Air Force Office of Scientific Research FA9550-18-1-0257.}
}

\bigskip
\noindent
{\bf Abstract.} In this paper, we propose a conservative low rank tensor method to approximate nonlinear Vlasov solutions. The low rank approach is based on our earlier work \cite{guo2021lowrank}. It takes advantage of the fact that the differential operators in the Vlasov equation are tensor friendly, based on which we propose to dynamically and adaptively build up low rank solution basis by adding new basis functions from discretization of the differential equation, and removing basis from a singular value decomposition (SVD)-type truncation procedure. For the discretization, we adopt a high order finite difference spatial discretization together with a second order strong stability preserving multi-step time discretization. 

While the SVD truncation will remove the redundancy in representing the high dimensional Vlasov solution, it will destroy the conservation properties of the associated full conservative scheme. In this paper, we develop a conservative truncation procedure with conservation of mass, momentum and kinetic energy densities. The conservative truncation is achieved by an orthogonal projection onto a subspace spanned by $1$, $v$ and $v^2$ in the velocity space associated with a weighted inner product. Then the algorithm performs a weighted SVD truncation of the remainder, which involves a scaling, followed by the standard SVD truncation and rescaling back. The algorithm is further developed in high dimensions with hierarchical Tucker tensor decomposition of high dimensional Vlasov solutions, overcoming the curse of dimensionality. An extensive set of nonlinear Vlasov examples are performed to show the effectiveness and conservation property of proposed conservative low rank approach. Comparison is performed against the non-conservative low rank tensor approach on conservation history of mass, momentum and energy.

\vfill

\newpage

%% file: introduction.tex

\section{Introduction}

The Vlasov-Poisson (VP) system is known as a fundamental model in plasma physics which describes the dynamics of dilute charged particles due to self-induced electrostatic forces from a statistical viewpoint. The numerical challenges of realistic Vlasov simulations include the high-dimensionality of the phase space, features with multiple scales in time and in phase space, preservation of physical invariants, among many others.

The Particle-In-Cell method employs a collection of sampled macro particles to represent the distribution function and is widely used for the plasma simulations mainly due to its distinct dimension-independent  rate of convergence \cite{dawson1983particle,birdsall2004plasma}. On the other hand, such a method suffers the inherent statistical noise, and hence may fail to accurately capture the physics of interest. For example,  if the resolution of the tail of the distribution function is desired, a noise-free grid-based deterministic method is a better choice.  Over the past two decades, the development of deterministic Vlasov solvers has attracted lots of research interest, see e.g. \cite{filbet2003comparison}.  Meanwhile, a conventional deterministic Vlasov simulation in a realistic and high-dimensional setting is prohibitively expensive because of  the curse of the dimensionality and the associated huge computational and storage cost. Several dimension reduction techniques have been developed in the literature to alleviate the curse of dimensionality for the Vlasov equation and other high-dimensional partial differential equations (PDEs). One such example is the sparse grid approach \cite{smolyak1963quadrature,zenger1991sparse,griebel1990parallelizable}, which can effectively reduce the computational complexity and is well-suited for the problems with moderately high dimensions. For the Vlasov simulations, we mention the sparse grid semi-Lagrangian (SL) method \cite{kormann2016sparse} and the sparse grid discontinuous Galerkin method \cite{guo2016sparse1,tao2018sparseguo}. Recently, the tensor approach emerges as a promising tool for feasible high-dimensional simulations. Such an approach aims to extract the underlying low rank structure of the data with advanced tensor decomposition, potentially breaking the curse of dimensionality. The popular tensor formats include the canonical polyadic (CP) format \cite{hitchcock1927expression,carroll1970analysis,harshman1970foundations,kolda2009tensor}, Tucker format \cite{tucker1966some,de2000multilinear}, hierarchical Tucker (HT) format \cite{hackbusch2009new,grasedyck2010hierarchical}, and tensor train (TT) format \cite{oseledets2009breaking,oseledets2011tensor,oseledets2012solution}. There are several pioneering works employing low rank tensor approach for nonlinear simulations, including the low rank SL method in the TT format  \cite{kormann2015semi}, a low rank method with the CP format based on the underlying Hamiltonian formulation \cite{ehrlacher2017dynamical}, a dynamical low rank method proposed in \cite{einkemmer2018low,einkemmer2020low} for which the dynamical low rank approximation of the Vlasov solution is evolved on the low rank manifold using a tangent space projection, and a dynamical tensor approximations for high dimensional linear and nonlinear PDEs based on functional tensor decomposition and dynamic tensor approximation \cite{dektor2020dynamically}.
 
With the existing understanding of the low rank solution structure for the Vlasov dynamics, as well as the observation that the differential operator in the Vlasov equation can be represented in the tensorized form, in \cite{guo2021lowrank}, we developed a low rank tensor approach to dynamically and adaptively build up low rank solution basis, by adding new basis functions from discretization of the PDE and then removing basis from an singular value decomposition (SVD)-type truncation procedure. In particular, we start from a low rank solution in a tensor format and add additional basis by applying the well-established high order finite difference upwind method coupled with the strong-stability-preserving (SSP) multi-step time discretizations \cite{gottlieb2011strong}; the solutions are being further updated by an SVD-type truncation to remove redundant bases.  We further generalize the algorithm to high-dimensional cases with the HT decomposition, which attains a storage complexity that is linearly scaled with the dimension, overcoming the curse of dimensionality. An issue associated with the low rank algorithm in \cite{guo2021lowrank} is the loss of mass conservation in the SVD truncation step. In fact, it is well-known that the VP system conserves mass, momentum, and energy \emph{locally} by respecting a set of macroscopic moment equations. The high order finite difference scheme with an SSP multi-step time integrator can be shown to locally preserve the macroscopic mass and momentum, i.e. when taking moments of the kinetic scheme, one will arrive at a consistent and locally conservative discretization of the corresponding macroscopic equations.  Here ``locally conservative" means that the scheme can be written in a flux difference form, where the fluxes represent the amount of macroscopic quantities transported between neighboring computational cells. However, the SVD truncation in the low rank algorithm would destroy local and global conservation property of the original numerical scheme. 

There exist several techniques developed to correct the conservation errors for low rank methods in various settings.   In \cite{kormann2015semi}, the low rank solution is rescaled so that  the total mass is conserved, and a similar mass correction technique is proposed in \cite{peng2020low} for a dynamical low rank method. 
In \cite{allmann2022parallel}, moment fitting is applied to the low rank solution so that the corrected moments match those solved from the macroscopic fluid equations.
In \cite{einkemmer2019quasi}, a dynamical low rank method with Lagrangian multipliers is developed to improve conservation properties for the total mass and momentum as well as local projected moment equations. 
More recently, along the same line, the truly local conservation of mass, momentum, and energy is attained for the dynamical low rank method \cite{einkemmer2021mass}. The idea is to fix certain basis functions in the dynamical low rank approximation and employ a modified Petrov–Galerkin formulation which is compatible with the remainder of the approximation.
 
Inspired by the recent work in \cite{einkemmer2021mass}, in this paper, we develop a novel locally mass, momentum and energy conservative truncation algorithm under the framework of the low rank tensor approach \cite{guo2021lowrank}. In particular, we apply an orthogonal projection operator to the low rank solution in a weighted inner product spaces of phase variables, to extract exactly the first few moments (i.e. the mass, momentum and kinetic energy density functions) of the low rank solution after the ``add basis" step; a weighted SVD truncation is then applied to the remainder of the projection. Such a truncation is called ``conservative truncation" throughout the paper. We further develop the conservative truncation algorithm for the 2D2V VP system with the HT tensor format using a dimension tree that separates the spatial and phase variables. For high order HT tensors, an additional projection step is needed, after the high order SVD (HOSVD) truncation of the remaining term, to ensure conservation of macroscopic variables from the first few moments. The low rank tensor algorithm with the conservative projection is theoretically proved and numerically verified to be locally mass and momentum conservative. 

This paper is organized as follows. In Section 2, we introduce the kinetic Vlasov model and the corresponding macroscopic conservation laws. In Section 3, we first review the low rank tensor approach for the 1D1V Vlasov equation in Section 3.1; then we introduce the orthogonal projection at the continuous level to extract exactly the first few moments of Vlasov solution in Section 3.2; we further extend such an orthogonal projection to the discrete level and develop the main conservative truncation algorithm in Section 3.3. In Section 4, we develop the conservative truncation algorithm for the 2D2V Vlasov model, and show that the low rank algorithm locally conserves the mass and momentum at the macroscopic level. In Section 5, we present an extensive set of 1D1V and 2D2V numerical results to demonstrate the effectiveness and the conservation properties of the low rank tensor algorithm.  We conclude the main contributions of the paper and comment on future directions in Section 6. 

%% file: algorithm.tex
\section{The kinetic Vlasov model and the corresponding macroscopic systems}
We consider the VP system 
\beq
\frac{\partial f}{\partial t}  
+  {\bf{v}} \cdot \nabla_{\bf{x}}  f 
+ {\bf{E}} ({\bf{x}},t) \cdot \nabla_{\bf{v}}  f = 0,
\label{vlasov1}
\eeq
\beq
 {\bf E}( {\bf x},t) = - \nabla_{\bf x} \phi({\bf x},t),  \quad -\triangle_{\bf x} \phi ({\bf x},t) = {{\bf \rho} ({\bf x},t)} - \rho_0,
\label{poisson}
\eeq
which describes the probability distribution function $f({\bf x}, {\bf v},t)$ of electrons in collisionless plasma. 
Here ${\bf E}$ is the electric field, and $\phi$ is the self-consistent electrostatic potential. $f$ couples to the long range fields via the charge density, ${\bf \rho}({\bf x},t) = \int_{\Omega_{\bv}} f({\bf x}, {\bf v},t) d {\bf v} $, where we take the limit of uniformly distributed infinitely massive ions in the background.  The VP system describes the movement of electrons due to self-induced electric field ${\bf E}$ determined by the Poisson equation. 

The Vlasov dynamics are well-known to conserve several physical invariants. In particular, let 
\begin{align}
\label{eq: mass_d}
\mbox{mass density:}& \quad \rho (\bx, t) = \int_{\Omega_{\bv}} f(\bx, \bv,t) d \bv, \\
\label{eq: current_d}
\mbox{current density:} &\quad\bJ (\bx, t) = \int_{\Omega_{\bv}}f(\bx, \bv,t) \bv d \bv,\\
\label{eq: kenergy_d}
\mbox{kinetic energy density:} &\quad \kappa(\bx,t) = \frac{1}{2} \int_{\Omega_{\bv}} \bv^{2}  f(\bx, \bv,t) d \bv,\\
\label{eq: energy_d}
\mbox{energy density:} &\quad e(\bx,t)=\kappa(\bx,t)+\frac{1}{2} \bE(\bx)^2.
\end{align}
Then, by taking a few first moments of the Vlasov equation,
the following conservation laws of mass, momentum and energy can be derived
\begin{align}
\rho_t + \nabla_\bx \cdot \bJ &= 0\label{eq:mass}\\
\partial_{t} \bJ +\nabla_{\bx} \cdot \mathbf{\sigma}&= \rho\bE \label{eq:mom}\\
\partial_{t} e +\nabla_{\bx} \cdot \mathbf{Q}& =0,\label{eq:ener} 
\end{align}
where $ \sigma(t, \bx)=\int_{\Omega_{\bv}}(\bv \otimes \bv) f(\bx, \bv,t) d \bv$ and $\mathbf{Q}(\bx,t) =\frac12\int_{\Omega_{\bv}}\bv\bv^2 f(\bx, \bv,t) d \bv$. 
It is well-known that the local conservation property of numerical schemes is critical to capture correct entropy solutions of hyperbolic systems such as macroscopic equations \eqref{eq:mass}-\eqref{eq:ener}. 

\section{A low rank tensor approach for the Vlasov dynamics with local conservation}

For simplicity of illustrating the basis idea, we only discuss a 1D1V example in this section.

\subsection{Review of a low rank tensor approach for Vlasov dynamics \cite{guo2021lowrank}}
The low rank tensor approach \cite{guo2021lowrank} is designed based on the assumption that our solution at time $t$ has a low-rank representation in the form of 
\begin{equation}
\label{eq: fn1}
f(x, v, t) = \sum_{l=1}^{r} \left(C_l(t) \ \ U_l^{(1)}(x, t) U_l^{(2)}(v, t)\right),
\end{equation}
where $\left\{U_l^{(1)}(x, t)\right\}_{l=1}^{r}$ and $\left\{U_l^{(2)}(v, t)\right\}_{l=1}^{r}$ are a set of time-dependent low rank {unit length orthogonal} basis in $x$ and $v$ directions, respectively, $C_l$ is the coefficient for the basis $U_l^{(1)}(x, t)U_l^{(2)}(v, t)$, and $r$ is the representation rank.  \eqref{eq: fn1} can be viewed as a Schmidt decomposition of functions in $(x, v)$ by truncating small singular values up to rank $r$. 

We assume a finite difference discretization of $f$ on a truncated 1D1V domain of $[x_{\min}, x_{\max}] \times [-v_{\max}, v_{\max}]$ with uniform $N_x \times N_v$ grid points 
\beq
\label{eq: x_grid}
x_{\text{grid}}: \quad x_{\min}=x_1< \cdots < x_i < \cdots < x_{N_x} = x_{\max}, 
\eeq
\beq
\label{eq: v_grid}
v_{\text{grid}}: \quad -v_{\max}=v_1< \cdots < v_j <\cdots < v_{N_v} = v_{\max}. 
\eeq
The numerical solution ${\bf f} \in \mathbb{R}^{N_x \times N_v}$, as an approximation to point values of the solution on the grids \eqref{eq: x_grid}-\eqref{eq: v_grid}, has the corresponding low-rank counterpart to \eqref{eq: fn1} as
\begin{equation}
\label{eq: fn2}
{\bf f} = \sum_{l=1}^{r} \left(C_l \ \ {\bf U}_l^{(1)}  \otimes {\bf U}_l^{(2)}\right), \quad 
(\mbox{or element-wise:} \quad
{\bf f}_{ij} = \sum_{l=1}^{r}  C_l {\bf U}_{l, i}^{(1)} {\bf U}_{l, j}^{(2)}), 
\end{equation}
where ${\bf U}_l^{(1)} \in \mathbb{R}^{N_x}$ and ${\bf U}_l^{(2)} \in \mathbb{R}^{N_v}$ can be viewed as approximations to corresponding basis functions in \eqref{eq: fn1}. \eqref{eq: fn2} can also be viewed as an SVD of the matrix ${\bf f} \in \mathbb{R}^{N_x \times N_v}$. The associated storage cost is $\mathcal{O}(r N)$, where we assume $N=N_x = N_v$. 

Our low rank  tensor approach adaptively updates low-rank basis and associated coefficients by two steps: an adding basis step by conservative hyperbolic solvers and a removing basis step via an SVD-type truncation. We consider a simple first order forward Euler discretization of 1D1V Vlasov equation \eqref{vlasov1} to illustrate the main idea.  We assume the solution in the form of \eqref{eq: fn2} with superscript $n$ for the solution at $t^n$.
\begin{enumerate}
\item {\em Add basis and obtain an intermediate solution ${\bf f}^{n+1, *}$.} 
A forward Euler discretization of \eqref{vlasov1} gives
\begin{equation}
\label{eq: fn3}
{f}^{n+1, *} = {f}^n - \Delta t (v \partial_x ({f}^n) + E^n \partial_v ({f}^n)).
\end{equation}
Here the electric field $E^n$ is solved by a Poisson solver. Thanks to the tensor friendly form of the Vlasov equation, at the fully discrete level, ${\bf f}^{n+1, *}$ can be represented in the following low-rank format:
\begin{equation}
\label{eq:lowrankmethod}
{\bf f}^{n+1, *} = \sum_{l=1}^{r^n} C_l^n \left[  \left( {\bf U}_l^{(1), n} \otimes {\bf U}_l^{(2), n}\right)  - \Delta t 
\left( D_x {\bf U}_l^{(1), n} \otimes v \star {\bf U}_l^{(2), n} + E^n \star {\bf U}_l^{(1), n} \otimes D_v {\bf U}_l^{(2), n}
\right)\right],
\end{equation}
where $D_x$ and $D_v$ represent high order {\em locally conservative upwind discretization} of spatial differentiation terms, and $\star$ denotes an element-wise multiplication operation. 
For example the discretization of $D_x {\bf U}_l^{(1), n} \otimes v \star {\bf U}_l^{(2), n}$ follows 
\begin{equation}
D^+_x {\bf U}_j^{(1), n} \otimes v^+ \star {\bf U}_j^{(2), n} + D^-_x {\bf U}_j^{(1), n} \otimes v^- \star {\bf U}_j^{(2), n}, 
\end{equation}
where $D^+_x$ and $D^-_x$ are a fifth order finite difference discretization of positive and negative velocities respectively, with $v^+ = \max(v, 0)$ and $v^-=\min(v, 0)$. 
Similarly, the discretization of $E^n \star {\bf U}_l^{(1), n} \otimes D_v {\bf U}_l^{(2), n}$ follows 
 \beq
 \bE^{n, +} \star {\bf U}_j^{(1), n} \otimes D^+_v {\bf U}_j^{(2), n}+\bE^{n, -} \star {\bf U}_j^{(1), n} \otimes D^-_v {\bf U}_j^{(2), n}
 \eeq
where $D^+_v$ and $D^-_v$ are a fifth order upwind finite difference discretization of positive and negative velocities respectively, with $\bE^+ = \max(\bE, 0)$ and $\bE^-=\min(\bE, 0)$. 
\item {\em Remove basis of ${\bf f}^{n+1, *}$ to update solution ${\bf f}^{n+1}$.} Since the number of basis has increased in a single step update, we perform an SVD-type truncation to remove redundant bases with a prescribed threshold $\varepsilon$. The truncation step has no guarantee of any mass, momentum or energy conservation property. 
The removing basis step costs $\mathcal{O}(r^2 N + r^3)$, where $r$ is the SVD rank of the numerical solution.  
\end{enumerate}
In this two-step process, both the basis and coefficients are updated. Extensions to schemes with high order spatial and temporal discretizations and to high dimensional problems, are developed in \cite{guo2021lowrank}. The low rank approach  \cite{guo2021lowrank} is built upon the classical high order methods for conservation laws and kinetic equations, yet it optimizes the computational efficiency by dynamically building low rank global basis and updating the corresponding coefficients via a SVD truncation procedure. While the SVD truncation significantly reduces the computational storage and CPU time, it also destroys the conservation property in the truncation step.

\subsection{An orthogonal projection with mass, momentum and energy conservation} 
To preserve the local mass, momentum and energy in the truncation step, following a similar idea in \cite{einkemmer2021mass}, we define a weighted inner product space and the associated norm for functions in $v$ as follows, 
\[
L^2_w(\mathbb{R}) = \{f: \mathbb{R} \rightarrow \mathbb{R}: \|f\|_w < \infty\}.
\]
where
\beq
\label{eq: inner_prod}
\langle f,  g \rangle_w = \int f(v) g(v)w(v)dv, \quad \|f\|_w =\sqrt{\langle f,  f \rangle_w}.
\eeq
The weight function $w(v)$ is chosen to have exponential decay so that $N \doteq \text{span}(1, v, v^2) \subset L^2_w(\mathbb{R})$. For example, we can take $w(v) = \exp(-v^2/2)$. Now we are ready to define an orthogonal decomposition of $f$ with respect to the subspace $N$. 

In the context of Vlasov dynamics, we consider $f(x, v)$ with a $v$-direction projection. 
We first scale function
\beq
\label{eq: rescale_f}
\tilde{f} = \frac{1}{w(v)}f. 
\eeq
Then we find $P_N(\tilde{f})\in N$ s.t.
\beq
\label{eq: projc}
\int P_{N}(\tilde{f}(x, v))  g(v) w(v)dv
= \int \tilde{f}(x, v)  g(v) w(v)dv = \int {f}(x, v)  g(v)dv, \quad \forall g\in N. 
\eeq
By taking $g = 1, v, v^2$, $P_{N}(\tilde{f}(x, v)) w(v)$ preserves the mass density $\rho(x)$, current density $J(x)$ and kinetic energy density $\kappa(x)$ of $f$ defined in \eqref{eq: mass_d}-\eqref{eq: kenergy_d}. In fact, we have the \emph{conservative} decomposition of $f$ as in the following Proposition.
\begin{prop}
\label{prop: f_decomp}
Let $w(v) = \exp(-v^2/2)$ and the weighted inner product defined as in \eqref{eq: inner_prod}.
 $f\in L^2_w(\mathbb{R})$ can be decomposed as
\beq
\label{eq: f_decom}
f = w(v) (P_{{N}}(\tilde{f}) + (I-P_{{N}})(\tilde{f})) 
\doteq w(v)(\tilde{f}_1 + \tilde{f}_2)
\doteq f_1 + f_2,
\eeq
where $f_1$ admits an explicit representation
 \begin{equation}\label{eq:f1_cont}
 f_1= w(v) \tilde{f}_1 = w(v) (c_1  + c_2 v + c_3 (v^2-c)),
  \end{equation}
with $c= \frac{\langle1, v^2\rangle_w}{\langle1, 1\rangle_w}$,  
 $c_1 = \frac{\rho(x)}{\|1\|_w^2}$, $c_2 =  \frac{J(x)}{\|v\|_w^2}$, and 
$c_3=\frac{2\kappa(x)-\rho(x)c}{ \|v^2-c\|_w^2}$, and $\rho(x)$, $J(x)$ and $\kappa(x)$ are as defined in \eqref{eq: mass_d}-\eqref{eq: kenergy_d}. 
$f_1$ preserves the mass, momentum and kinetic energy density of $f$, while $f_2$ has zero mass, current and kinetic energy density.
\end{prop}
\begin{proof}
We first compute $c$, so that $\{1, v, v^2-c\}$ is a set of orthogonal basis. By taking $g=1, v, v^2-c$ in \eqref{eq: projc}, we can determine the constants $c_1$, $c_2$, and $c_3$. By construction of the orthogonal projection,  $f_1$ preserves the mass, momentum and kinetic energy density of $f$, while $f_2$ has zero of them.
\end{proof}

\subsection{Low rank truncation with mass, momentum and energy conservation}
Similar to the weighted inner product \eqref{eq: inner_prod}, the projection operator \eqref{eq: projc} and the conservative decomposition of $f$ in \eqref{eq: f_decom}, we define their corresponding discrete analogue. Consider ${\bf f}, {\bf g} \in \mathbb{R}^{N_v}$, where ${\bf f}$ and ${\bf g}$ can be viewed as function values of $f$ and $g$ at $v_\text{grid}$ \eqref{eq: v_grid}, respectively.
We define the standard $l^2$ inner product as
\beq
\label{eq: inner_prod_2_d}
\langle {\bf f},  {\bf g} \rangle = h_v \sum_j f_j g_j,  
\eeq
where $h_v$ is the mesh size in $v$-direction, serving as the quadrature weights for the uniform $v_{\text{grid}}$. We also define weighted inner product and the associated norm as
\beq
\label{eq: inner_prod_d}
\langle {\bf f},  {\bf g} \rangle_{\bf w} = \sum_j f_j g_j w_j, \quad 
\quad \|{\bf f}\|_{\bf w} =\sqrt{\langle {\bf f},  {\bf f} \rangle_{\bf w}}
\eeq
${\bf w}\in \mathbb{R}^{N_v}$ with $w_j = w(v_j) h_v$ is the quadrature weights for $v$-integration with weight function $w(v)$. Correspondingly, we let
\[
l^2_{\bf w} = \{{\bf f}\in\mathbb{R}^{N_v}: \|{\bf f}\|_{\bf w} < \infty\}.
\]
We denote ${\bf 1}_v\in \mathbb{R}^{N_v}$ the vector of all ones, $\bv$ and $\bv^2$ $\in \mathbb{R}^{N_v}$  the coordinates and coordinates square of $v_\text{grid}$ \eqref{eq: v_grid}, respectively. Consider the subspace $\mathcal{N}\doteq \text{span}\{{\bf 1}_v, \bv, \bv^2\}\subset l^2_w$, a conservative low rank truncation of numerical solution ${\bf f}\in \mathbb{R}^{N_x\times N_v} $ written in the low rank form of \eqref{eq: fn2} can be obtained from steps below. 
\begin{enumerate}
\item {\bf Compute macroscopic quantities of ${\bf f}$.} We compute the discrete macroscopic charge, current and kinetic energy density ${\boldsymbol \rho}$, ${\bf J}$ and ${\boldsymbol \kappa} \in \mathbb{R}^{N_x}$ by quadrature 
   \begin{align}
\left(\begin{array}{l}
{\boldsymbol \rho}\\
{\bf J}\\
{\boldsymbol \kappa} 
\end{array}
\right )
 = \sum_{l=1}^{r} C_l
 \left 
 \langle \bU^{(2)}_{l}, 
 \left(\begin{array}{l}
 {\bf 1}_v \\
\bv\\
\frac12\bv^2
\end{array}
\right )
\right \rangle
\ \bU^{(1)}_l. 
\label{eq:rho_j_kappa}
\end{align}   
They are the discrete analog of \eqref{eq: mass_d}, \eqref{eq: current_d} and \eqref{eq: energy_d}, i.e.
\begin{equation*}
    \bm{\rho}(i) = \langle {\bf f}(i, :), {\bf 1}_v\rangle, \quad
    {\bf J}(i) = \langle {\bf f}(i, :), \bm{v} \rangle, \quad
    \bm{\kappa}(i) =  \langle \bm{ f}(i, :), \frac12\bm{ v}^2 \rangle, \quad
     \quad i=1, \cdots, {N_x}. 
\end{equation*}
\item {\bf Scale.} Similar to \eqref{eq: rescale_f}, we scale ${\bf f}$ as 
\beq
\label{eq: rescale}
\tilde{\bf f} = \frac{1}{{\bf w}} \star {\bf f} =  \sum_{l=1}^{r} \left(C_l \ \ {\bf U}_l^{(1)}  \otimes \left(\frac{1}{{\bf w}} \star  {\bf U}_l^{(2)}\right)\right).
\eeq
\item {\bf Project.} We perform an orthogonal projection of $\tilde{\bf f}$ with respect to the inner product \eqref{eq: inner_prod_d} onto subspace $\mathcal{N}$, i.e.
\beq
\label{eq: proj}
\langle P_{\mathcal{N}}(\tilde{\bf f}), {\bf g} \rangle_\bw
= \langle \tilde{\bf f}, {\bf g} \rangle_\bw, 
\quad \forall {\bf g}\in \mathcal{N}. 
\eeq
Similar to Proposition~\ref{prop: f_decomp}, ${\bf w}\star P_{\mathcal{N}}(\tilde{\bf f})$ preserves the mass, momentum and kinetic energy densities in the discrete sense. We show below in Proposition~\ref{prop: f_decomp_d} a conservative decomposition of ${\bf f}$.
\begin{prop}
\label{prop: f_decomp_d}
${\bf f}$ can be decomposed as
\beq
\label{eq: f_decom_d}
{\bf f} = {\bf w} \star (P_{\mathcal{N}}(\tilde{{\bf f}}) + (I-P_{\mathcal{N}})(\tilde{{\bf f}})) 
\doteq {\bf w} \star (\tilde{{\bf f}}_1 + \tilde{{\bf f}}_2)
\doteq {\bf f}_1 + {\bf f}_2
\eeq
where $f_1$ can be represented as  a rank three tensor
 \begin{equation}\label{eq:f1}
 {\bf f}_1= {\bf w} \star  \tilde{{\bf f}}_1 =  \sum_{k=1}^3 \ \mM_k \otimes ({\bf w} \star \bV_k),
  \end{equation}
   where $\bV_1 = {\bf 1}_v$, $\bV_2 ={\bf v}$, and $\bV_3 = {\bf v}^2- c {\bf 1}_v$, and 
   $c=\frac{\langle {\bf 1}_v, {\bf v}^2\rangle_{\bf w}}{\|{\bf 1}_v\|_{\bf w}^2}$, 
   \beq
   \label{eq: mM}
   \mM_1 = \frac{\boldsymbol \rho}{\|{\bf 1}_v\|_\bw^2}, \quad \mM_2 = \frac{\bf J}{\|{\bf v}\|_{\bf w}^2}, \quad  
    \mM_3 = \frac{2 {\boldsymbol \kappa}-c{\boldsymbol \rho}}{\|{\bf v}^2- c {\bf 1}_v\|_{\bf w}^2}.
    \eeq
   ${\boldsymbol \rho}$, ${\bf J}$ and ${\boldsymbol \kappa}$ are the discrete mass, momentum and kinetic energy density of ${\bf f}$ as in \eqref{eq:rho_j_kappa}. ${\bf f}_1$ preserves the discrete mass, momentum and kinetic energy density of ${\bf f}$, while ${\bf f}_2$ has zero of them, i.e.   $
\langle {\bf f}_2(i, :), {\bf 1}_v\rangle
=
\langle {\bf f}_2(i, :), {\bf v}\rangle
=
\langle {\bf f}_2(i, :), {\bf v}^2\rangle
=0, \quad i=1, \cdots, {N_x}.
$
\end{prop}
 \begin{proof} 
The proof follows in the same manner as Proposition \ref{prop: f_decomp}, but considering the discrete weighted inner product and the associated norm \eqref{eq: inner_prod_d}.  
\end{proof}
\item {\bf Truncate in $l^2_{\bf w}$.} To ensure the conservation in the truncation process, we propose to perform an SVD truncation of $\tilde{\bf f}_2$ in the decomposition \eqref{eq: f_decom_d} with respect to the weighted inner product \eqref{eq: inner_prod_d}. Because of the weight function, we scale $\tilde{\bf f}_2$ by ${\sqrt{\bf w}} \star \tilde{\bf f}_2$, where $\sqrt{\cdot}$ is in the element-wise sense and $\star$ is for the element-wise multiplication. Then, a standard SVD truncation procedure is applied, denoted as $\mathcal{T}_\varepsilon ({\sqrt{\bf w}} \star \tilde{\bf f}_2)$. Here and below, $\varepsilon$ is truncation threshold. Lastly, we rescale $\mathcal{T}_\varepsilon ({\sqrt{\bf w}} \star \tilde{\bf f}_2)$ back and obtain the truncated $\tilde{\bf f}_2$. Such a weighted truncation writes
\beq
\label{eq: cons_trun_f2}
\mathcal{T}^{\bf w}_{\varepsilon} (\tilde{\bf f}_2) =\frac{1}{\sqrt{\bf w}}  \star \mathcal{T}_\varepsilon (\sqrt{\bf w} \star \tilde{\bf f}_2).
\eeq
In summary, ${\bf f}_2$ is truncated to
\[{\bf w}\star\mathcal{T}^{\bf w}_{\varepsilon} (\tilde{\bf f}_2) = \sqrt{\bf w}  \star \mathcal{T}_\varepsilon (\sqrt{\bf w} \star \tilde{\bf f}_2)= \sqrt{\bf w}  \star \mathcal{T}_\varepsilon (\frac{1}{\sqrt{\bf w}} \star {\bf f}_2).
\] 
\item {\bf Update.} We obtain the low rank truncation of ${\bf f}$ with local mass, momentum and energy conservation, denoted as  
\beq
T_c({\bf f}) = {\bf f}_1 + {\bf w}\star\mathcal{T}^{\bf w}_{\varepsilon} (\tilde{\bf f}_2). 
\label{eq: Tc}
\eeq
We call the proposed truncation \eqref{eq: Tc} the conservative truncation. 
The following proposition guarantees that the local conservation of mass, momentum, and kinetic energy density of ${\bf f}$ is preserved in the proposed truncation procedure of $T_c({\bf f})$. 
\begin{prop}\label{prop:conservation}
$T_c({\bf f})$ has the same discrete charge, current and kinetic energy density as ${\bf f}$. 
\end{prop}
\begin{proof} With Proposition~\ref{prop: f_decomp_d}, it is sufficient to show that ${\bf w}\star\mathcal{T}^{\bf w}_{\varepsilon} (\tilde{\bf f}_2)$ has zero charge, current and kinetic energy density. Since $\tilde{\bf f}_2$ is orthogonal to $\mathcal{N}$ with respect to the weighted inner product \eqref{eq: inner_prod_d}, and its truncation is performed in the same inner product space, $\mathcal{T}^{\bf w}_{\varepsilon} (\tilde{\bf f}_2)$ is also orthogonal to $\mathcal{N}$, i.e.
\[
\langle \mathcal{T}^{\bf w}_{\varepsilon} (\tilde{\bf f}_2), g\rangle_{\bf w} =0, \quad g = 1,\, {\bf v},\, {\bf v}^2,
\]
i.e. ${\bf w}\star\mathcal{T}^{\bf w}_{\varepsilon} (\tilde{\bf f}_2)$ has zero discrete charge, current and kinetic energy density.
\end{proof}
\end{enumerate}

Next we establish the local conservation of mass and momentum in the low rank tensor approach with the conservative truncation \eqref{eq: Tc}. Since the full algorithm (without truncation) does not have energy conservation property, the low rank tensor scheme does not preserve energy conservation.  If the associated full-rank method is able to locally preserve the energy density, so will the corresponding the low rank method with the proposed conservative truncation. 
\begin{prop}(Local mass and momentum conservation of the low rank tensor approach with conservative truncation.) If the discrete differential operators $D_x,\,D_v$ employed are conservative, i.e., can be written in a flux difference form, and linear, i.e., can preserve linear relations, then the proposed low rank method with an SSP multi-step time integrator preserves the mass and momentum locally; that is the schemes for ${\boldsymbol\rho}^n$ and $\bJ^n$, from integrating the scheme on ${\bf f}^{n}$, are  consistent and conservative discretization of the macroscopic moment equations \eqref{eq:mass}-\eqref{eq:mom}. 
\label{prop: cons}
\end{prop}
\begin{proof} 
Without loss of generality, we prove the proposition for the low rank method with the forward Euler time integrator for simplicity.
By taking discrete moments of \eqref{eq:lowrankmethod} and from \eqref{eq:rho_j_kappa}, we derive the discrete evolution equations for the charge density ${\boldsymbol \rho}$ and current density $\bJ$ in $\mathbb{R}^{N_x}$,
\begin{align}\label{eq:rhofd}
	{\boldsymbol\rho}^{n+1,*} &= {\boldsymbol \rho}^{n} - \Delta t D_x(\bJ^{n}), \\
	\bJ^{n+1,*} &= 	\bJ^{n} - \Delta t (D_x({\boldsymbol\sigma}^{n}) - {\boldsymbol\rho}^n \star {\bf E}^n),\label{eq:Jfd}
\end{align}
where we have used the fact that $D_x$ and $D_v$ are linear, and the summation by parts
\begin{equation}\label{eq:sbp}
\langle D_v\bU^{(2)}_{l}, \bv \rangle = -\langle \bU^{(2)}_{l}, \mathbf{1}_v \rangle
\end{equation}
holds by assuming the basis $\bU^{(2)}_{l}$, $l =1,\ldots,r^n,$ vanishes at the boundary. 
Combining \eqref{eq:rhofd}-\eqref{eq:Jfd} with  the property of the conservative truncation from Proposition \ref{prop:conservation}, we have 
\begin{align}\label{eq:rhofd2}
	{\boldsymbol\rho}^{n+1} &={\boldsymbol\rho}^{n+1,*}= 	{\boldsymbol\rho}^{n} - \Delta t D_x(\bJ^{n}), \\
	\bJ^{n+1} &= \bJ^{n+1,*}=	\bJ^{n} - \Delta t (D_x(\sigma^{n}) - {\boldsymbol\rho}^n \star \bE^n),\label{eq:Jfd2}
\end{align}
which are  consistent and conservative discretization of  the macroscopic moment equations \eqref{eq:mass}-\eqref{eq:mom}.
\end{proof}

\begin{rem} If a WENO finite difference discretization is employed for $D_v$, then the local conservation of momentum cannot be guaranteed. This is because the WENO discretization does not satisfy the summation by parts property \eqref{eq:sbp}. A WENO finite difference discretization can be employed for $D_x$ for local conservation.
\end{rem}

\begin{rem}\label{rem:35} If only the mass conservation is desired, then we can modify the subspace  $\mathcal{N}=\text{span}\{{\bf 1}_v\}$ for projection; similarly, if only the mass and momentum conservation is desired, we let $\mathcal{N}=\text{span}\{{\bf 1}_v, \bv\}$. We denote the corresponding projections as $P_1$ and $P_2$, respectively, and denote the projection to 
$\mathcal{N}=\text{span}\{{\bf 1}_v, \bv, \bv^2\}$ as $P_3$. Numerical performances of different projection operators will be assessed in the Section \ref{sec:numerical} for numerical experiments. 
\end{rem}

\begin{rem} The choice of weight function $w(v)$ can affect the performance. In particular, if the weight function does not have sufficient decay, then the solution may not be close enough to zero at the boundary of the velocity domain, and large error could be incurred. On the other hand,  when rescaling $f_2$ by dividing ${w}$, small weights at the boundary can introduce numerical instability. 
\end{rem}

Finally, the proposed conservative truncation algorithm is summarized in the Algorithm 1.

\bigskip
\begin{algorithm}[H]
\label{alg: low_rank}
  \caption{The conservative truncation procedure for the 1D1V VP system.}
  \begin{itemize}
\item Input: the pre-compressed low rank solution at time $t^{n+1}$:
$${\bf f}^{n+1,*} = \sum_{l=1}^{r^*} C_l^*\ \bU_l^{(1),*}\otimes \bU_l^{(2),*}.$$
\item Output: the compressed low rank solution $f^{n+1}$ with the same density, current density, and kinetic energy functions as $f^{n+1,*}$.
\end{itemize}
  \begin{enumerate}
  \item Compute ${\boldsymbol \rho}^{n+1}$,  ${\bf J}^{n+1}$,  ${\boldsymbol \kappa}^{n+1}$ of ${\bf f}^{n+1,*}$ from \eqref{eq:rho_j_kappa}. 
  \item Compute $\mM^{n+1}$ by \eqref{eq: mM} from ${\boldsymbol \rho}^{n+1}$,  ${\bf J}^{n+1}$,  ${\boldsymbol \kappa}^{n+1}$.
  \item Compute ${\bf f}_1$ from \eqref{eq:f1}.
\item Perform the truncation of ${\bf f}_2$ by \eqref{eq: cons_trun_f2}. 
  \item Update the compressed low rank solution by \eqref{eq: Tc}.
\end{enumerate}
  \end{algorithm}

 \section{2D2V Vlasov-Poisson system by the HT format}
 We extend the proposed conservative algorithm to the 2D2V case by the HT format. 
 Below, we briefly review the fundamentals of the HT format for efficiently representing tensors in $d$ dimensions, and the low rank tensor method with the HT format for solving the 2D2V VP system \eqref{vlasov1}. 
\begin{equation}\label{eq:vp4d}
	f_t + v_1f_{x_1} + v_2f_{x_2}+ E_1f_{v_1} + E_2f_{v_2} = 0,
\end{equation}
where the electric field $(E_1, E_2)$ is solved from the coupled Poisson's equation. The macroscopic equations can be obtained from taking moments of \eqref{eq:vp4d} in the form of \eqref{eq:mass}-\eqref{eq:ener}.

\subsection{HT format for high order tensors} 
 We  denote the dimension index $D=\{1,2,\ldots,d\}$  and define a \emph{dimension tree} $\mathcal{T}$ which is a binary tree containing a subset $\alpha\subset D$ at each node. Furthermore, $\mathcal{T}$ has $D$ as the root node and $\{1\},\  \{2\},\ldots,\ \{d\}$ as the leaf nodes. Each non-leaf node $\alpha$ has two children nodes denoted as  $\alpha_{l}$ and $\alpha_{r}$ with $\alpha = \alpha_{l}\bigcup\alpha_r$ and $\alpha_{l}\bigcap\alpha_r = \emptyset$. For example, the dimension tree $\mathcal{T}$ given in Figure \ref{fig:dimtree} can be used to approximate $f(x_1,x_2,v_1,v_2)$ in \eqref{eq:vp4d} in the HT format. The efficiency of the HT format lies in the nestedness property \cite{hackbusch2009new}: for a non-leaf node $\alpha$ with two children nodes $\alpha_l,\,\alpha_r$, then 
 \begin{equation}
\text{range}(\mathcal{M}^{(\alpha)}(\ba))\subset  \text{range}(\mathcal{M}^{(\alpha_l)}(\ba)\otimes\mathcal{M}^{(\alpha_r)}(\ba)),
 \end{equation}
which implies that there exists a third order tensor $\bB^{(\alpha)}\in\mathbb{R}^{r_{\alpha_l}\times r_{\alpha_r}\times r_{\alpha}}$, known as the transfer tensor, such that
 \begin{equation} 
 \label{eq:htd_nested}
 \bU_{l_\alpha}^{(\alpha)} = \sum_{l_{\alpha_l}=1}^{r_{\alpha_l}}\sum_{l_{\alpha_r=1}}^{r_{\alpha_l}} \bB^{(\alpha)}_{l_{\alpha_l},l_{\alpha_r},l_{\alpha}}\bU_{l_{\alpha_l}}^{(\alpha_l)}\otimes \bU_{l_{\alpha_r}}^{(\alpha_r)},\quad l_\alpha =1,\ldots,r_\alpha.
 \end{equation}
 In other words, the frame vectors at the parent node can be recovered by those at the two children nodes $\alpha_l,\,\alpha_r$ with the transfer tensor. 
 By recursively making use of \eqref{eq:htd_nested}, a tensor in the HT format stores a frame at each leaf node and a third order transfer tensor at each non-leaf node based on a dimension tree.  Denote $ \br = \{r_\alpha\}_{\alpha\in\mathcal{T}}$ as the hierarchical ranks. The storage of the HT format scales as $\mathcal{O}(dr^3+drN)$, where $r=\max \br$. If $r$ is reasonably low, then the HT format avoids the curse of dimensionality. In summary, the HT format is fully characterized by the three key components, including a dimension tree, frames at leaf nodes and transfer tensors at non-leaf nodes, see Figure \ref{fig:dimtree} for the data layout. 
 
 \begin{figure}
\centering
\subfigure[]{
 \begin{tikzpicture}[
     level/.style={sibling distance=40mm/#1},
  every node/.style = {shape=rectangle, rounded corners,
    draw, align=center,
    top color=white, bottom color=blue!20}
    ]
  \node {$\{1,\,2,\,3,\,4\}$}
    child { node {$\{1,\,2\}$} 
    	child{ node{$\{1\}$}}
	child{ node{$\{2\}$}}
    }
    child { node {$\{3,\,4\}$}
    	child{ node{$\{3\}$}}
	child{ node{$\{4\}$}} 
	};
\end{tikzpicture}}\quad
\subfigure[]{
\begin{tikzpicture}[
  every node/.style = {shape=rectangle, rounded corners,
    draw, 
    top color=white, bottom color=blue!20},
    level/.style={sibling distance=40mm/#1}
    ]
  \node {$\bB^{(1,2,3,4)}$}
      child { node {$\bB^{(1,2)}$} 
    	child{ node{$\bU^{(1)}$}}
	child{ node{$\bU^{(2)}$}}
    }
    child { node {$\bB^{(3,4)}$}
    	child{ node{$\bU^{(3)}$}}
	child{ node{$\bU^{(4)}$}} 
	};
\end{tikzpicture}}
\vspace{0.2in}
\label{fig:dimtree}
 \caption{Dimension tree $\mathcal{T}$ to express fourth-order tensors in the HT format. }
\end{figure}
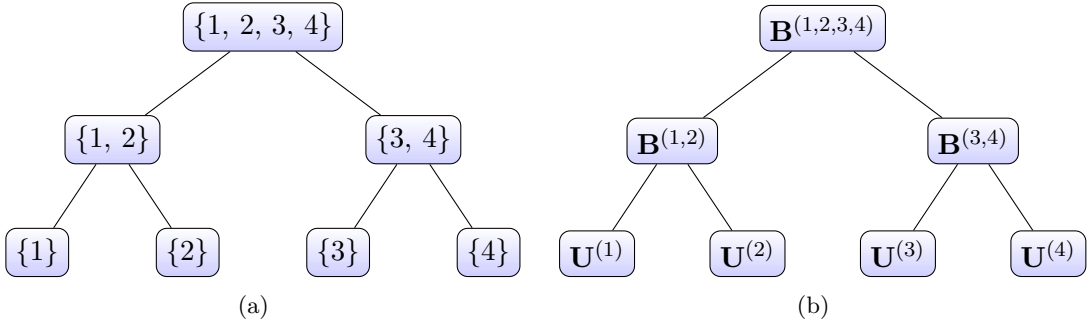


\subsection{A conservative low rank tensor method in HT for the 2D2V VP system}
We follow the low rank tensor method for updating the 2D2V VP solution in  \cite{guo2021lowrank}, but propose to perform a conservative hierarchical  HOSVD truncation with preservation of charge, current and kinetic energy density. We assume at each time step, the solution ${\bf f}$ is expressed as the fourth order tensor in the HT format with dimension tree $\mathcal{T}$ together with frames $\bU^{(1)},\bU^{(2)},\bU^{(3)},\bU^{(4)}$ at four leaf nodes, corresponding to directions $x_1,\,x_2,\,v_1,\,v_2$, respectively, and transfer tensors $\bB^{(1,2,3,4)},\bB^{(1,2)},\bB^{(3,4)}$, see Figure \ref{fig:dimtree}. In particular, 
\beq
\label{eq:htd_f_nested_0}
{\bf f} = \sum_{l_{12}=1}^{r_{12}}{\sum_{l_{34}=1}^{r_{34}}} \bB^{(1,2, 3, 4)}_{l_{12},l_{34}}\bU_{l_{12}}^{(1, 2)}\otimes \bU_{l_{34}}^{(3, 4)},
\eeq
with 
\beq
\label{eq: U12_htd}
\bU_{l_{12}}^{(1, 2)} = \sum_{l_{1}=1}^{r_{1}}{\sum_{l_{2}=1}^{r_{2}}} \bB^{(1,2)}_{l_{1},l_{2}, l_{12}}\bU_{l_{1}}^{(1)}\otimes \bU_{l_{2}}^{(2)},\quad l_{12} =1,\ldots,r_{12}, 
\eeq
and 
\beq
\label{eq: U34_htd}
\bU_{l_{34}}^{(3, 4)} = \sum_{l_{3}=1}^{r_{3}}{\sum_{l_{4}=1}^{r_{4}}} \bB^{(3, 4)}_{l_{3},l_{4}, l_{34}}\bU_{l_{3}}^{(3)}\otimes \bU_{l_{4}}^{(4)},\quad l_{34} =1,\ldots,r_{34}. 
\eeq
 Further, the electric field $E_1$ and $E_2$ are represented in the second order HT format. 
 
 Now we are ready to formulate a conservative HT truncation of VP solutions in a low rank tensor format. We denote pre-compressed solution ${\bf f}$ in the HT format as in \eqref{eq:htd_f_nested_0}. The idea of the conservative truncation is similar to the 1D1V case. We first introduce a weighted inner product space for ${\bf f}, {\bf g} \in \mathbb{R}^{N_{v_1} \times N_{v_2}}$ 
\beq
\label{eq: 2d2v_inner_w}
\langle {\bf f}, {\bf g}\rangle_{\bf w} = h_{v_1}h_{v_2} \sum_{j_1=1}^{N_{v_1}} \sum_{j_2 = 1}^{N_{v_2}} f_{j_1, j_2} g_{j_1, j_2}  w_{j_1, j_2}, 
\eeq
where $h_{v_1}$ and $h_{v_2}$ are the mesh size for the  $v_1$ and $v_2$ grids, respectively. Here
\beq
\label{eq: w_12}
{\bf w}\doteq {\bf w}^{(1)}\otimes {\bf w}^{(2)} \in \mathbb{R}^{N_{v_1} \times N_{v_2}}, \quad {\bf w}^{(1)}\in \mathbb{R}^{N_{v_1}}, {\bf w}^{(2)}\in \mathbb{R}^{N_{v_2}}, 
\eeq 
where ${\bf w}^{(1)}$, ${\bf w}^{(2)}$ are vectors as point values of the weight function $w(v)= \exp(-\frac{{v}^2}{2})$ at corresponding $v_1$ and $v_2$ grids. 
Then we seek a decomposition ${\bf f} = {\bf f}_{1} + {\bf f}_{2}$, where ${\bf f}_{1}$ comes from the rescaling and orthogonal projection of ${\bf f}$ with respect to the weighted inner product \eqref{eq: 2d2v_inner_w} onto the subspace 
 \beq
 \label{eq: 2d2v_subspa}
 \mathcal{N}\doteq \text{span}\{{\bf 1}_{v_1 \otimes v_2}, {\bf v}_1\otimes {\bf 1}_{v_2},  {\bf 1}_{v_1}\otimes{\bf v}_2, {\bf v}_1^2\otimes {\bf 1}_{v_2}+{\bf 1}_{v_1}\otimes{\bf v}_2^2\}
 \eeq
for the conservation of mass, momentum and kinetic energy density, see Proposition~\ref{prop:2d2v} below for details. For further discussion, we also introduce a standard inner product for ${\bf f}, {\bf g} \in \mathbb{R}^{N_{v_1} \times N_{v_2}}$ 
\beq
\label{eq: 2d2v_inner}
\langle {\bf f}, {\bf g}\rangle = h_{v_1}h_{v_2} \sum_{j_1=1}^{N_{v_1}} \sum_{j_2 = 1}^{N_{v_2}} f_{j_1, j_2} g_{j_1, j_2}. 
\eeq
  
\begin{prop}\label{prop:2d2v} 
Let ${\bf f}_1$ come from the scaling and orthogonal projection of rescaled ${\bf f}$ with respect to the weighted inner product \eqref{eq: 2d2v_inner_w} onto the subspace $\eqref{eq: 2d2v_subspa}$, in a similar spirit to the 1D1V case \eqref{eq: f_decom_d}.  Assume ${\bf f}$ is written in the low rank HT format \eqref{eq:htd_f_nested}. ${\bf f}_1$ can be represented in  low rank HT format (consistently with the subscript $1$ in the notations),
\begin{eqnarray}
\label{eq:htd_f1_nested}
{\bf f}_1 = \tilde{P} ({\bf f}) & \doteq& \sum_{l_{12}=1}^{4}{\sum_{l_{34}=1}^{4}} (\bB_1^{(1,2, 3, 4)})_{l_{12},l_{34}}(\bU_1^{(1, 2)})_{l_{12}}\otimes (\bU_1^{(3, 4)})_{l_{34}},\\
\label{eq:htd_f1_nested_b}
& \stackrel{=}{\eqref{eq: B1} }&  \sum_{k=1}^{4} {(\bU_1^{(1, 2)})_k}\otimes {(\bU_1^{(3, 4)})_k},
\end{eqnarray}
where we introduce the notation of $\tilde{P}$ as the scaled orthogonal projection. The specifications of the frame vectors and transfer tensors are outlined below.
\bit
\item For ${\bf f}_1$, the hierarchical ranks are $r_{12} = r_{34} = 4$, $r_3=r_4=3$. $r_1$ and  $r_2$ are the same as those for ${\bf f}$. 
\item
The transfer tensor, $\bB_1^{(1, 2, 3, 4)}$, of size $4\times 4$, is an identity matrix, 
\beq
\label{eq: B1}
\bB_1^{(1, 2, 3, 4)} = {\bf I}_{4\times 4}.
\eeq
$(\bU_1^{(3, 4)})_k$, $k =1,\ldots,4,$ is constructed from an orthonormal set of basis $\{\bV_1, \cdots \bV_4\}$ in the $v_1-v_2$ dimensions defined in \eqref{eq: V}. We have $r_3=r_4=3$ and 
\beq
\label{eq: U34_htd_f1}
(\bU_1^{(3, 4)})_{k} = \sum_{l_{3}=1}^{3}{\sum_{l_{4}=1}^{3}} (\bB_1^{(3, 4)})_{l_{3},l_{4}, k} ({\bf w}^{(1)}\star(\bU_1^{(3)})_{l_{3}})\otimes ({\bf w}^{(2)}\star(\bU_1^{(4)})_{l_{4}})
\eeq
with ${\bf w}^{(1)}$ and ${\bf w}^{(2)}$ the same as in \eqref{eq: w_12}. 
The frame vectors  for node $3$ are 
\begin{eqnarray}
(\bU_1^{(3)})_1 = \frac{1}{c_1} {\bf 1}_{v_1} , \quad
(\bU_1^{(3)})_2 = \frac{1}{c_2}{\bf v}_1 , \quad
(\bU_1^{(3)})_3 = \frac{1}{c_3}({\bf v}_1^2-c {\bf 1}_{v_1}),
\label{eq: U3U4}
\end{eqnarray}
where $c= \frac{\innerw{\mathbf{1}_{v_1},\bv_1^2}}{\innerw{\mathbf{1}_{v_1},\mathbf{1}_{v_1}}}$ is the orthogonalization constant of the basis, and $c_l$, $l=1, 2, 3$ are normalization constants for the corresponding basis of ${\bf 1}_{v_1}$, ${\bf v}_1$ and ${\bf v}_1^2-c  {\bf 1}_{v_1}$. 
We have the same frame vectors for the node $4$ but for $v_2$, assuming that the weight function and domain in $v_2$ is the same as $v_1$, 
\begin{eqnarray}
(\bU_1^{(4)})_1 = \frac{1}{c_1} {\bf 1}_{v_2} , \quad
(\bU_1^{(4)})_2 = \frac{1}{c_2}{\bf v}_2 , \quad
(\bU_1^{(4)})_3 = \frac{1}{c_3}({\bf v}_2^2-c {\bf 1}_{v_2})
\label{eq: U3U4_b}
\end{eqnarray} 
 The transfer tensor $\bB_1^{(3, 4)}$ is a tensor of size $3\times 3 \times4$. It has zero elements, except the following specification for $(\bB_1^{(3, 4)})_{l_{3},l_{4}, l_{34}}$
\begin{equation}
\label{eq:B34f}
(\bB_1^{(3,4)})_{1,1,1} =(\bB_1^{(3,4)})_{2,1,2}=(\bB_1^{(3,4)})_{1,2,3}= 1, \quad
(\bB_1^{(3,4)})_{3,1,4} = (\bB_1^{(3,4)})_{1,3,4} = \frac1{\sqrt{2}}.
\end{equation} 
\item 
$(\bU_1^{(1, 2)})_{k}$, $k=1, \cdots 4$, is in the following form, 
\beq
\label{eq: U12_htd_b}
(\bU_1^{(1, 2)})_{k} = \sum_{l_{1}=1}^{r_{1}}{\sum_{l_{2}=1}^{r_{2}}} (\bB_1^{(1,2)})_{l_{1},l_{2}, k}(\bU_1^{(1)})_{l_{1}}\otimes (\bU_1^{(2)})_{l_{2}},
\eeq
with the same frame vectors as ${\bf f}$ on the leaf nodes $1$ and $2$, meaning that $\bU_1^{(1)}$ and $\bU_1^{(2)}$ are the same as $\bU^{(1)}$ and $\bU^{(2)}$, respectively.  $\bB_1^{(1, 2)}$, of size $r_1 \times r_2 \times 4$, has its elements as, 
\beq
\label{eq: B1_12}
(\bB_1^{(1, 2)})_{l_1, l_2, k}=\sum_{l_{12}=1}^{r_{12}} \sum_{l_{34}=1}^{r_{34}}\bB^{(1,2)}_{l_1,l_2,l_{12}}\bS_{k,l_{34}}\bB^{(1,2,3,4)}_{l_{12},l_{34}}, \quad k = 1, \cdots 4, 
\eeq
with $r_{12}$, $r_{34}$ and transfer tensors $\bB^{(1,2)}$ and $\bB^{(1,2,3,4)}$ from the HT representation of ${\bf f}$, and 
\beq
\bS_{k,l_{34}} = \langle \bU_{l_{34}}^{(3, 4)} , \bV_k\rangle, \quad k=1, \cdots 4, 
\eeq 
where $\bU_{l_{34}}^{(3, 4)}$, $l_{34}=1,\ldots,r_{34}$, is the frame tensor for node (3, 4) of ${\bf f}$. The inner product in the sense of \eqref{eq: 2d2v_inner} can be evaluated in a dimension-by-dimension manner. 
\eit
Finally ${\bf f}_1$ has the same discrete macroscopic charge, current and kinetic energy density as with $\bf f$, which are denoted as $\boldsymbol\rho$, ${\bf J}_{1}$, ${\bf J}_{2}$, and $\boldsymbol\kappa$.  
The discrete macroscopic charge, current and kinetic energy density of ${\bf f}_1$ are 
\beq
\label{eq: rho1}
{\boldsymbol\rho} = (c_1)^2 (\bU_1^{(1, 2)})_1  = (c_1)^2 \sum_{l_1, l_2} (\bB_1^{(1, 2)})_{l_1, l_2, 1} \bU^{(1)}_{l_1}\otimes \bU^{(2)}_{l_2}, 
\eeq
\beq
\label{eq: J1_1}
{\bJ_1} = c_1 c_2 (\bU_1^{(1, 2)})_2  = c_1 c_2 \sum_{l_1, l_2} (\bB_1^{(1, 2)})_{l_1, l_2, 2} \bU^{(1)}_{l_1}\otimes \bU^{(2)}_{l_2},
\eeq
\beq
\label{eq: J1_2}
{\bJ_2} = c_1 c_2 (\bU_1^{(1, 2)})_3  = c_1 c_2 \sum_{l_1, l_2} (\bB_1^{(1, 2)})_{l_1, l_2, 3} \bU^{(1)}_{l_1}\otimes \bU^{(2)}_{l_2},
\eeq
\beq
\label{eq: kappa1}
{\boldsymbol\kappa} = \frac{1}{\sqrt{2}}c_1 c_3 (\bU_1^{(1, 2)})_4 +  c \boldsymbol\rho   =\frac{1}{\sqrt{2}} c_1 c_3 \sum_{l_1, l_2} (\bB_1^{(1, 2)})_{l_1, l_2, 4} \bU^{(1)}_{l_1}\otimes \bU^{(2)}_{l_2}+  c \boldsymbol\rho.
\eeq
They are the same as those of ${\bf f}$, which can be obtained by the HT tensor contraction 
\begin{align}
\left(\begin{array}{l}
{\boldsymbol\rho}\\
{\bf J}_1\\
{\bf J}_2\\
{\boldsymbol \kappa} 
\end{array}
\right )
&= \sum_{l_{12}} \sum_{l_{34}} \bB^{(1,2, 3, 4)}_{l_{12},l_{34}}
 \left 
 \langle \bU_{l_{34}}^{(3, 4)}, 
 \left(\begin{array}{l}
{\bf 1}_{v_1 \otimes v_2} \\
{\bf v}_1\otimes {\bf 1}_{v_2}\\
{\bf 1}_{v_1}\otimes{\bf v}_2\\
\frac12{\bf v}_1^2\otimes {\bf 1}_{v_2}+\frac12{\bf 1}_{v_1}\otimes{\bf v}_2^2
\end{array}
\right )
\right \rangle
\bU_{l_{12}}^{(1, 2)}, \nonumber\\
&= \sum_{l_{12}} \sum_{l_{34}} \bB^{(1,2, 3, 4)}_{l_{12},l_{34}}
 \left(\begin{array}{l}
S_{1, l_{34}} (c_1)^{2}\\
S_{2, l_{34}} c_1 c_2\\
S_{3, l_{34}} c_1 c_2\\
S_{4, l_{34}} c_1 c_3 \frac{1}{\sqrt{2}} + S_{1, l_{34}} (c_1)^2 c 
\end{array}
\right )
\bU_{l_{12}}^{(1, 2)}.
\label{eq:rho_j_kappa_2d2v}
\end{align}
\end{prop}
\begin{proof} 
We first scale ${\bf f}$ by ${\bf w}$ in \eqref{eq: w_12}
\beq
\label{eq:htd_f_nested}
\tilde{\bf f} = \sum_{l_{12}=1}^{r_{12}}{\sum_{l_{34}=1}^{r_{34}}} \bB^{(1,2, 3, 4)}_{l_{12},l_{34}}\bU_{l_{12}}^{(1, 2)}\otimes \frac{1}{\bf w} \star \bU_{l_{34}}^{(3, 4)}.
\eeq 
Let $ \tilde{\bU}_{l_{34}}^{(3, 4)} \doteq \frac{1}{\bf w} \star \bU_{l_{34}}^{(3, 4)}$,  where $\star$ is the element-wise multiplication in the corresponding dimensions. 
Then we perform the orthogonal projection of $\tilde{\bf f}$ onto the subspace of \eqref{eq: 2d2v_subspa} to obtain $\tilde{\bf f}_1$. Finally we rescale back to ${\bf f}_1 = {\bf w} \star \tilde{\bf f}_1$.

To perform the orthogonal projection, the orthogonalization of basis in \eqref{eq: 2d2v_subspa} gives the frame vectors for the node 3 and 4 as specified in \eqref{eq: U3U4} and the transfer tensor $\bB_1^{(3, 4)}$ as specified in \eqref{eq:B34f}. This gives a set of  orthonormal basis of $(\bU_1^{(3, 4)})_l$, $l=1, 2, 3, 4$. We denote them as $\bV_l$, $l=1, 2, 3, 4$, 
\begin{align}
{\bf V}_1 &=(\bU_1^{(3)})_1 \otimes (\bU_1^{(4)})_1, \quad {\bf V}_2 = (\bU_1^{(3)})_2 \otimes(\bU_1^{(4)})_1\nonumber\\
{\bf V}_3 &=(\bU_1^{(3)})_1 \otimes (\bU_1^{(4)})_2, \quad {\bf V}_4 = \frac1{\sqrt{2}}(\bU_1^{(3)})_3 \otimes(\bU_1^{(4)})_1 +  \frac1{\sqrt{2}} (\bU_1^{(3)})_1 \otimes(\bU_1^{(4)})_3.
\label{eq: V}
\end{align}
To perform the orthogonal projection to \eqref{eq: 2d2v_subspa}, only the bases of node $(3,4)$ are affected. In particular, let ${\bf P}^{(3,4)}$ be the projection operator at the node $(3, 4)$, then 
\begin{align}
{\bf P}^{(3,4)} \tilde{\bf f} &={\bf P}^{(3,4)}  \sum_{l_{12}=1}^{r_{12}}{\sum_{l_{34}=1}^{r_{34}}} \bB^{(1,2, 3, 4)}_{l_{12},l_{34}}\bU_{l_{12}}^{(1, 2)}\otimes \frac{1}{\bf w} \star \bU_{l_{34}}^{(3, 4)} \nonumber\\
& = \sum_{l_{12}=1}^{r_{12}}{\sum_{l_{34}=1}^{r_{34}}} \bB^{(1,2, 3, 4)}_{l_{12},l_{34}}\bU_{l_{12}}^{(1, 2)}\otimes 
\left({\bf P}^{(3,4)} (\frac{1}{\bf w} \star \bU_{l_{34}}^{(3, 4)})\right), 
\label{eq: P34}
\end{align}
where 
\beq
\label{eq: SS}
{\bf P}^{(3,4)}(\frac{1}{\bf w} \star \bU_{l_{34}}^{(3, 4)})=
{\bf P}^{(3,4)}\tilde{\bU}^{(3,4)}_{l_{34}} = \sum_{k=1}^4  {\langle \tilde{\bU}^{(3,4)}_{l_{34}}, \bV_k \rangle_{\bw}}\bV_k \doteq 
 \sum_{k=1}^4  \bS_{k,l_{34}} \bV_k,
\eeq
with
\beq
\label{eq: S}
\bS_{k,l_{34}}  = {\langle \tilde{\bU}^{(3,4)}_{l_{34}}, \bV_k \rangle_{\bw}} =  \langle \frac{1}{\bf w} \star \bU_{l_{34}}^{(3, 4)}, \bV_k \rangle_{\bw} = \langle \bU_{l_{34}}^{(3, 4)} , \bV_k\rangle.
\eeq
Plug \eqref{eq: SS} into \eqref{eq: P34}, 
\begin{align}
{\bf P}^{(3,4)} \tilde{\bf f} 
& = \sum_{l_{12}=1}^{r_{12}}{\sum_{l_{34}=1}^{r_{34}}} \bB^{(1,2, 3, 4)}_{l_{12},l_{34}}\bU_{l_{12}}^{(1, 2)}\otimes 
(\sum_{k=1}^4  \bS_{k,l_{34}} \bV_k), \nonumber\\
&= \sum_{l_{12}=1}^{r_{12}} {\sum_{l_{34}=1}^{r_{34}}} \bB^{(1,2, 3, 4)}_{l_{12},l_{34}} \left(\sum_{l_1=1}^{r_{1}} \sum_{l_2=1}^{r_{2}} (\bU_{l_1}^{(1)}\otimes\bU_{l_2}^{(2)})\bB^{(1,2)}_{l_1,l_2,l_{12}}\right) \otimes 
(\sum_{k=1}^4  \bS_{k,l_{34}} \bV_k),  \nonumber\\
& =\sum_{k=1}^4  \left(\sum_{l_1=1}^{r_{1}} \sum_{l_2=1}^{r_{2}} \underbrace{\left(\sum_{l_{12}=1}^{r_{12}}{\sum_{l_{34}=1}^{r_{34}}} \bB^{(1,2, 3, 4)}_{l_{12},l_{34}} \bB^{(1,2)}_{l_1,l_2,l_{12}}  \bS_{k,l_{34}} \right)}_{\doteq (\bB_1^{(1,2)})_{l_1,l_2,k}}(\bU_{l_1}^{(1)}\otimes\bU_{l_2}^{(2)})\right) \otimes \bV_k,\\
&=\sum_{k=1}^4  \left(\sum_{l_1=1}^{r_{1}} \sum_{l_2=1}^{r_{2}}  (\bB_1^{(1,2)})_{l_1,l_2,k} (\bU_{l_1}^{(1)}\otimes\bU_{l_2}^{(2)})\right) \otimes \bV_k, 
\label{eq: P34_2}
\end{align}
where we let $(\bB_1^{(1,2)})_{l_1,l_2,k}$ as specified in \eqref{eq: B1_12}, and 
\beq
\label{eq: U12k}
(\bU_1^{(1, 2)})_k \doteq
\left(\sum_{l_1=1}^{r_{1}} \sum_{l_2=1}^{r_{2}}  (\bB_1^{(1,2)})_{l_1,l_2,k} (\bU_{l_1}^{(1)}\otimes\bU_{l_2}^{(2)})\right)
\eeq
as in \eqref{eq: U12_htd_b}.

The macroscopic charge, momentum and energy density of ${\bf f}_1$, \eqref{eq: rho1}-\eqref{eq: kappa1} can be derived, from the form of ${\bf f}_1$ in \eqref{eq:htd_f1_nested_b}, the form of $\bU_1^{(1, 2)}$ in \eqref{eq: U12k}, and 
\[
{\bf 1}_{v_1 \otimes v_2} = (c_1)^2 {\bf V}_1, \quad {\bf v}_1\otimes {\bf 1}_{v_2} = c_1 c_2 {\bf V}_2, \quad
{\bf 1}_{v_1}\otimes  {\bf v}_2 = c_1 c_2 {\bf V}_3, 
\]
\[
\frac12{\bf v}_1^2\otimes {\bf 1}_{v_2}+\frac12{\bf 1}_{v_1}\otimes{\bf v}_2^2 =
\frac{1}{\sqrt{2}}c_1 c_3 {\bf V}_4 + c (c_1)^2 {\bf V}_1, 
\]
which is due to \eqref{eq: V}.
The agreement of macroscopic charge, momentum and energy density of ${\bf f}$ and ${\bf f}_1$ is a direct consequence of manipulation of equalities in \eqref{eq: U12_htd_b}, \eqref{eq: B1_12}, \eqref{eq:rho_j_kappa_2d2v} and orthonormal property of the basis in $\bU_1^{(3, 4)}$. 

 \begin{figure}
\centering
\begin{tikzpicture}[
  every node/.style = {shape=rectangle, rounded corners,
    draw, 
    top color=white, bottom color=blue!20},
    level/.style={sibling distance=55mm/#1}
    ]
  \node {$\bB_1^{(1,2,3,4)}$ in \eqref{eq: B1}}   
  child { node {$\bB_1^{(1,2)}$ in \eqref{eq: B1_12}}
    	child{ node{$\bU^{(1)}$}}
	child{ node{$\bU^{(2)}$}}
          }
    child { node {$\bB_1^{(3,4)}$ in \eqref{eq:B34f}}
    	child{{node{$\bU_1^{(3)}$ in \eqref{eq: U3U4}}}}
	child{{node{$\bU_1^{(4)}$ in \eqref{eq: U3U4_b}}}} 
	};
\end{tikzpicture}
 \caption{The data layout of ${\bf f}_1$. Here $\bU^{(1)}$ and $\bU^{(2)}$ are the same as those vector frames for ${\bf f}$. \label{fig:f1}}
\end{figure}
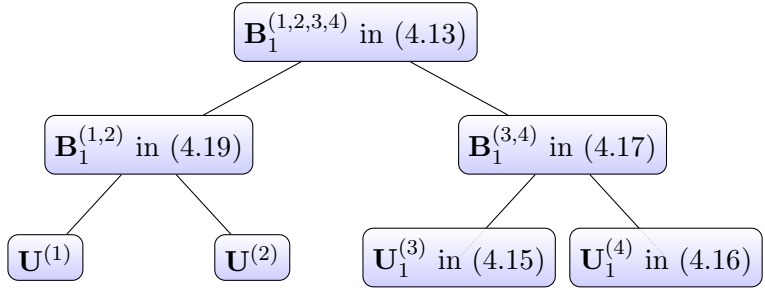

\end{proof}

Once the orthogonal projection is performed,  ${\bf f}_{1}$ stays untouched for the truncation. As with the 1D1V case, we will perform HOSVD to truncate the remainder ${\bf f}_{2} \doteq {\bf f} - {\bf f}_{1}$. If a standard HT truncation is directly applied to ${\bf f}_{2}$, unlike the 1D1V case the conservation property cannot be guaranteed. Below, we elaborate and investigate such an issue, and further develop a conservative projection procedure after the HT truncation to ensure charge, current and kinetic energy density conservation.

We start with a brief description of a naive weighted hierarchical HOSVD truncation procedure, as a direct analog of the 1D1V case. 
First, we scale ${\bf f}_2$ according to the weights and define $\tilde{\bf f}_2=\frac{1}{\sqrt{\bf w}} \star {\bf f}_2$. Note that the rescaling is computed dimension-by-dimension. The standard hierarchical HOSVD  root-to-leaf truncation with threshold $\varepsilon$ is applied to $\tilde{\bf f}_2$ and $\mT_\varepsilon(\tilde{\bf f}_2)$ is obtained. Finally, ${\bf f}_2$ is defined as 
\beq
\label{eq: f2_2d2v}
{\bf f}_2={\sqrt{\bf w}} \star \mT_\varepsilon(\tilde{\bf f}_2). 
\eeq
The issue with such procedure is the loss of macroscopic conservation in the root-to-leaf truncation. 
In particular, the HT truncation $\mT_\varepsilon(\tilde{\bf f}_2)$ can be represented as, see \cite{grasedyck2010hierarchical},
\begin{equation}
\mT_\varepsilon(\tilde{\bf f}_2) = (\pi_1\otimes \pi_2 \otimes \pi_3 \otimes \pi_4)(\pi_{12}\otimes\pi_{34})\, {\bf f}_2,
\end{equation}
where $\pi_\alpha$ denotes the orthogonal projection on the subspace spanned by $\tilde{r}_{\alpha}$ leading left singular  vectors $\{\tilde{\bU}^{(\alpha)}_{i}\}_{i=1}^{\tilde{r}_\alpha}$ of the matricization $\mM^{(\alpha)}({\bf f}_2)$. However, the bases at node $(3,4)$  are no longer orthogonal to $\{\bV_1,\,\bV_2,\,\bV_3,\,\bV_4\}$ in \eqref{eq: V} due to the truncation $\pi_3 $ and $ \pi_4$ at leaf nodes. Hence ${\bf f}_2$ as in \eqref{eq: f2_2d2v} may have nonzero charge, current and kinetic energy density, breaking the conservation property. To fix the issue, we propose to apply the operator {$({\bf I} - \tilde{P})$}, to ${\sqrt{\bf w}} \star \mT_\varepsilon(\tilde{\bf f}_2)$ 
to ensure zero charge, current and kinetic energy density of truncated ${\bf f}_2$. Here $\tilde{P}$ is the same as that specified in \eqref{eq:htd_f1_nested} in Proposition~\ref{prop:2d2v}. We introduce the following notation,  
\beq
\widetilde{\mT_\varepsilon}({\bf f}_2) \doteq ({\bf I} - \tilde{P}) \left({\sqrt{\bf w}} \star \mT_\varepsilon(\tilde{\bf f}_2)\right).
\label{eq: widetilde_T}
\eeq
Finally, the conservative HT truncation of ${\bf f}$ is done as follows
\beq
T_c({\bf f})\doteq
{\bf f}_1 + \widetilde{\mT_\varepsilon}({\bf f}_2). 
\label{eq: Tc2d2v}
\eeq
The following proposition is a straightforward consequence from the orthogonal projection.
\begin{prop} 
$\widetilde{\mT_\varepsilon}({\bf f}_2)$
 has zero charge density, zero current density, and zero kinetic energy density.  Hence, $T_c({\bf f})$ preserves the charge, current, and kinetic energy densities ($\boldsymbol \rho$, ${\bf J}_1$, ${\bf J}_2$, $\boldsymbol\kappa$) of the original ${\bf f}$.
\end{prop}
\begin{proof}
The zero charge, current and kinetic energy density is a direct consequence of the ${\bf I} - \tilde{P}$ projection operator. From this fact, together with from Proposition~\ref{prop:2d2v}, $T_c({\bf f})$ preserves the charge, current, and kinetic energy densities ($\boldsymbol \rho$, ${\bf J}_1$, ${\bf J}_2$, $\boldsymbol\kappa$) of the original ${\bf f}$. 
\end{proof}

\begin{prop} (Local mass and momentum conservation for the 2D2V VP system.) If the discrete differential operators $D_x,\,D_v$ employed are conservative, i.e., can be written in a flux difference form, and linear, i.e., can preserve linear relations, then the proposed low rank method with an SSP multi-step time integrator preserves the mass and momentum locally; that is the schemes for $\boldsymbol \rho$, ${\bf J}_1$, ${\bf J}_2$, from integrating the scheme on ${\bf f}^{n}$, are  consistent and conservative discretization of the macroscopic moment equations \eqref{eq:mass}-\eqref{eq:mom}. 
\end{prop}
\begin{proof}
The proof is based on a conservative and linear discretization of the discrete differential operators $D_x,\,D_v$ employed and from the conservative truncation procedure as proposed above. The details are similar to that of Proposition \ref{prop: cons} and hence omitted for brevity.
\end{proof}

\begin{rem} To ensure the local conservation, ${\bf f}_1$ should not be further compressed. In the numerical simulation, we still truncate  ${\bf f}_1$ at nodes (1, 2) with threshold $10^{-15}$ to remove the redundancy from the add basis procedure. Thus the local conservation property is preserved on the same scale of machine precision, i.e. $10^{-15}$. 
\end{rem}

We summarized the conservative truncation procedure as Algorithm \ref{alg: low_rank2d} below for the 2D2V VP solution. 

\bigskip
\begin{algorithm}[H]
\label{alg: low_rank2d}
  \caption{The conservative truncation procedure for the 2D2V VP solution.}
  \begin{itemize}
\item Input: the pre-compressed low rank solution ${\bf f}$ in the HT format with dimension tree given in Figure \ref{fig:dimtree} (a) and the associated data including the frame tensors $\bU^{(1)}$, $\bU^{(2)}$, $\bU^{(3)}$, and $\bU^{(4)}$ at nodes $(1)$, $(2)$, $(3)$, and $(4)$, respectively, and transfer tensors $\bB^{(1,2)}$, $\bB^{(3,4)}$, and $\bB^{(1,2,3,4)}$ at nodes $(1,2)$, $(3,4)$, and $(1,2,3,4)$, respectively. 
\item Output: the compressed low rank solution $T_c({\bf f})$ in the HT format with the same charge, current, and kinetic energy density functions as ${\bf f}$.
\end{itemize}

  \begin{enumerate}
  \item Compute the rescaled orthogonal projection to obtain ${\bf f}_1 = \tilde{P}({\bf f})$, in the HT format with data layout in Figure \ref{fig:f1}. The transfer tensor $\bB_1^{(1, 2, 3, 4)}$ is the identity matrix of size $4\times 4$, $(\bB_1^{(1, 2)})_{l_{12}}$ is a matrix of size $r_1 \times r_2 \times 4$ from \eqref{eq: B1_12}, and $\bB_1^{(3,4)}$ is a matrix of size $3 \times 3 \times 4$ from \eqref{eq:B34f}, with the frame tensors $\bU_1^{(1)} = \bU^{(1)}$, $\bU_1^{(2)}=\bU^{(2)}$, $\bU_1^{(3)}$ and $\bU_1^{(3)}$ from \eqref{eq: U3U4} and \eqref{eq: U3U4_b} respectively. At the end of this step, we truncate node (1, 2) of ${\bf f}_1$ with threshold $10^{-15}$ to remove redundant basis in $\bU_1^{(1, 2)}$. 
    \item Perform the HOSVD truncation, together with an orthogonal projection operator, to ${\bf f}_2 \doteq {\bf f}-{\bf f}_1$ to ensure zero charge, current and energy densities: %
      \begin{enumerate}
      \item 
    Compute ${\bf f}_2 ={\bf f} - {\bf f}_1$, and scale it to obtain $\tilde{\bf f}_2 = \frac{1}{\sqrt{\bf w}}\star{\bf f}_2$.
    \item 
    Apply the standard HOSVD truncation to $\tilde{\bf f}_2$, and apply rescaling to obtain ${\sqrt{\bf w}}\star\mT_\varepsilon(\tilde{\bf f}_2)$.
    \item 
    Apply ${\bf I} - \tilde{P}$ to ${\sqrt{\bf w}}\star\mT_\varepsilon(\tilde{\bf f}_2)$ to obtain $\widetilde{\mT_\varepsilon}({\bf f}_2)$, i.e. \eqref{eq: widetilde_T}, with the same $\tilde{P}$ operator as in the previous step. 
    \end{enumerate}    
  \item Update the compressed low rank solution $T_c({\bf f})={\bf f}_1 + \widetilde{\mT_\varepsilon}({\bf f}_2)$ from \eqref{eq: Tc2d2v}.
\end{enumerate}
  \end{algorithm}

%

%% file: numerical.tex
\section{Numerical results}
\label{sec:numerical}
In this section we present a collection of numerical examples to demonstrate the efficacy of the proposed conservative low rank tensor method for simulating the VP system. In the simulations, fifth order upwind finite difference methods are employed for spatial discretization, together with a second order SSP multi-step method denoted by SSPML2 for temporal discretization. The numerical solutions of high dimensions are represented in the HT format \cite{hackbusch2012tensor}. We also compare the proposed conservative low rank methods against  the nonconservative version in terms of efficiency and ability to conserve the physical invariants.  Unless otherwise noted, we let the weight function $w(\bv)=\exp(-\frac{|\bv|^2}{2})$.

\subsection{1D1V Vlasov-Poisson system}

\begin{exa}\label{ex:weak1d}(Weak Laudau damping.) 
We consider the weak Landau damping test with initial condition 
\begin{equation}
\label{eq:landau1d}
f(x,v,t=0) =\frac{1}{\sqrt{2 \pi}} \left(1+\alpha  \cos \left(k x\right)\right)\exp\left(-\frac{v^2}{2}\right),
\end{equation}
where $\alpha=0.01$ and $k=0.5$. \end{exa} 
The computational domain is set to be $[0,L_x]\times[-L_v,L_v]$ with $L_x=2\pi/k$ and $L_v=6$. We set $\varepsilon=10^{-5}$ for truncation. In Figure \ref{fig:weak1d_elec}, we report the time histories of the electric energy and numerical ranks of the solutions computed by the conservative and non-conservative methods for comparision. It is observed that both methods are able to predict the correct damping rate of the electric energy, and meanwhile, the numerical ranks of the conservative method are bigger than those of the non-conservative counterpart. In Figure \ref{fig:weak1d_invar}, the time histories of the relative deviation of total mass, momentum and energy are plotted. The conservative method is found to be able to conserve the total mass and momentum up to machine precision regardless of the mesh size used. As mentioned above, the proposed conservative method cannot conserve the total energy, as the time integrator employed is not energy conserving. With the mesh refinement, the conservarion error of total energy decreases.
Meanwhile, the non-conservative method can conserve the total mass and energy up to the magnitude of truncation threshold $\varepsilon$, but the total momentum is conserved to the machine precision, which is attributed to the symmetry of the solution in the velocity direction.  
 Noteworthy, though both methods cannot conserve the total energy, the conservative one does a better job in energy conservation compared to the non-conservative counterpart, and it is because that the kinetic energy is preserved in the truncation.

\begin{figure}[h!]
	\centering
	\subfigure[]{\includegraphics[height=50mm]{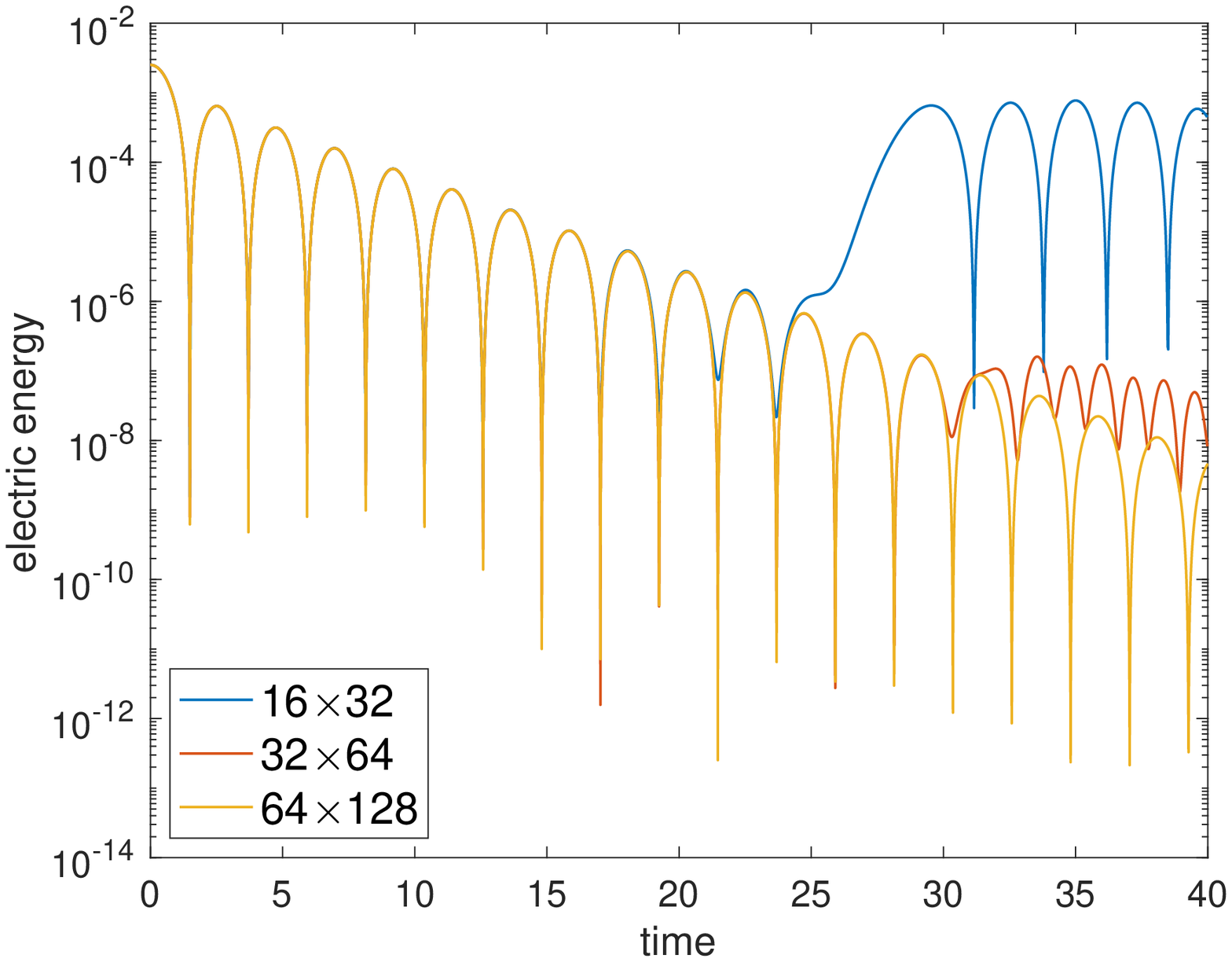}}
		\subfigure[]{\includegraphics[height=50mm]{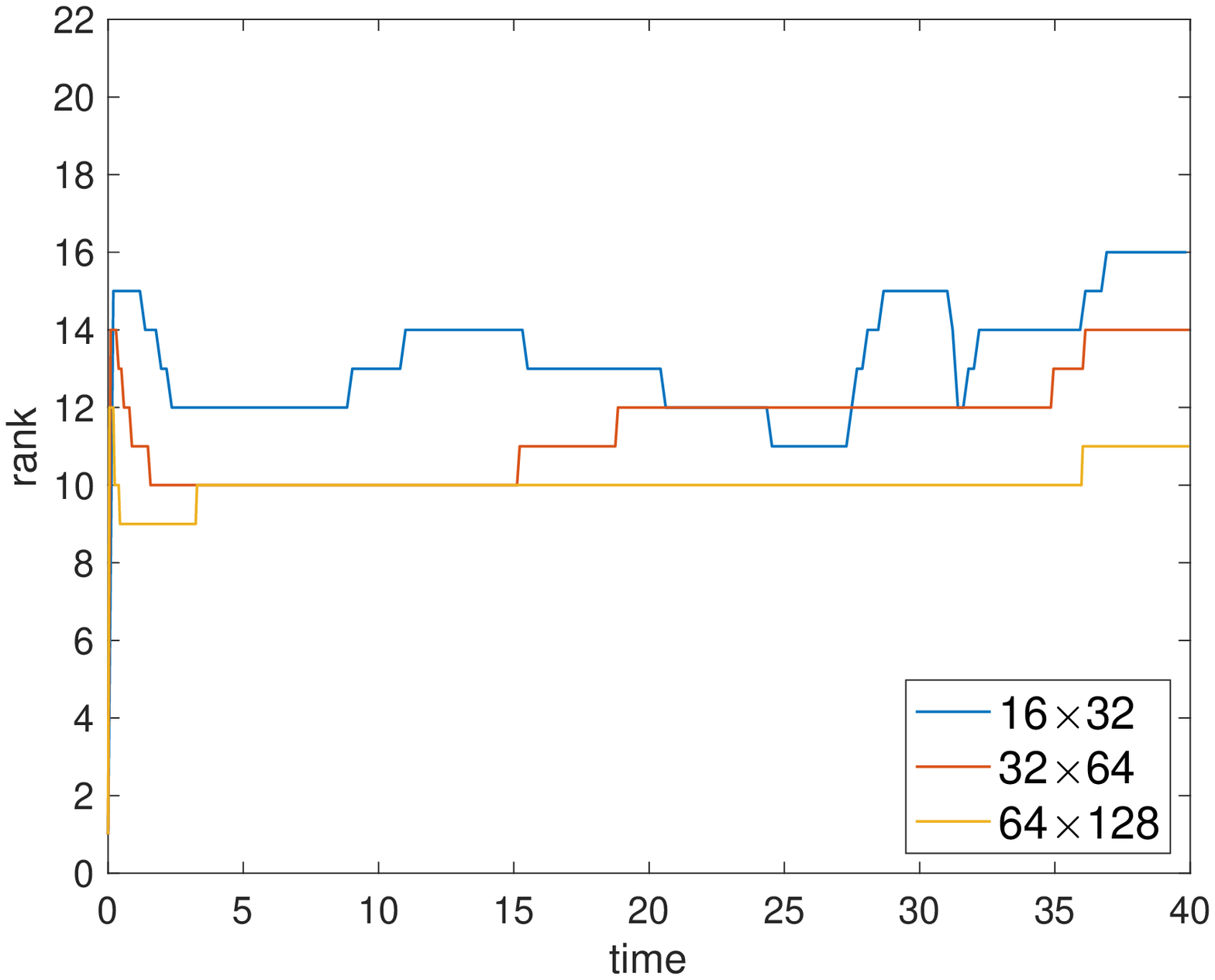}}
			\subfigure[]{\includegraphics[height=50mm]{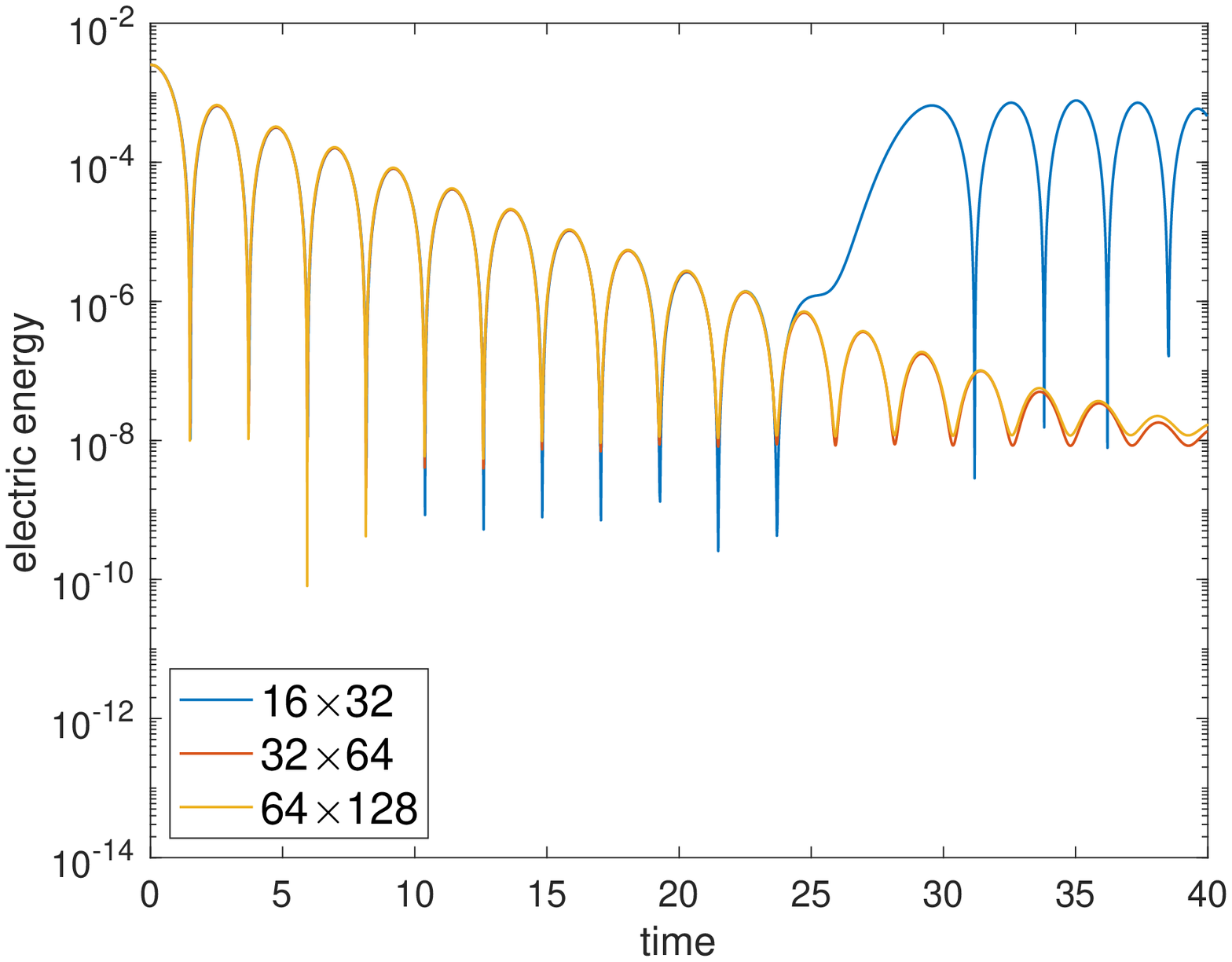}}
		\subfigure[]{\includegraphics[height=50mm]{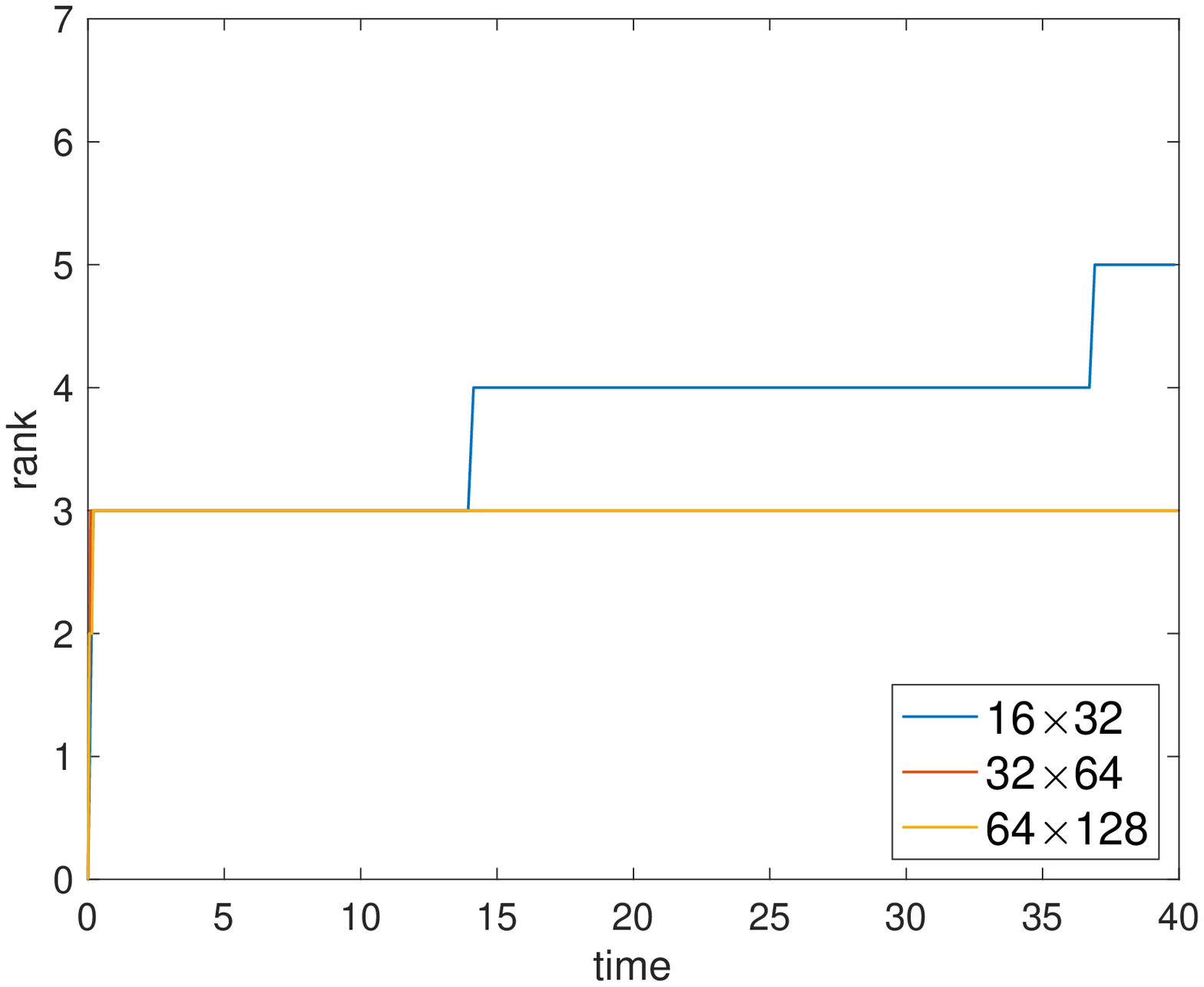}}
	\caption{Example \ref{ex:weak1d}.  The time evolution of the electric energy (a, c) and the rank of the numerical solutions (b, d). Conservative method (a, b) and non-conservative method (c, d). $\varepsilon=10^{-5}$.}
	\label{fig:weak1d_elec}
\end{figure}

\begin{figure}[h!]
	\centering
	\subfigure[]{\includegraphics[height=40mm]{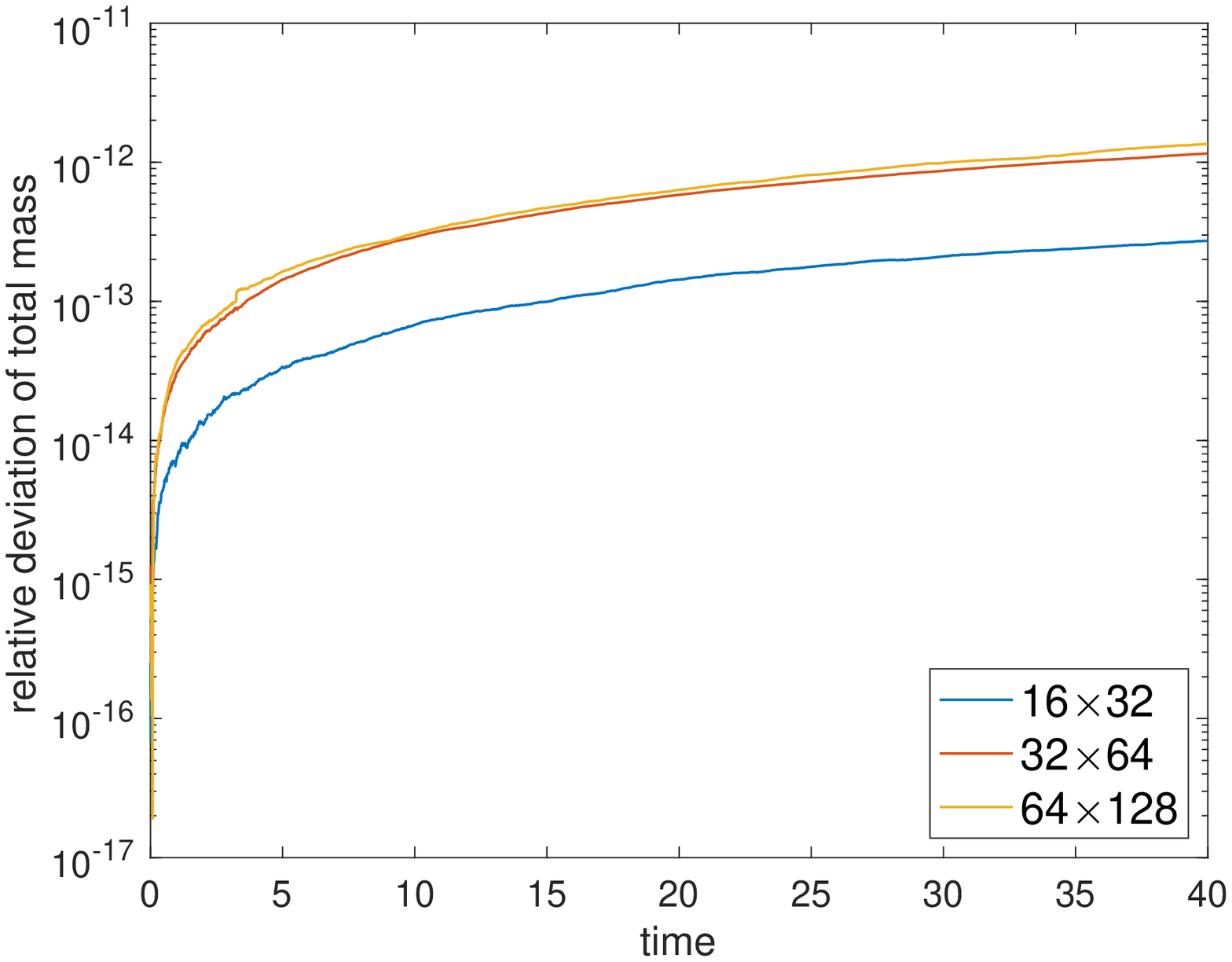}}
		\subfigure[]{\includegraphics[height=40mm]{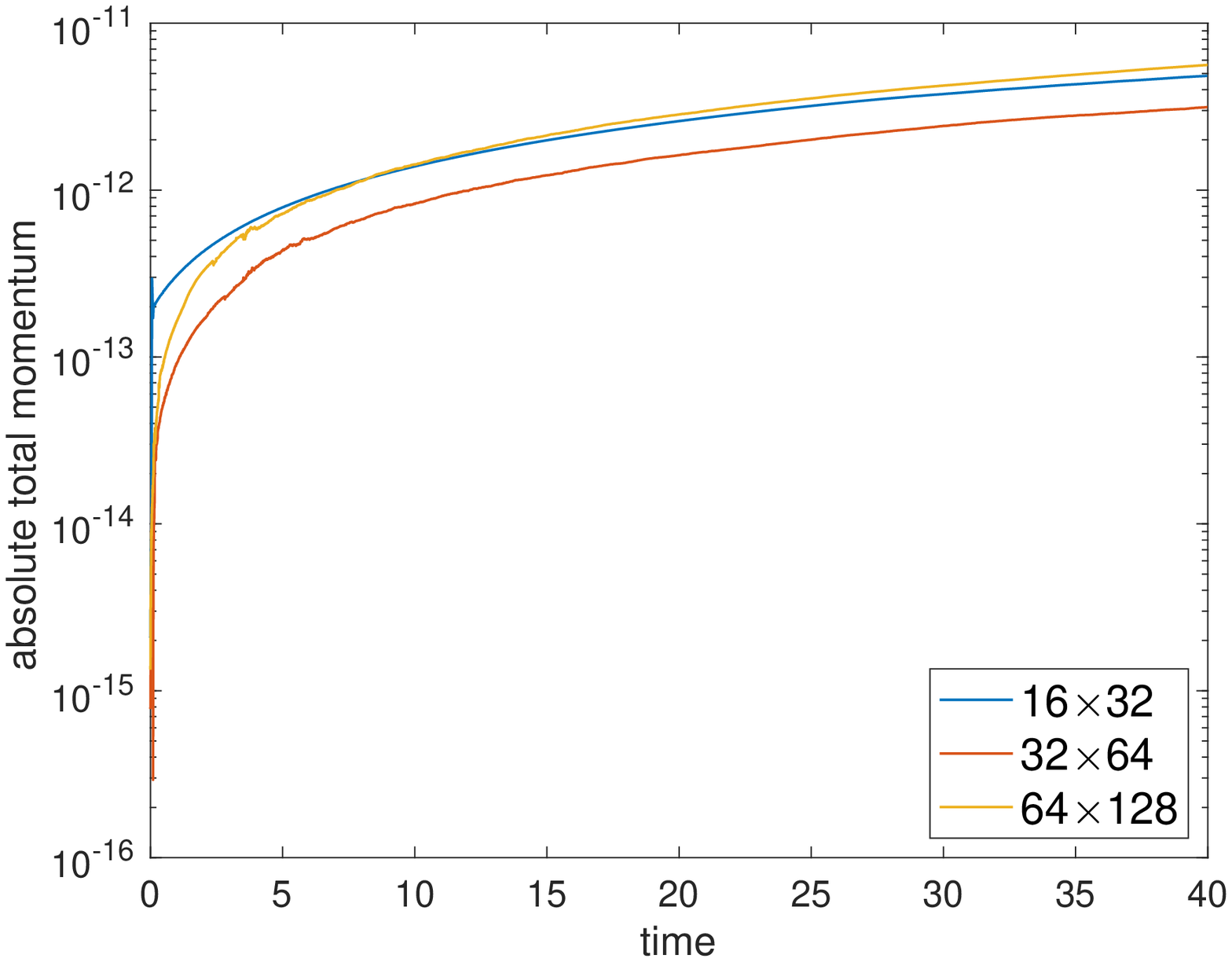}}
		\subfigure[]{\includegraphics[height=40mm]{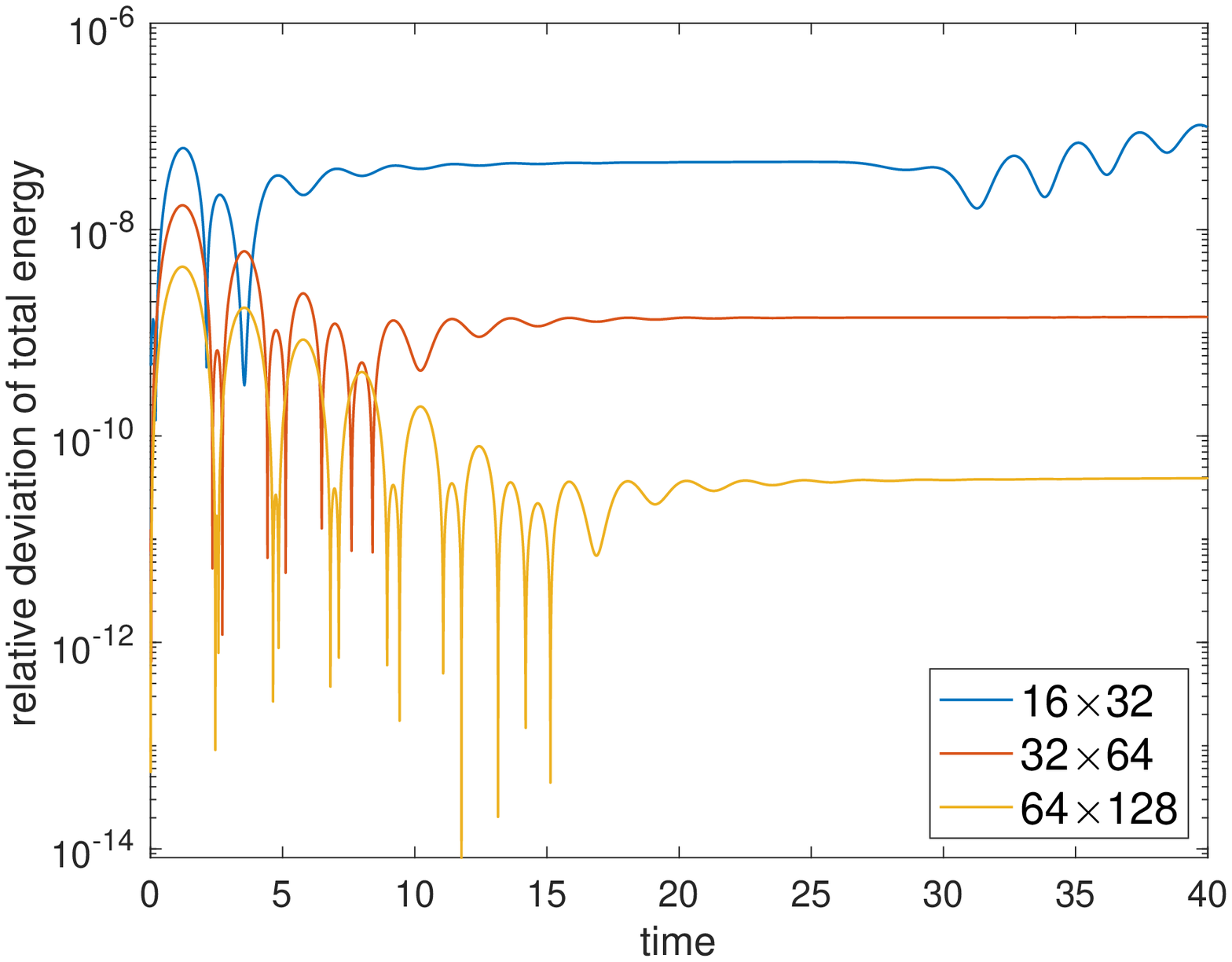}}
			\subfigure[]{\includegraphics[height=40mm]{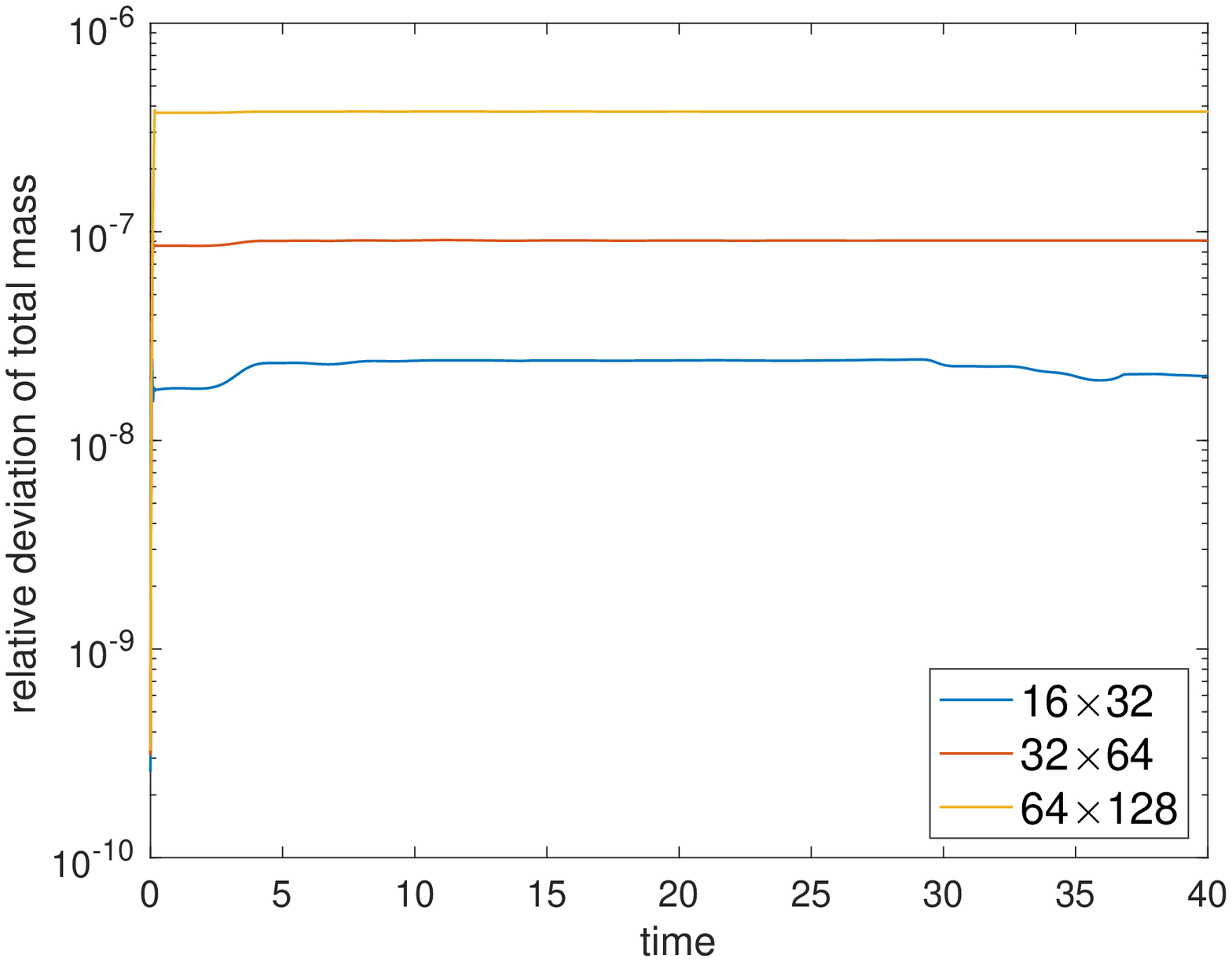}}
		\subfigure[]{\includegraphics[height=40mm]{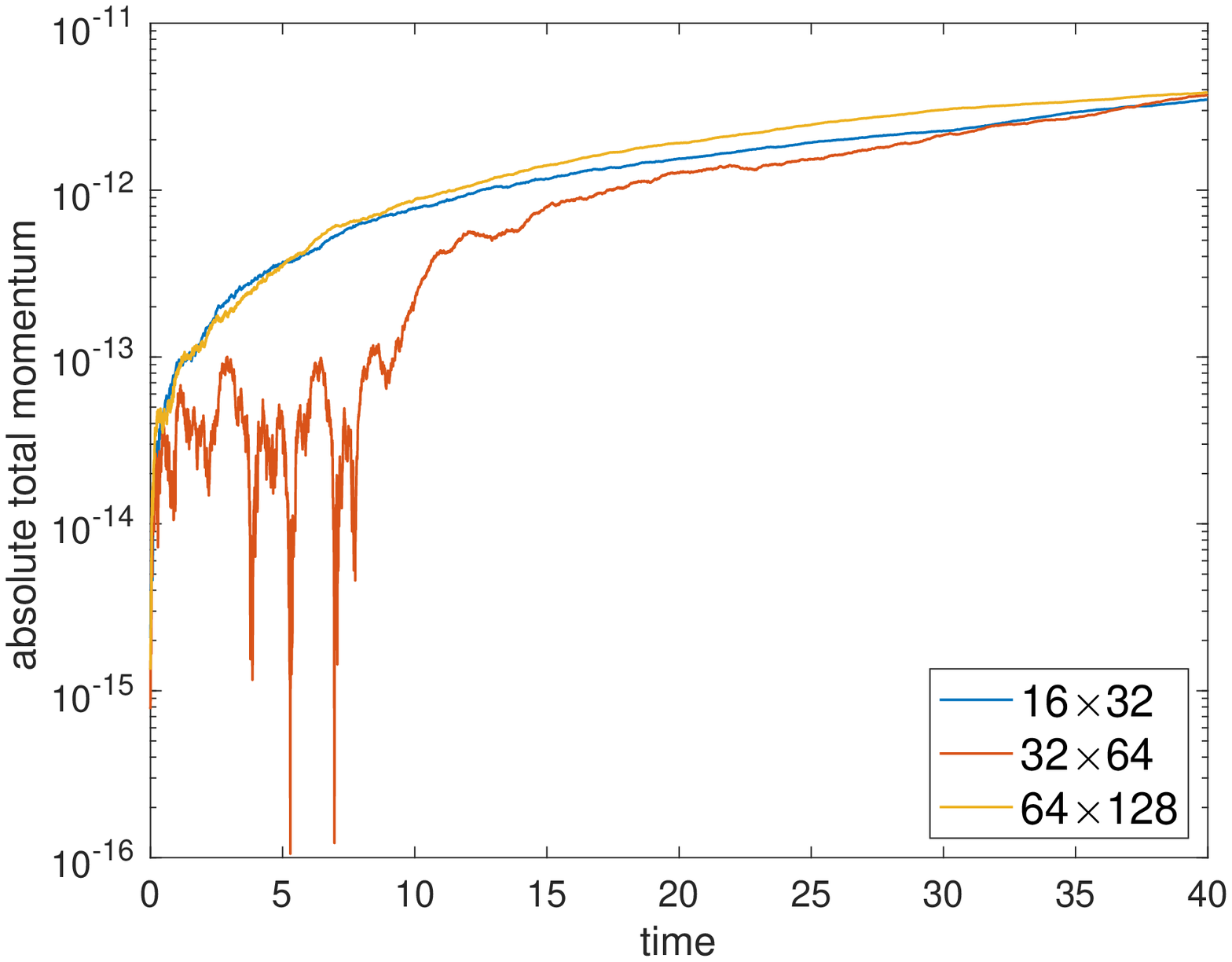}}
		\subfigure[]{\includegraphics[height=40mm]{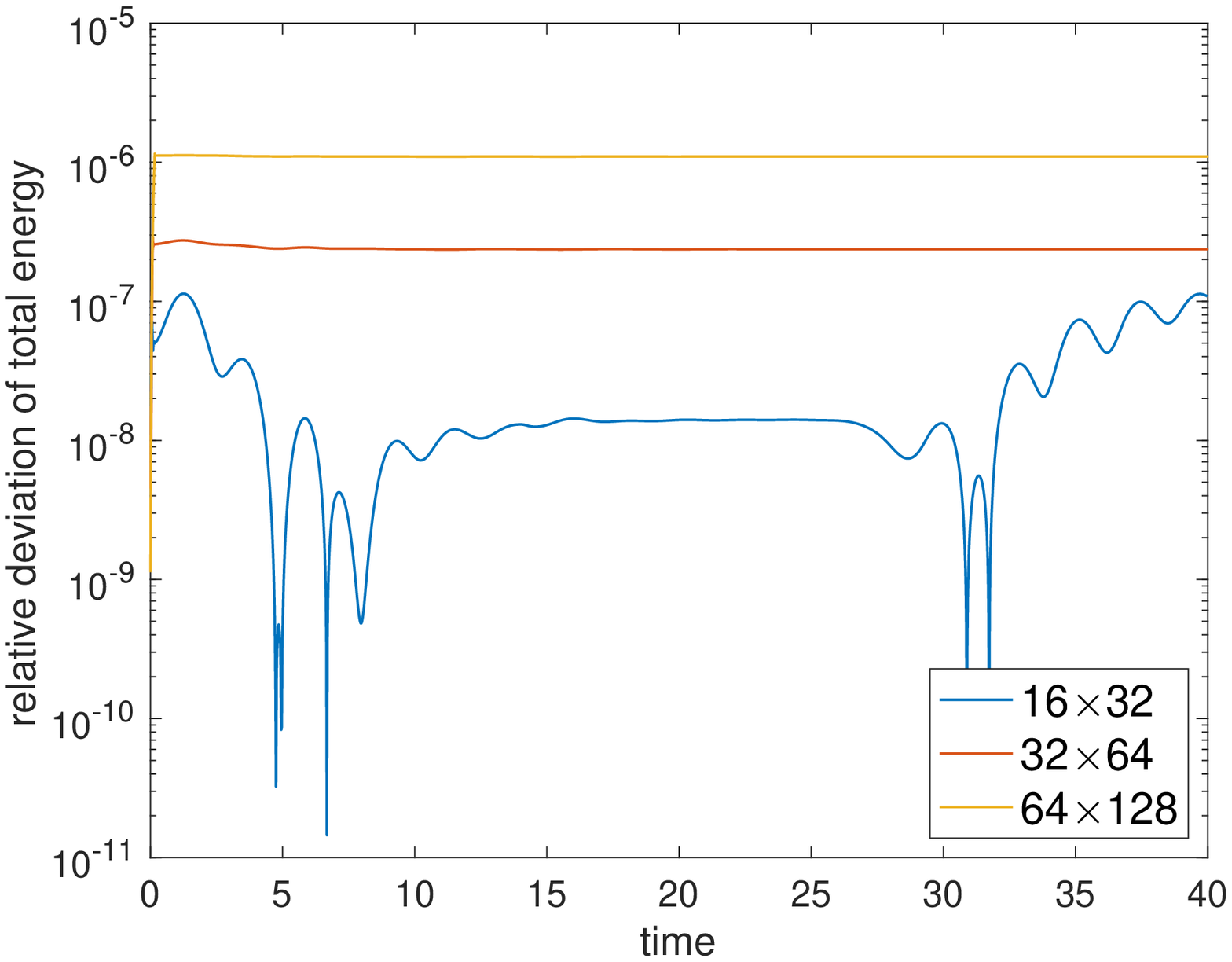}}
	\caption{Example \ref{ex:weak1d}.  The time evolution of relative deviation of total mass (a, d), absolute total momentum (b, d), and relative deviation of total energy (c, f).  Conservative method (a, b, c) and non-conservative method (d, e, f). $\varepsilon=10^{-5}$.}
	\label{fig:weak1d_invar}
\end{figure}

\begin{exa}\label{ex:strong1d}(Strong Laudau damping.) We consider the strong Landau damping test with the initial condition \eqref{eq:landau1d}
and a bigger perturbation parameter $\alpha = 0.5$. 
 \end{exa} 

Two truncation thresholds $\varepsilon = 10^{-3},\, 10^{-4}$ are used to compare the performance of the proposed conservative method with the non-conservative one. In Figure \ref{ex:strong1d}, we report the time evolution of the electric energy together with the ranks of the numerical solutions for $\varepsilon = 10^{-3}$. We observe that the conservative method is able to capture correctly the nonlinear dynamics of the strong Landau damping as opposed to the non-conservative method. This is because the truncation error due to the large threshold used greatly pollutes the accuracy for the non-conservative method, while by design the conservative method exactly conserves the mass and momentum densities in the low rank setting. Such conservation help resolve the nonlinear Vlasov dynamics with a relatively large truncation threshold. As observed in Figure \ref{fig:strong1d_invar1}, the conservative method can conserve the total mass and momentum up to the machine precision regardless of the mesh size used. Again, the energy conservation is not observed, but the conservation error decreases with mesh refinement, which is not the case for the non-conservative method. Then we consider a smaller truncation threshold $\varepsilon = 10^{-4}$, and the truncation error is reduced accordingly. Both methods generate consitent results as plotted in Figure \ref{fig:strong1d_elec2}. As with the weak case, the conservative method has slightly larger ranks than the non-conservative method. We have a similar observation of the methods in 
conserving the invariants in Figure \ref{fig:strong1d_invar2}, as that in Figure \ref{fig:strong1d_invar1}. 

\begin{figure}[h!]
	\centering
	\subfigure[]{\includegraphics[height=50mm]{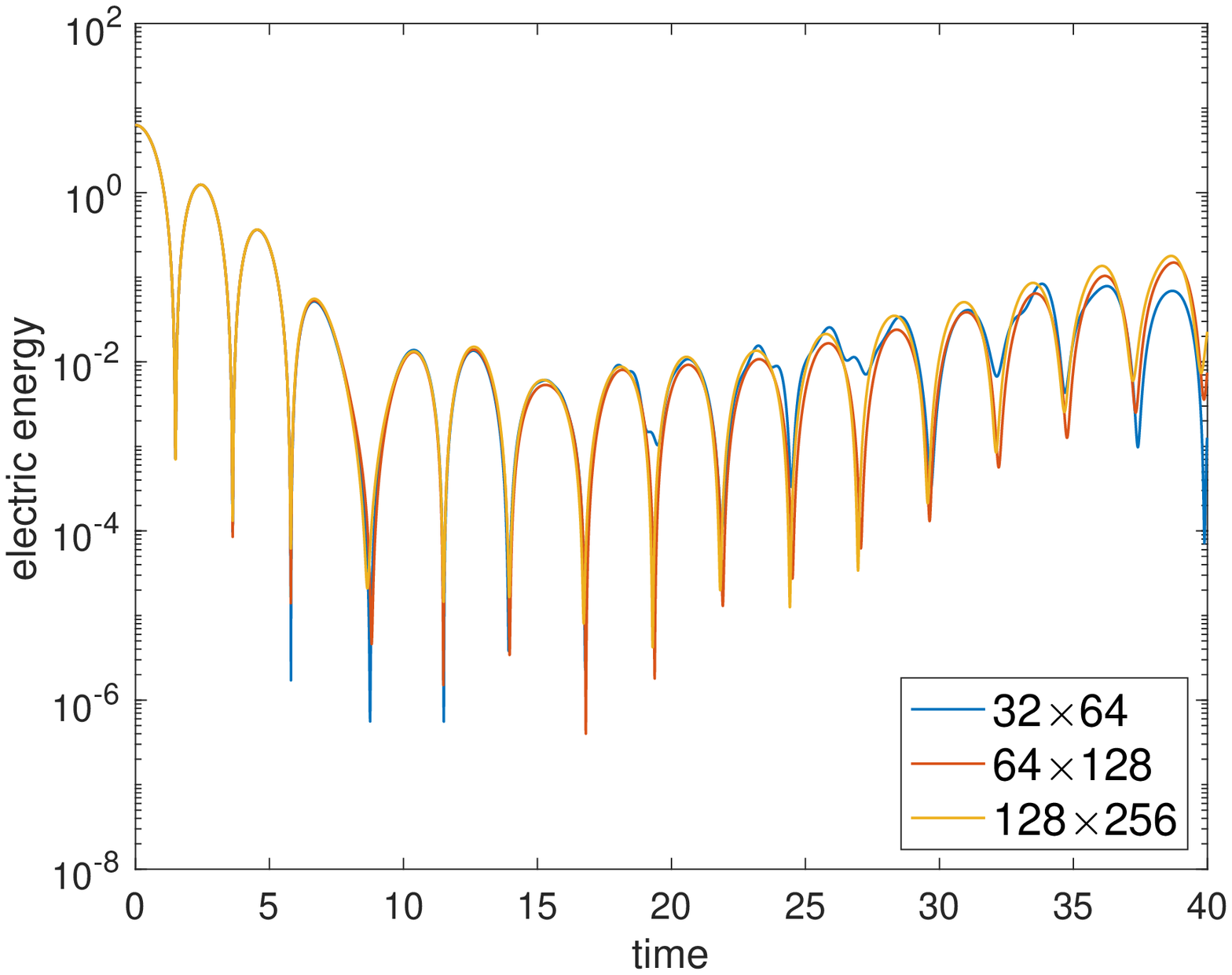}}
		\subfigure[]{\includegraphics[height=50mm]{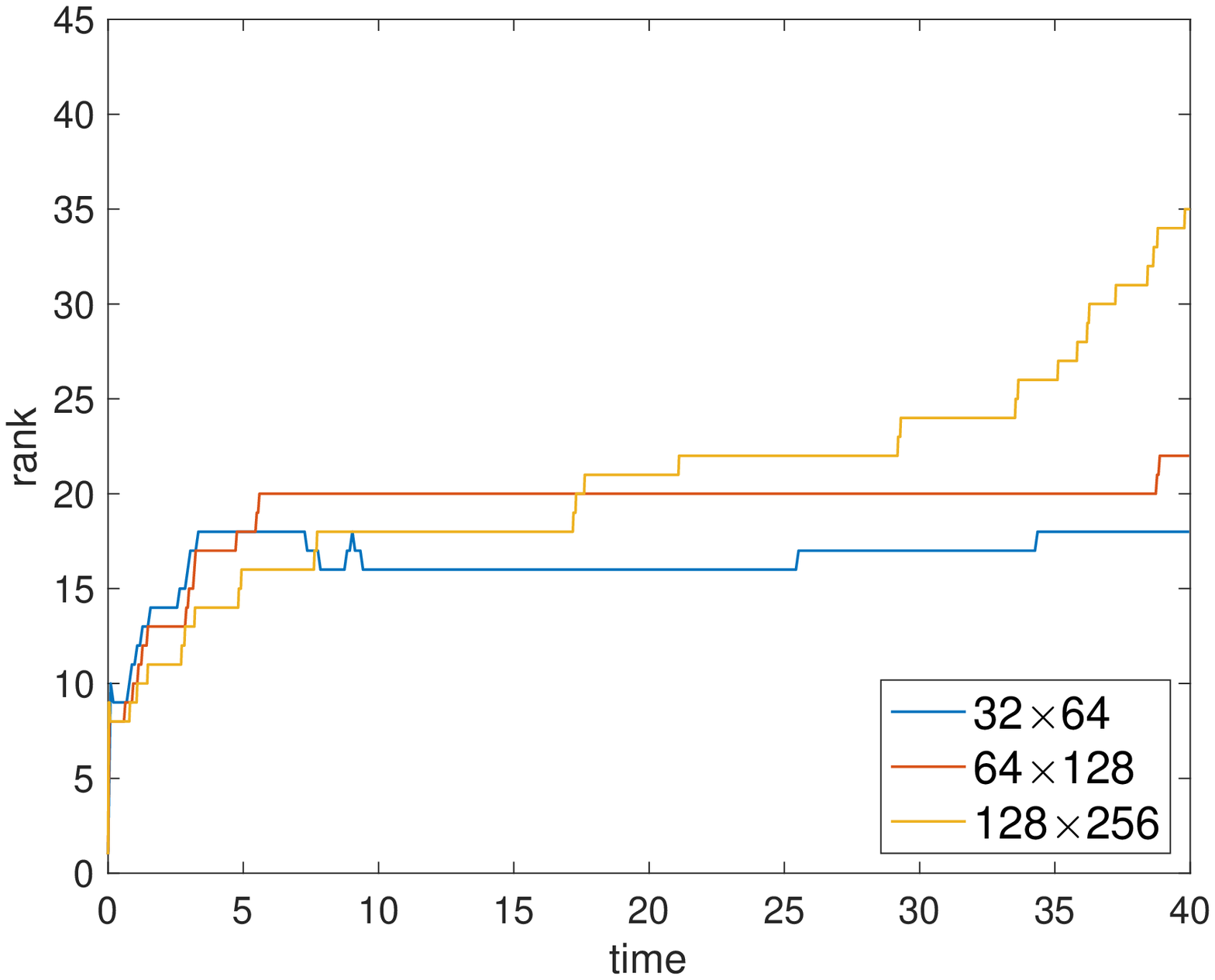}}
			\subfigure[]{\includegraphics[height=50mm]{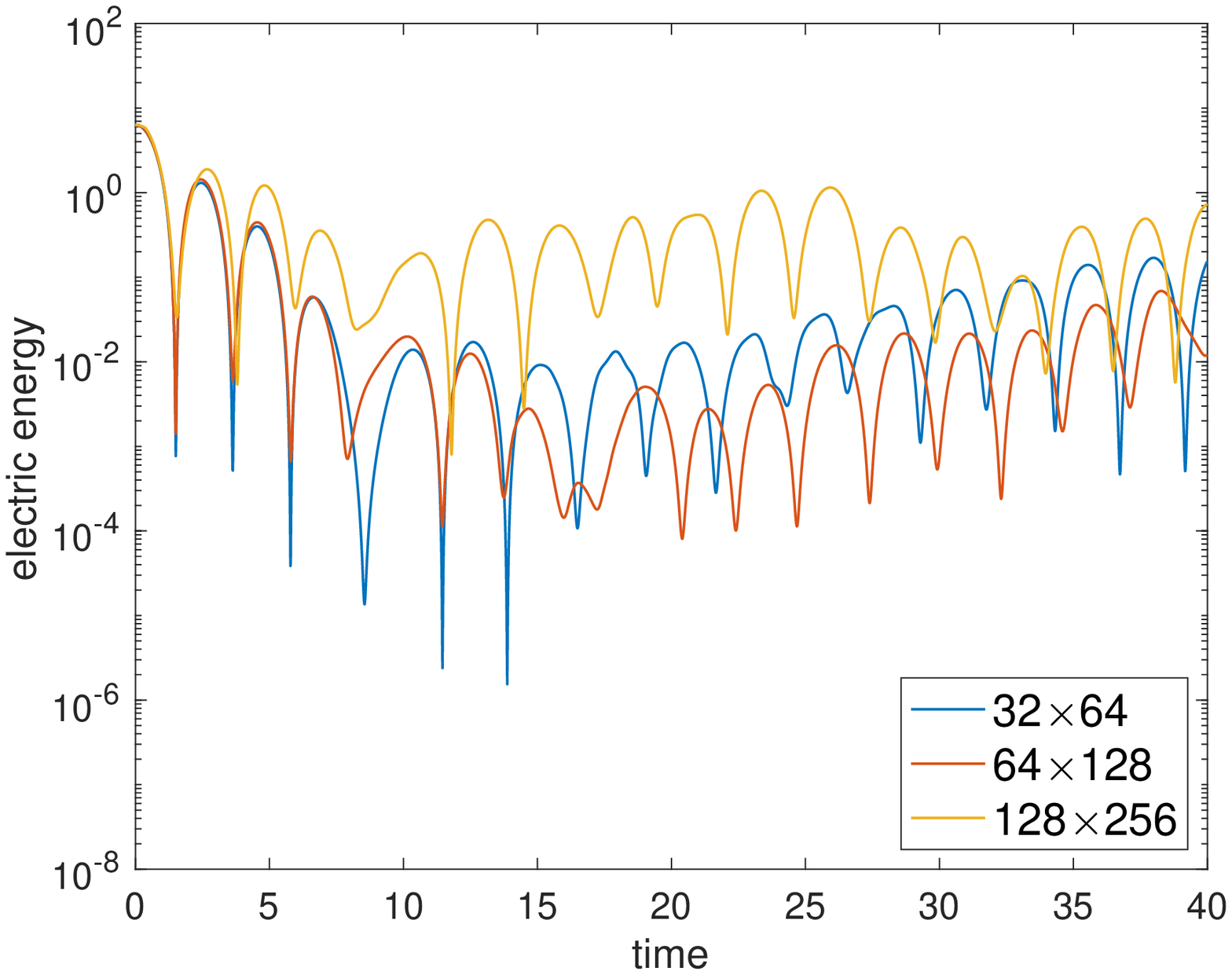}}
		\subfigure[]{\includegraphics[height=50mm]{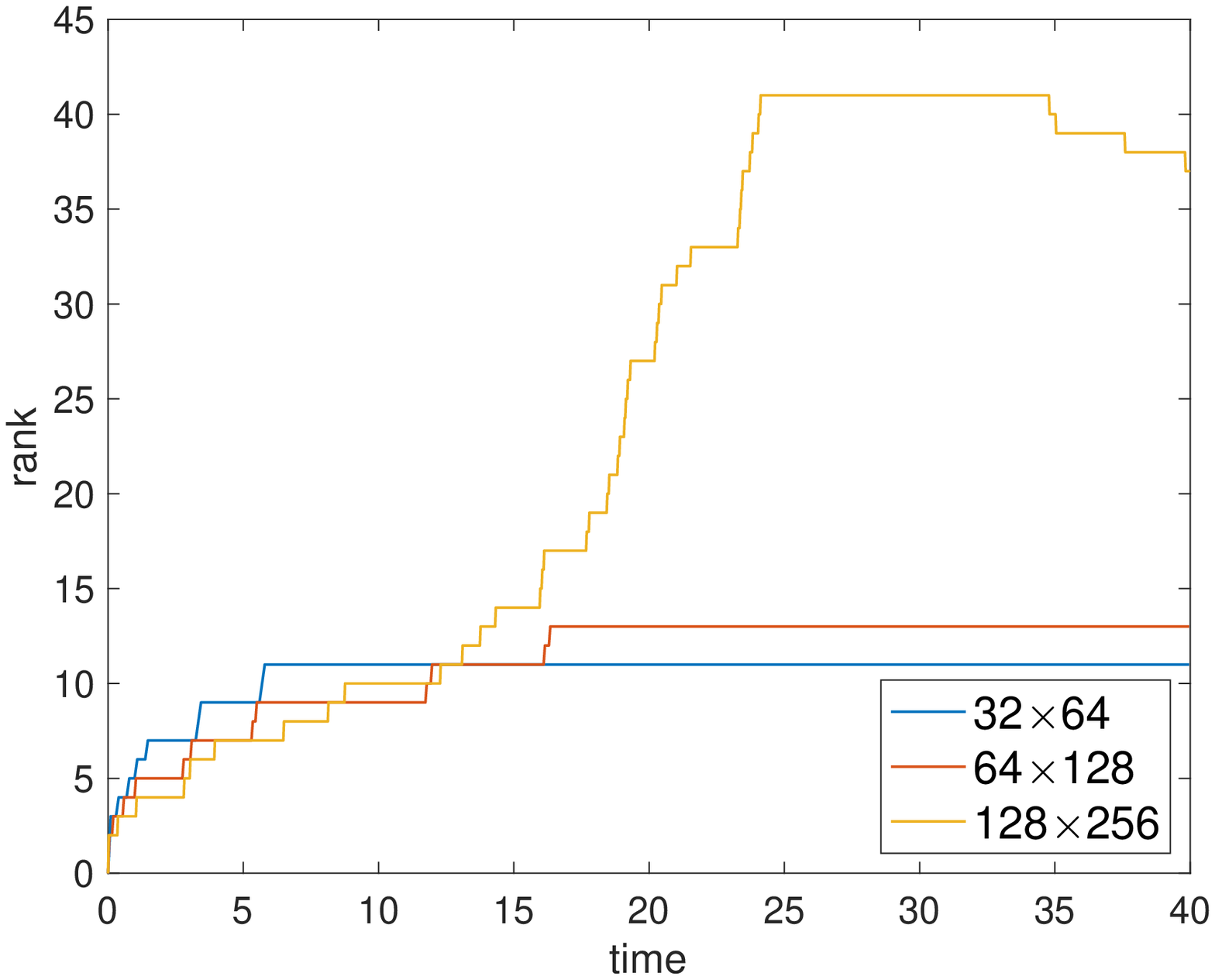}}
	\caption{Example \ref{ex:strong1d}.  The time evolution of the electric energy (a, c) and the rank of the numerical solutions (b, d). Conservative method (a, b) and non-conservative method (c, d). $\varepsilon=10^{-3}$.}
	\label{fig:strong1d_elec1}
\end{figure}

\begin{figure}[h!]
	\centering
	\subfigure[]{\includegraphics[height=40mm]{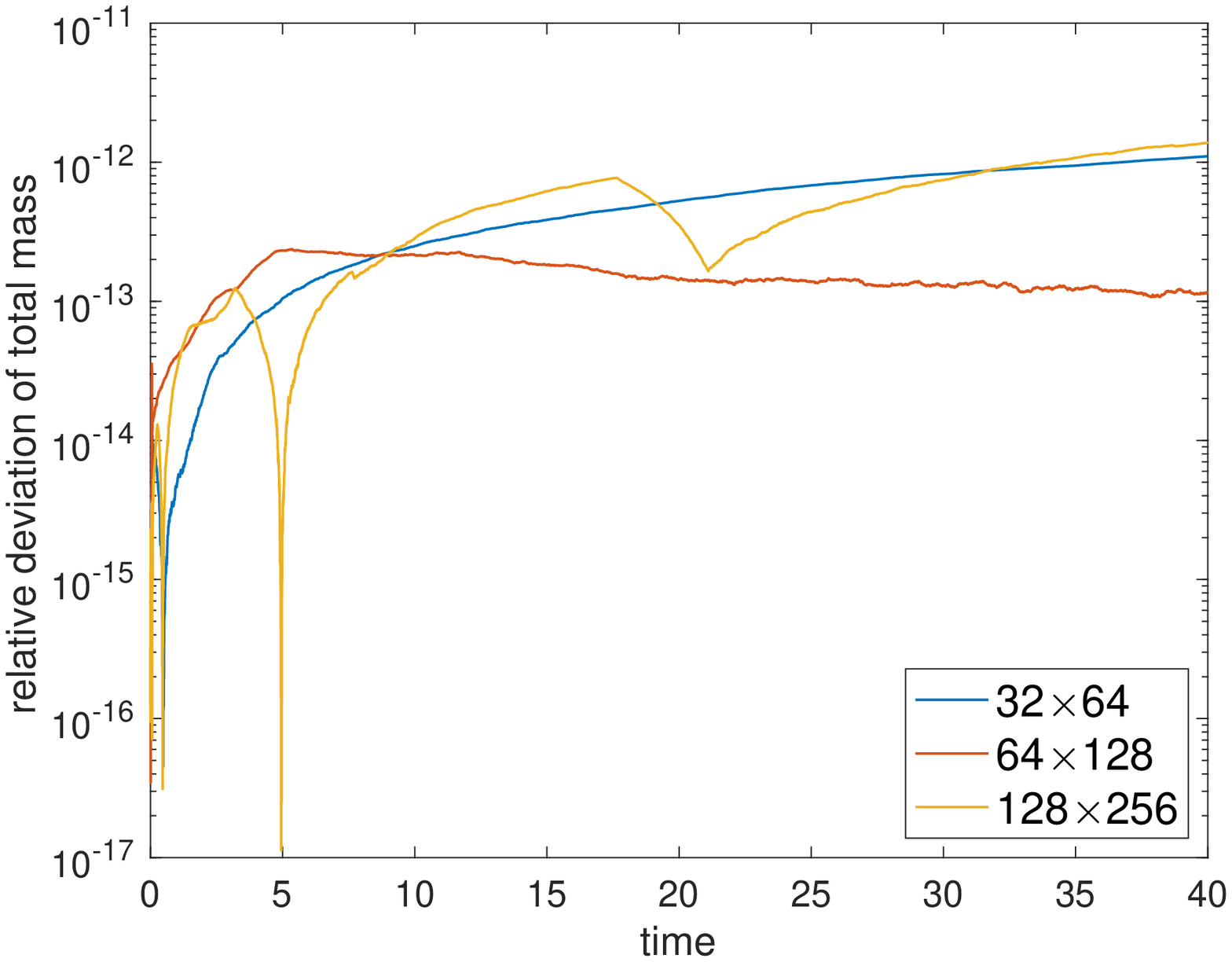}}
		\subfigure[]{\includegraphics[height=40mm]{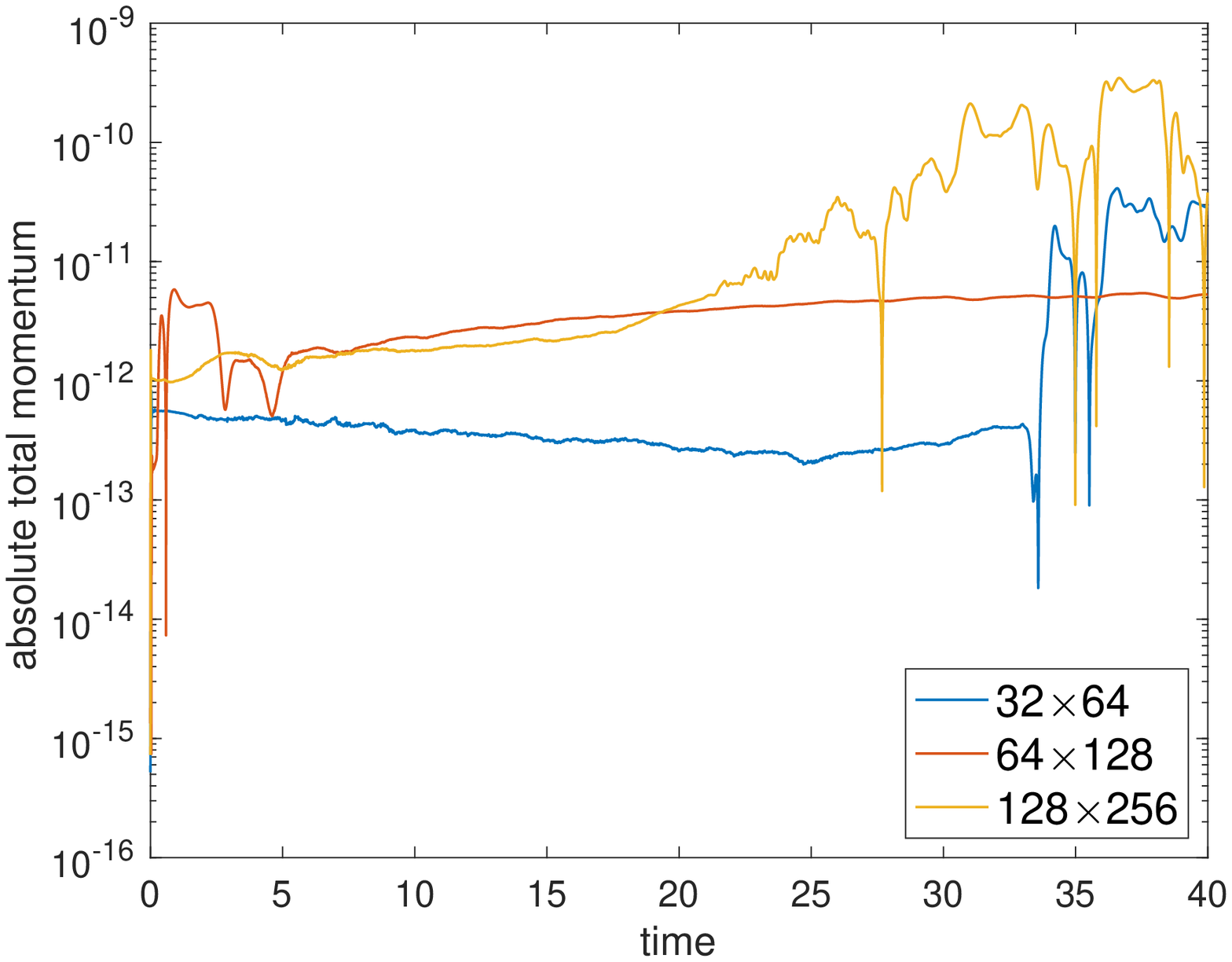}}
		\subfigure[]{\includegraphics[height=40mm]{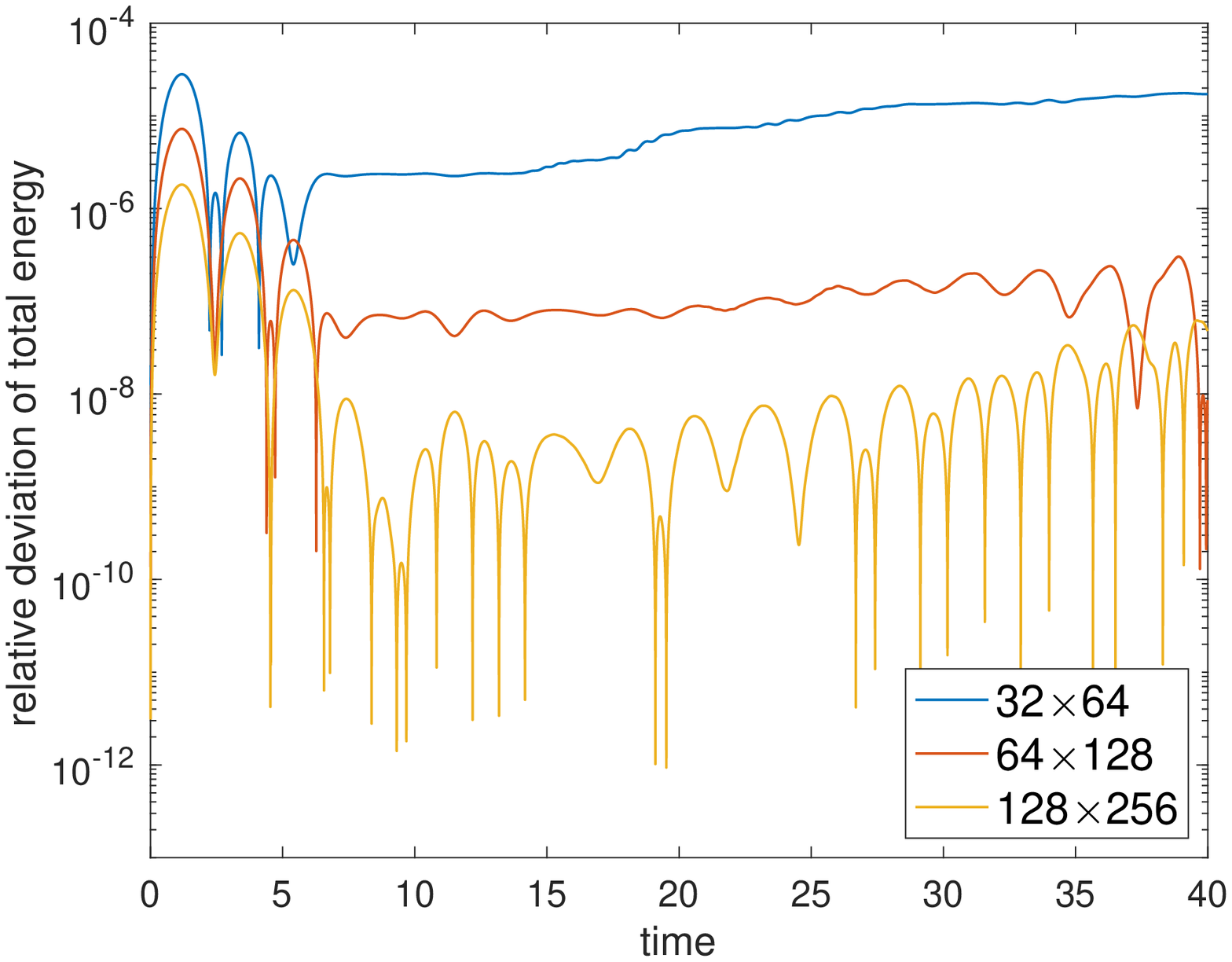}}
			\subfigure[]{\includegraphics[height=40mm]{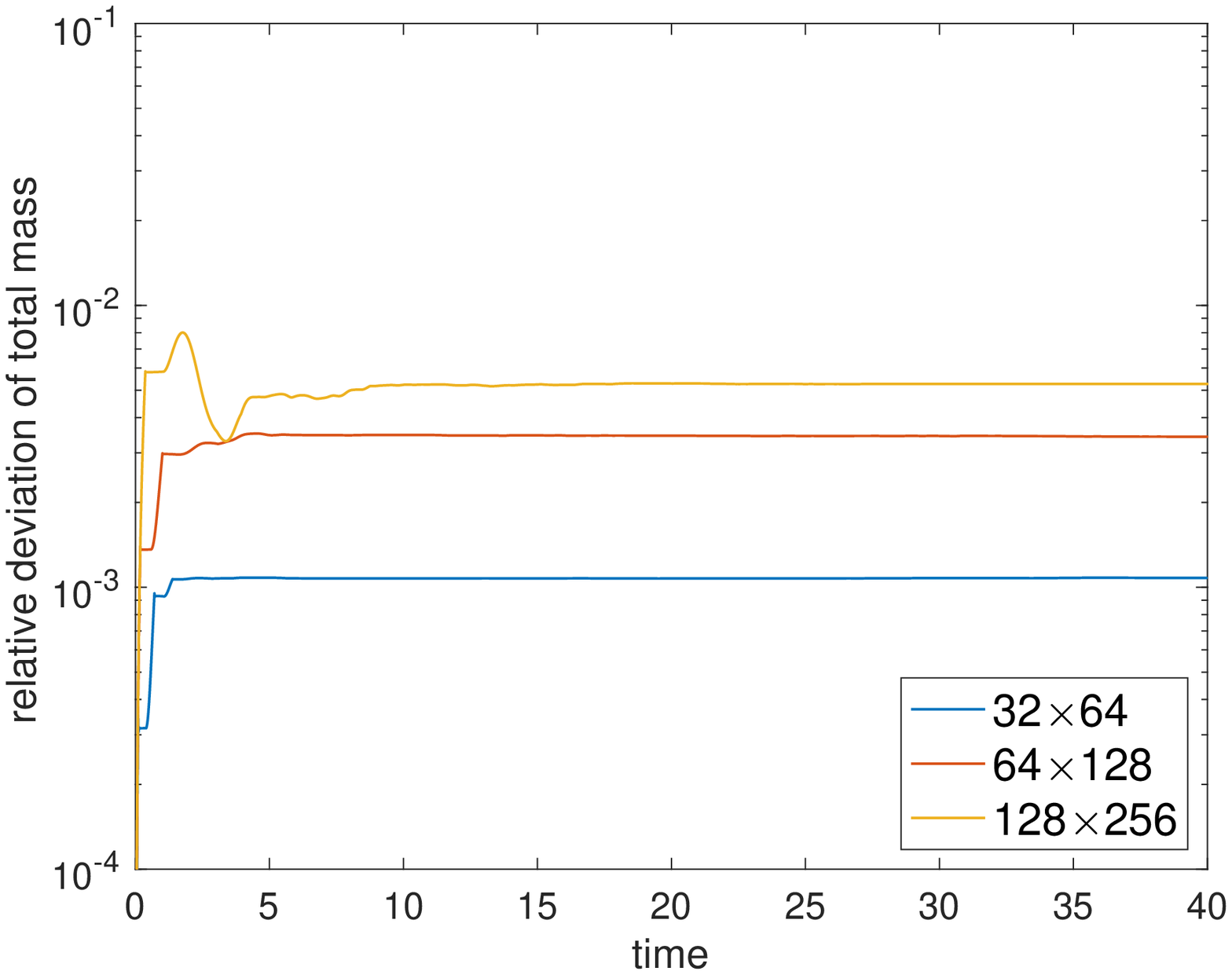}}
		\subfigure[]{\includegraphics[height=40mm]{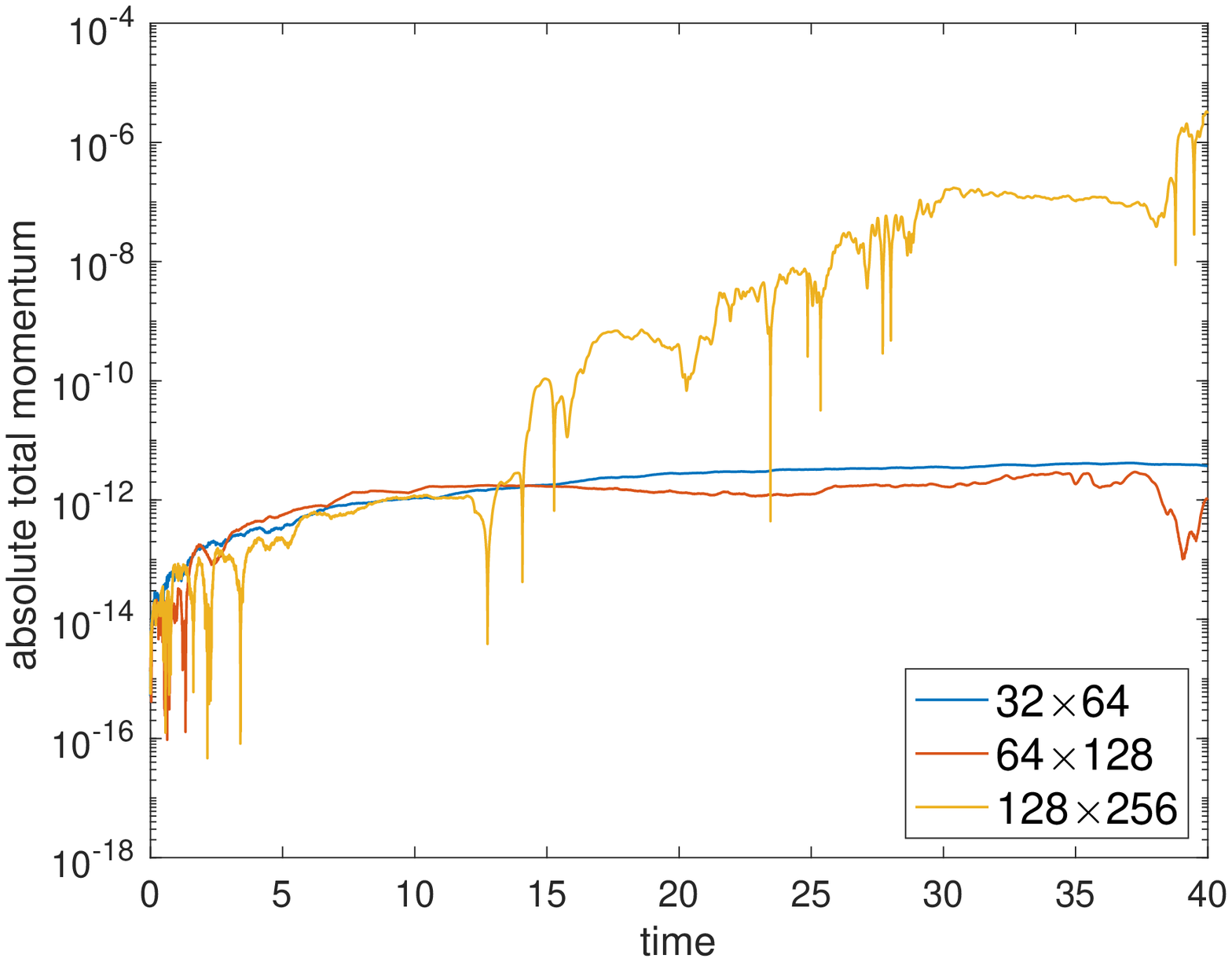}}
		\subfigure[]{\includegraphics[height=40mm]{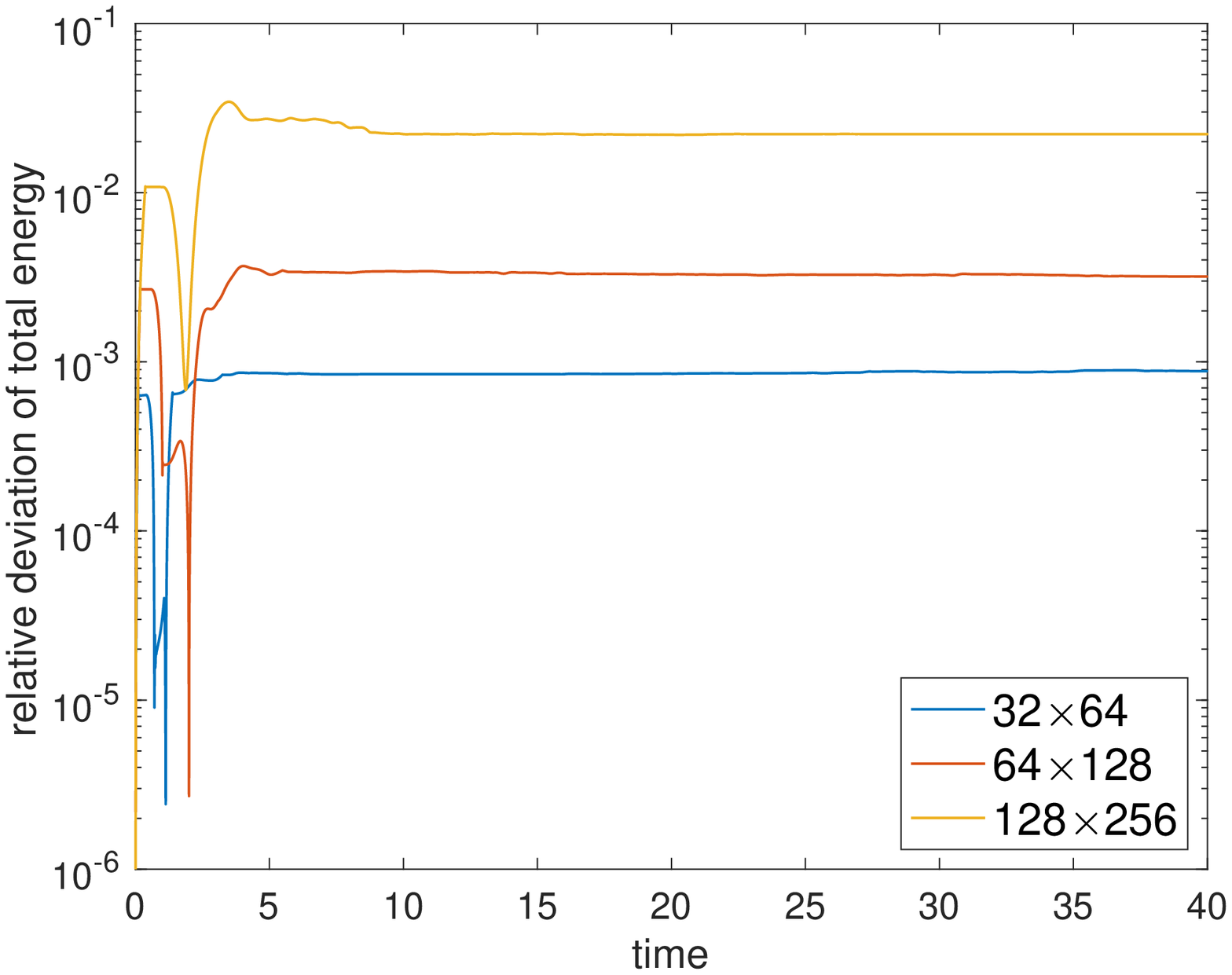}}
	\caption{Example \ref{ex:strong1d}.  The time evolution of relative deviation of total mass (a, d), absolute total momentum (b, e), and relative deviation of total energy (c, f). Conservative method (a, b, c) and non-conservative method (d, e, f). $\varepsilon=10^{-3}$.}
	\label{fig:strong1d_invar1}
\end{figure}

\begin{figure}[h!]
	\centering
	\subfigure[]{\includegraphics[height=50mm]{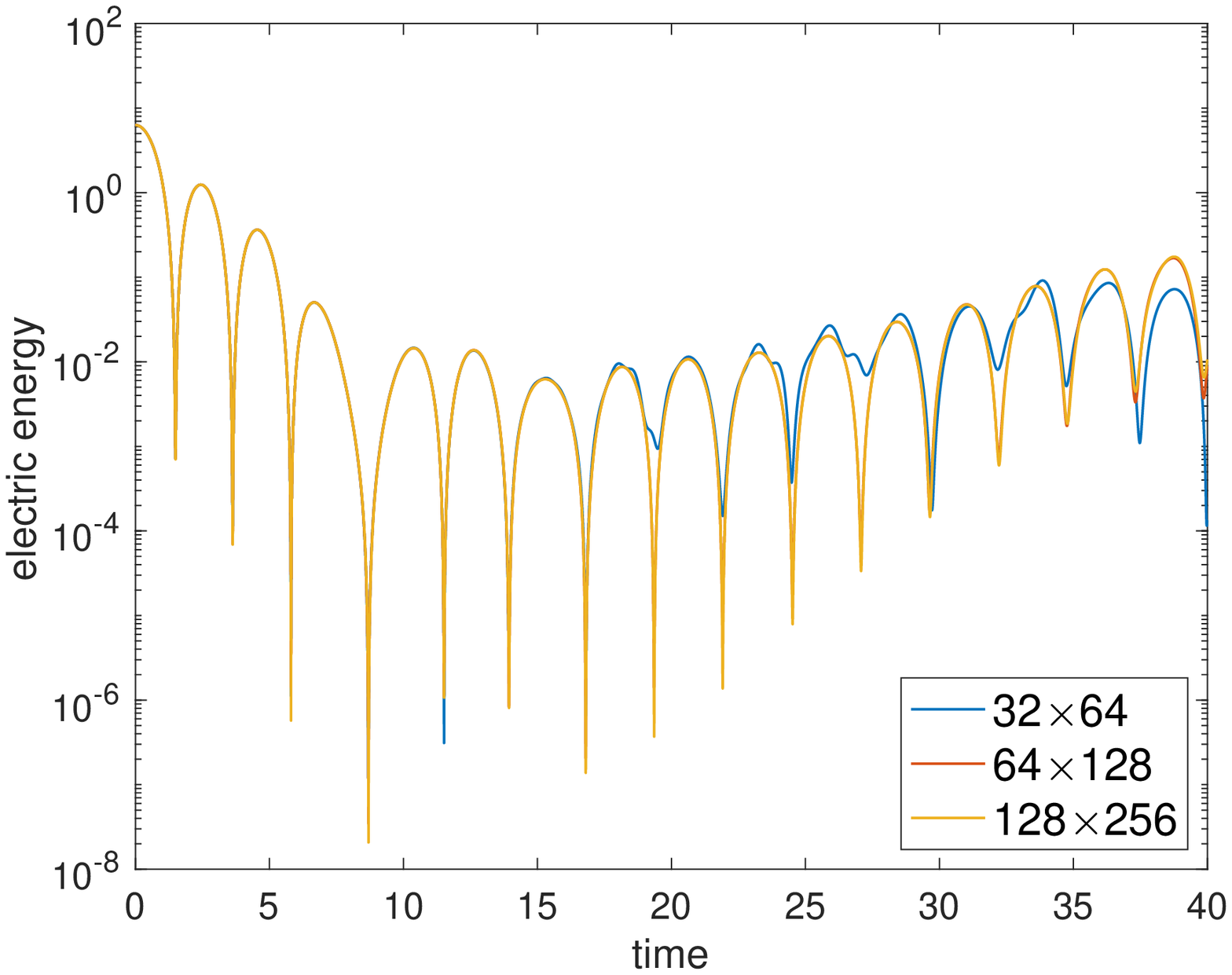}}
		\subfigure[]{\includegraphics[height=50mm]{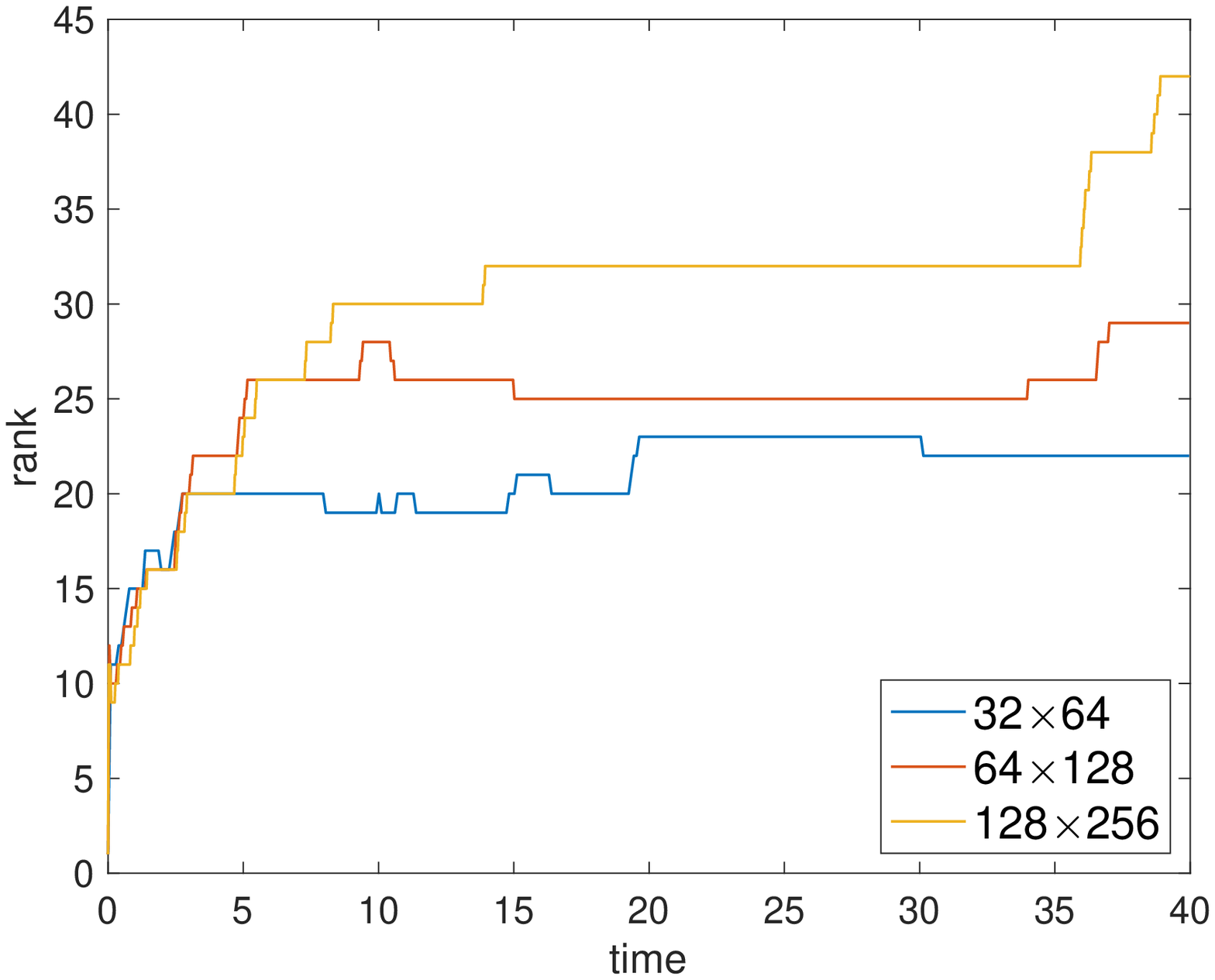}}
			\subfigure[]{\includegraphics[height=50mm]{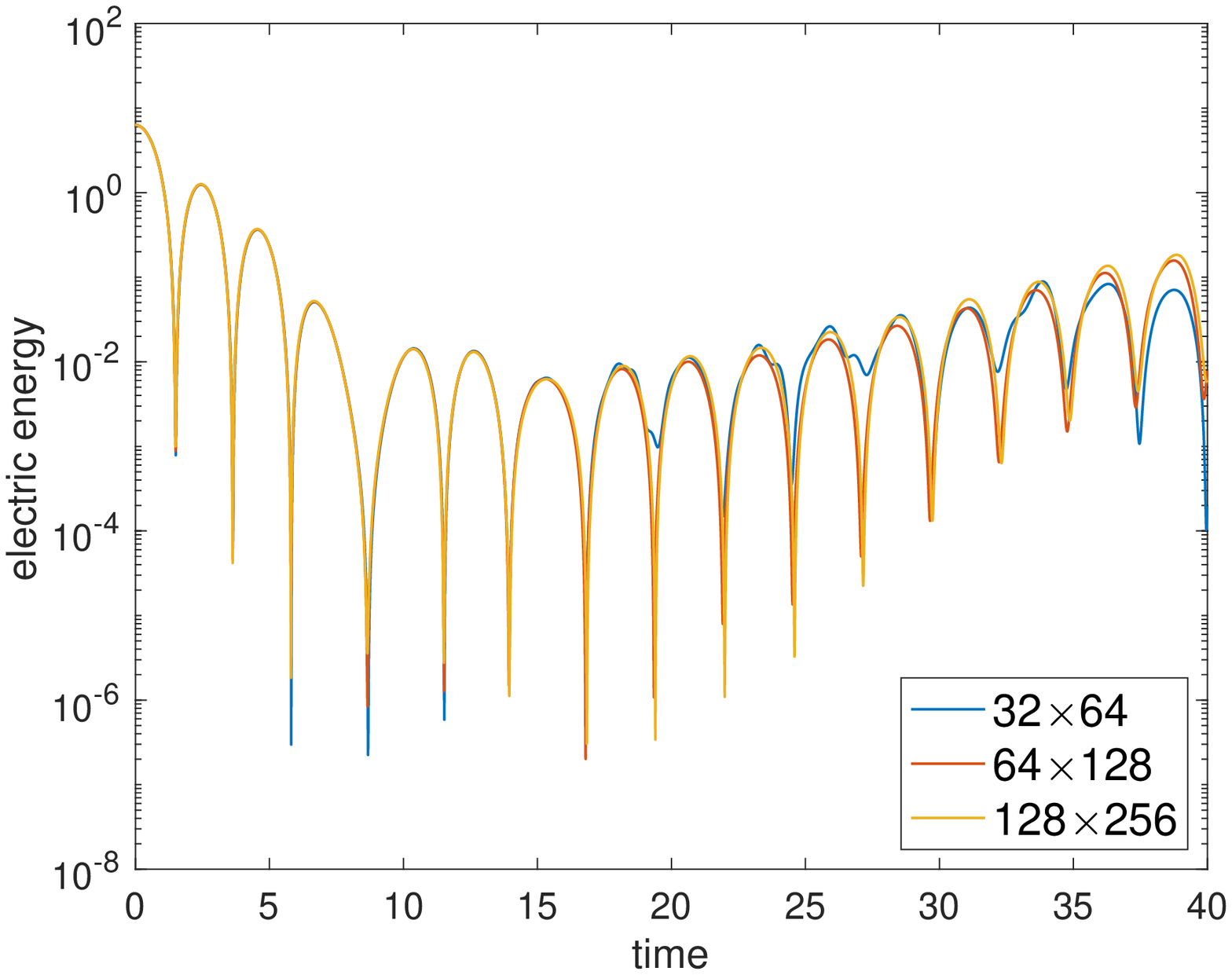}}
		\subfigure[]{\includegraphics[height=50mm]{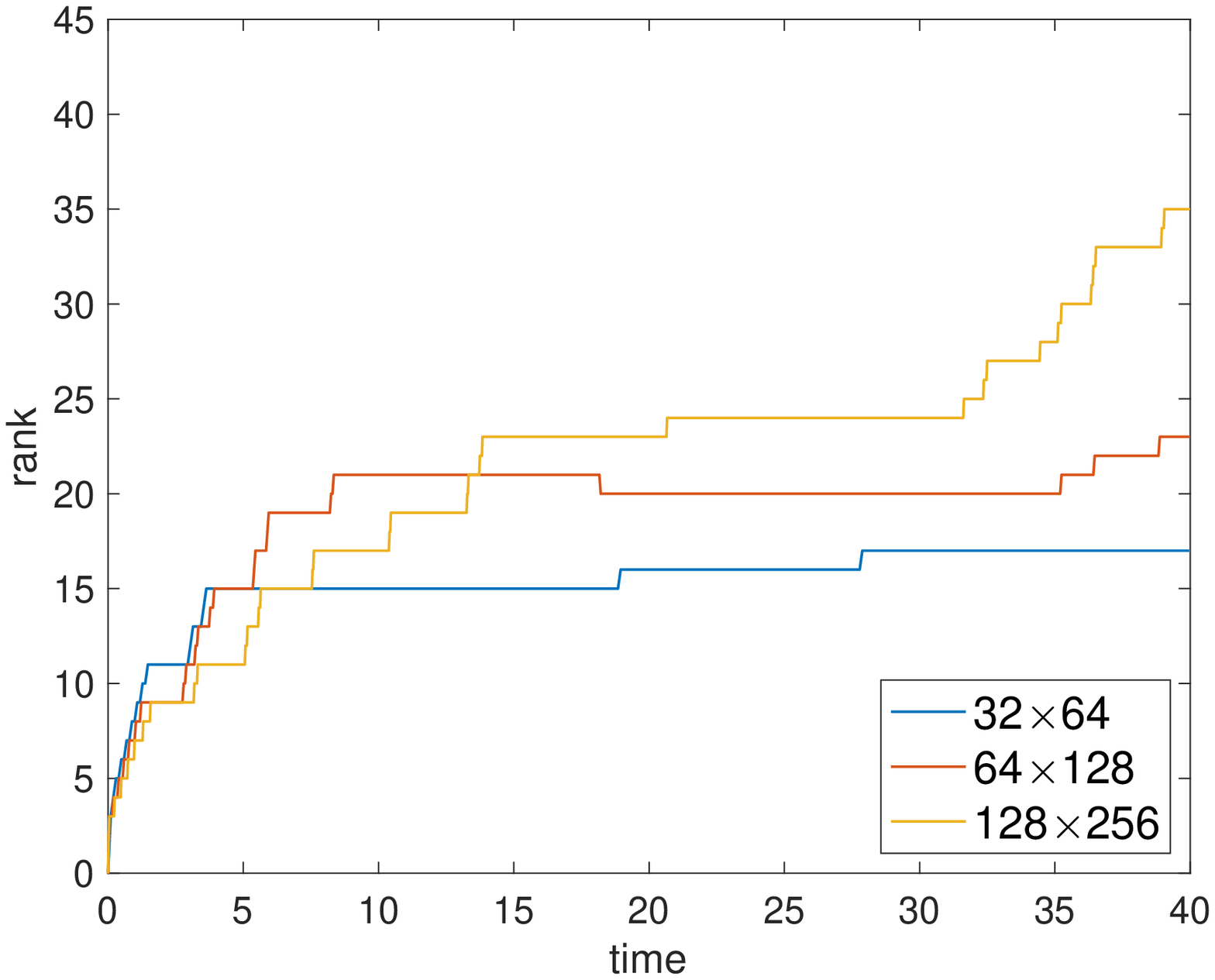}}
	\caption{Example \ref{ex:strong1d}.  The time evolution of the electric energy (a, c) and the ranks of the numerical solutions (b, d). Conservative method (a, b) and non-conservative method (c, d).  $\varepsilon=10^{-4}$.}
	\label{fig:strong1d_elec2}
\end{figure}

\begin{figure}[h!]
	\centering
	\subfigure[]{\includegraphics[height=40mm]{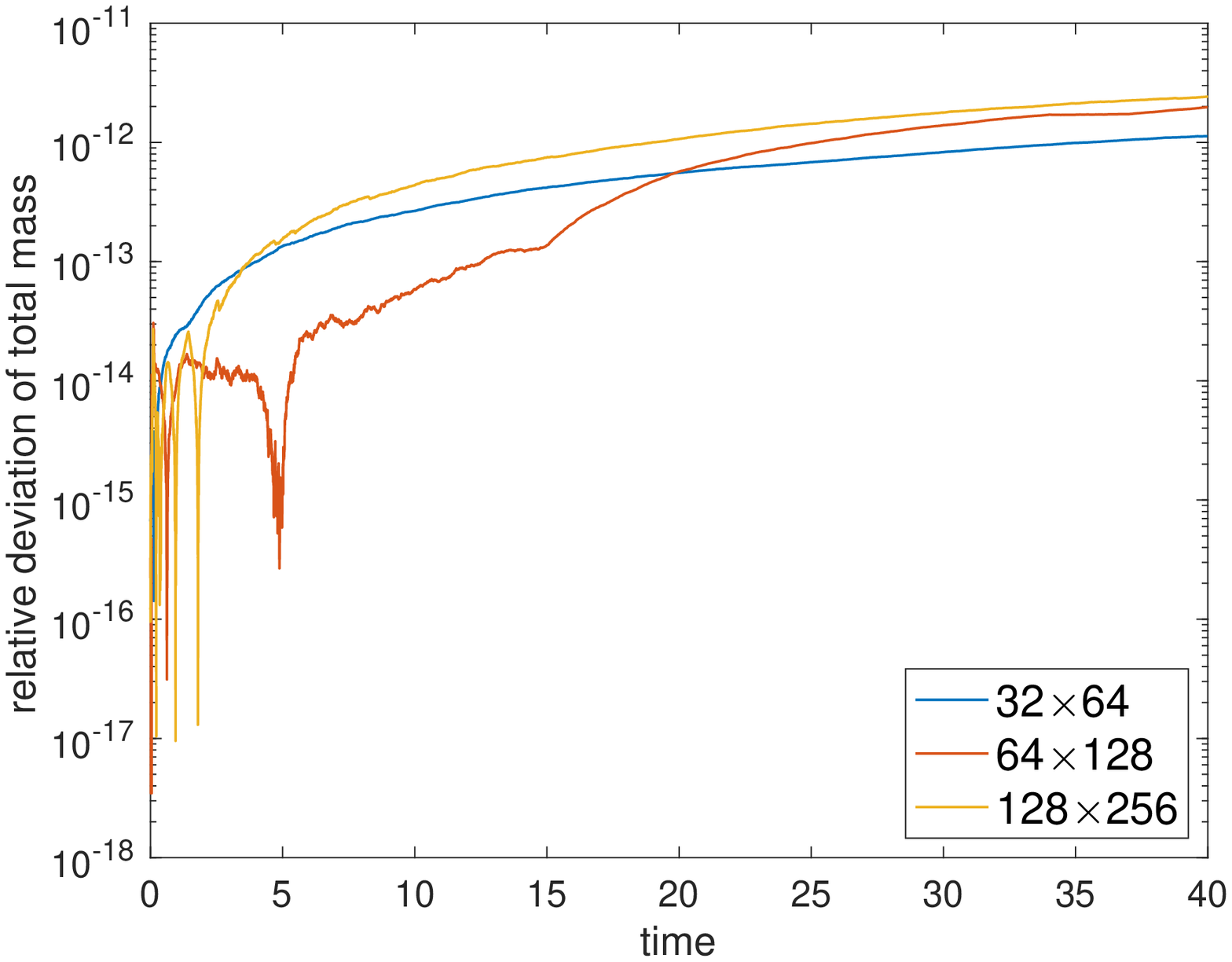}}
		\subfigure[]{\includegraphics[height=40mm]{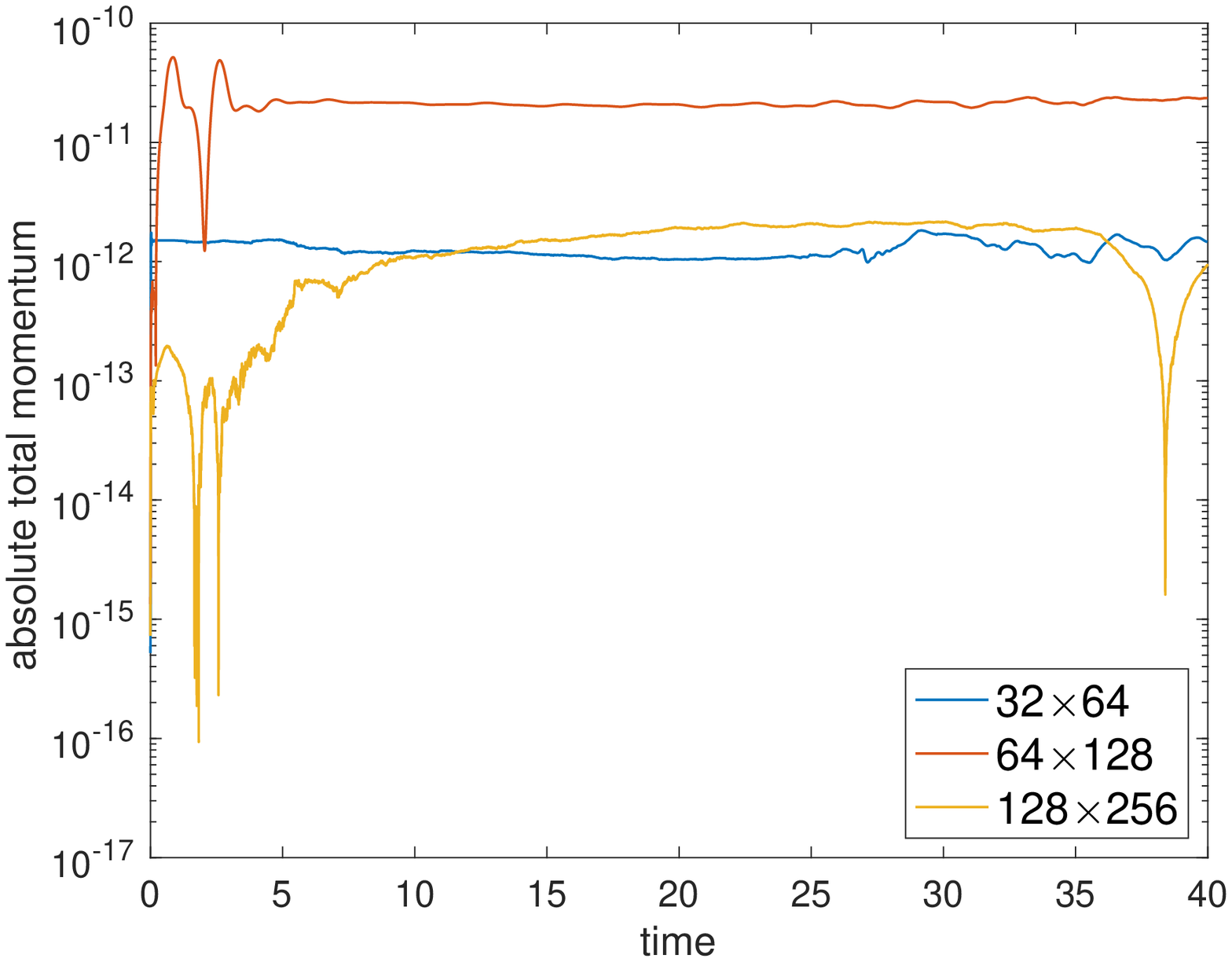}}
		\subfigure[]{\includegraphics[height=40mm]{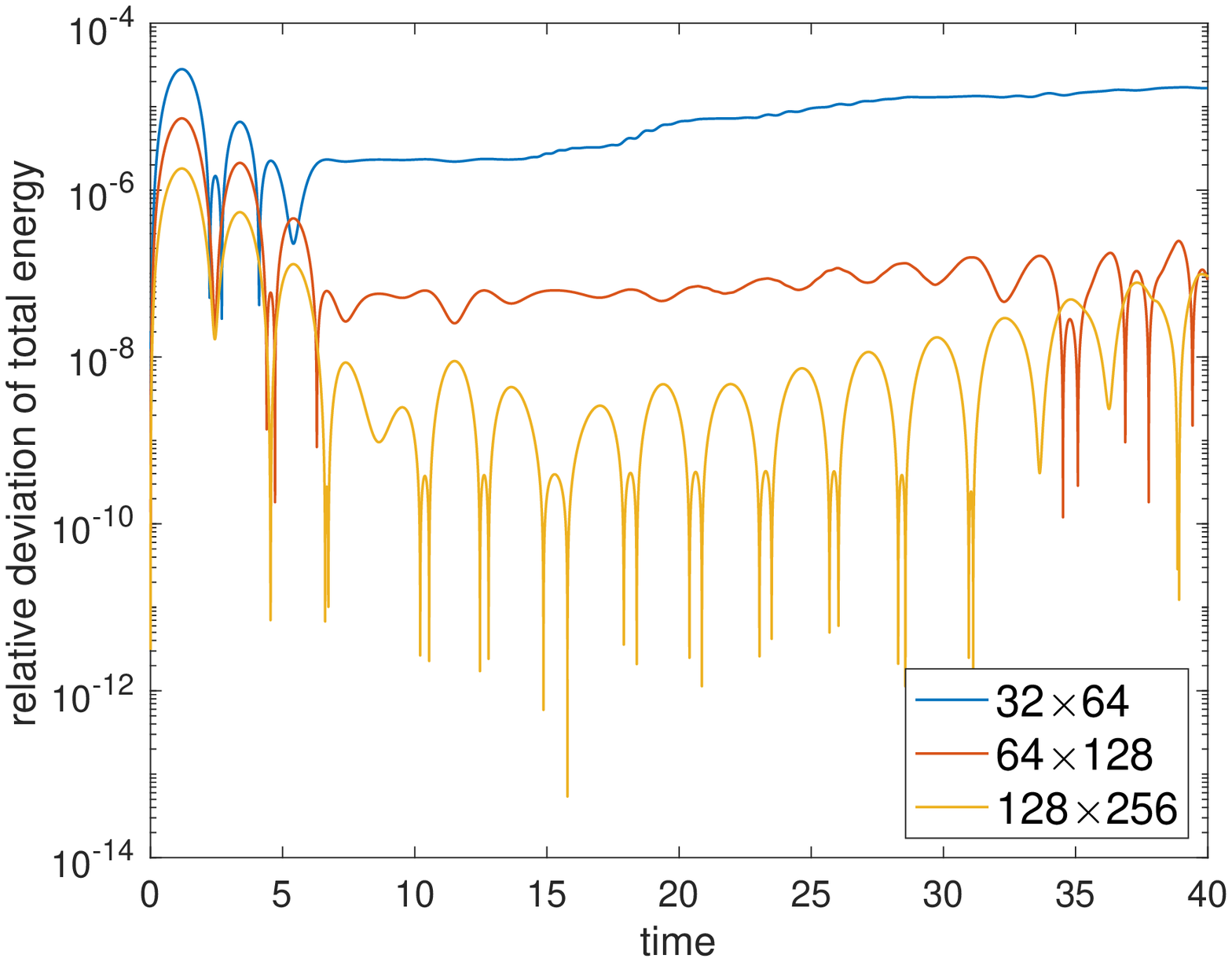}}
			\subfigure[]{\includegraphics[height=40mm]{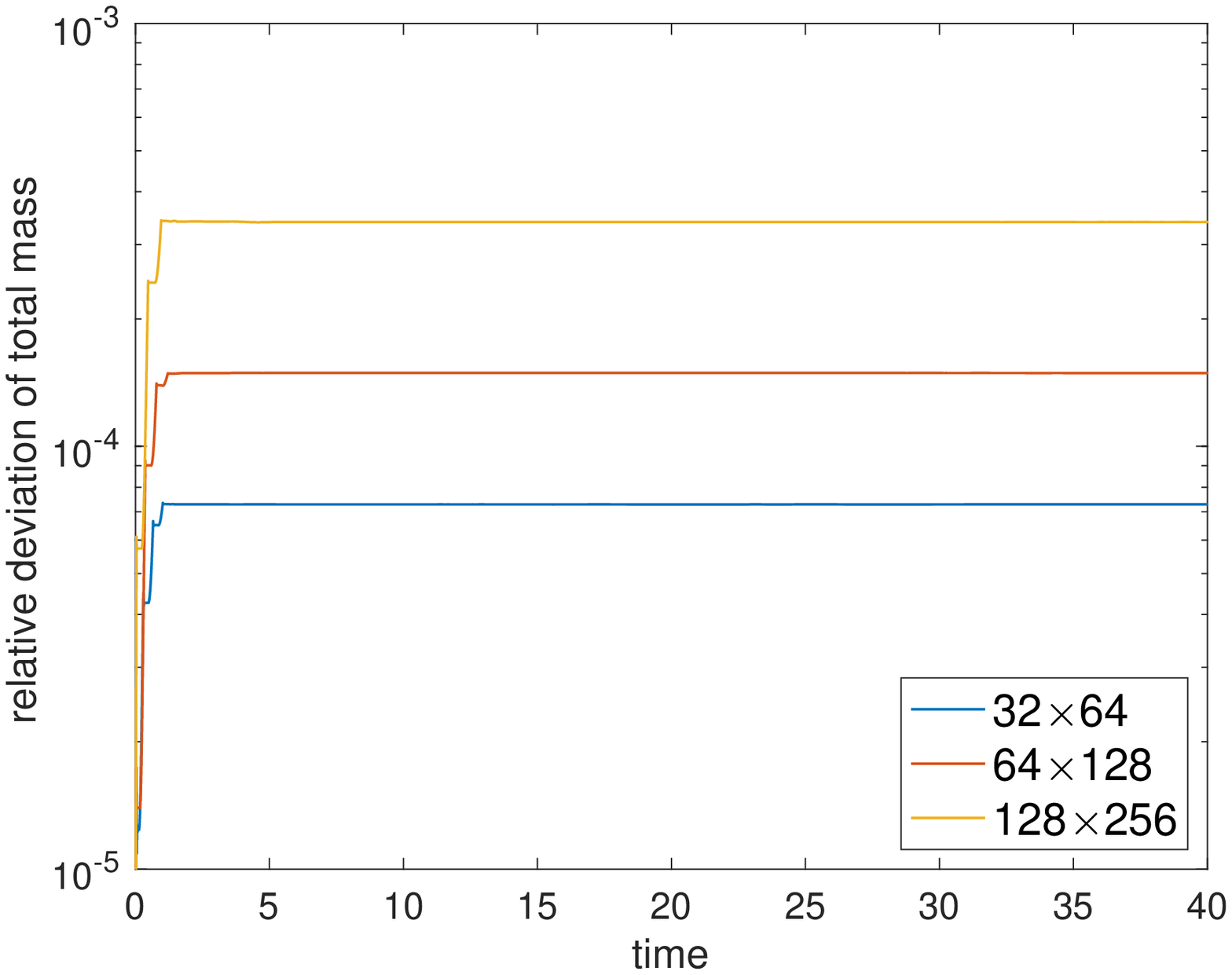}}
		\subfigure[]{\includegraphics[height=40mm]{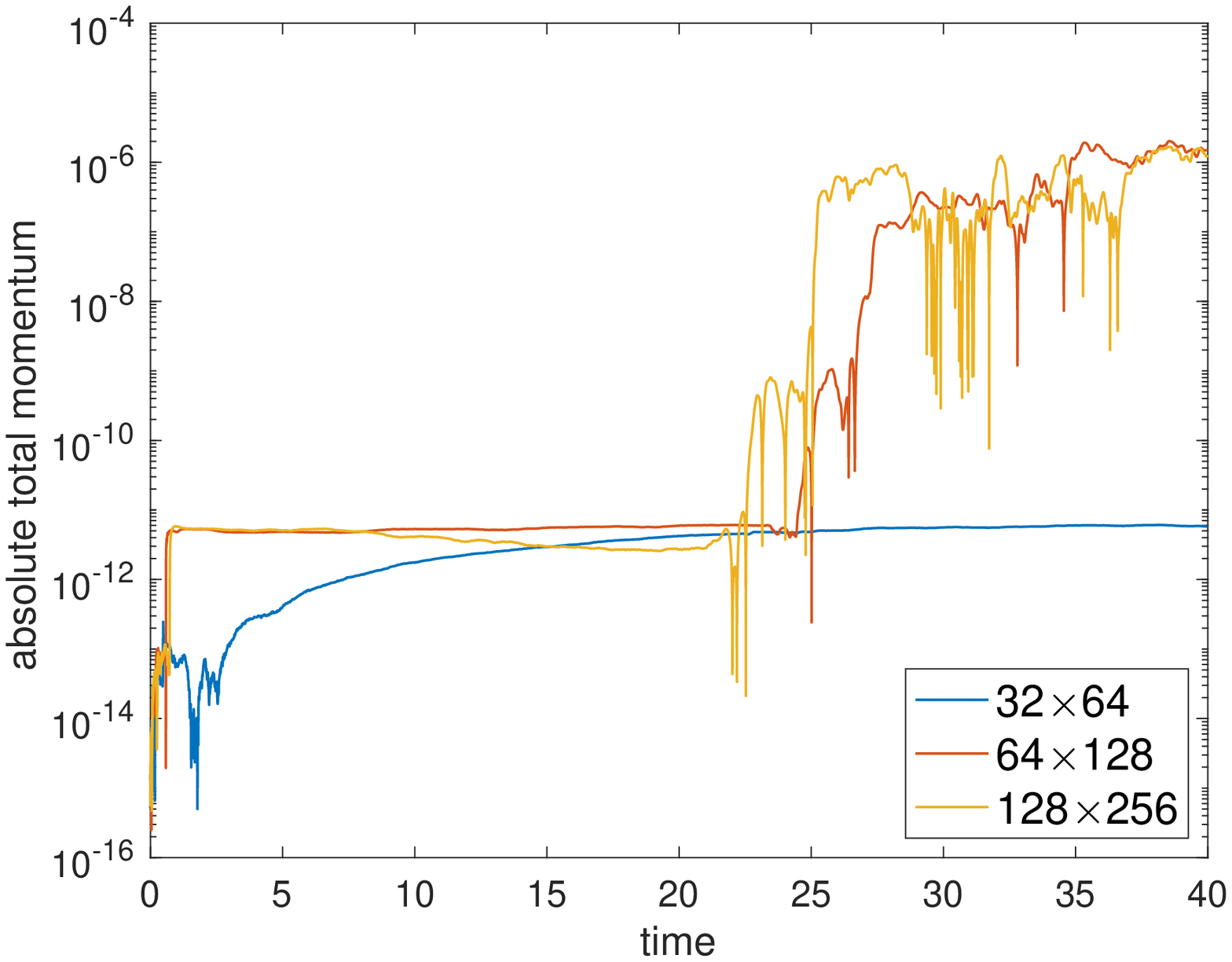}}
		\subfigure[]{\includegraphics[height=40mm]{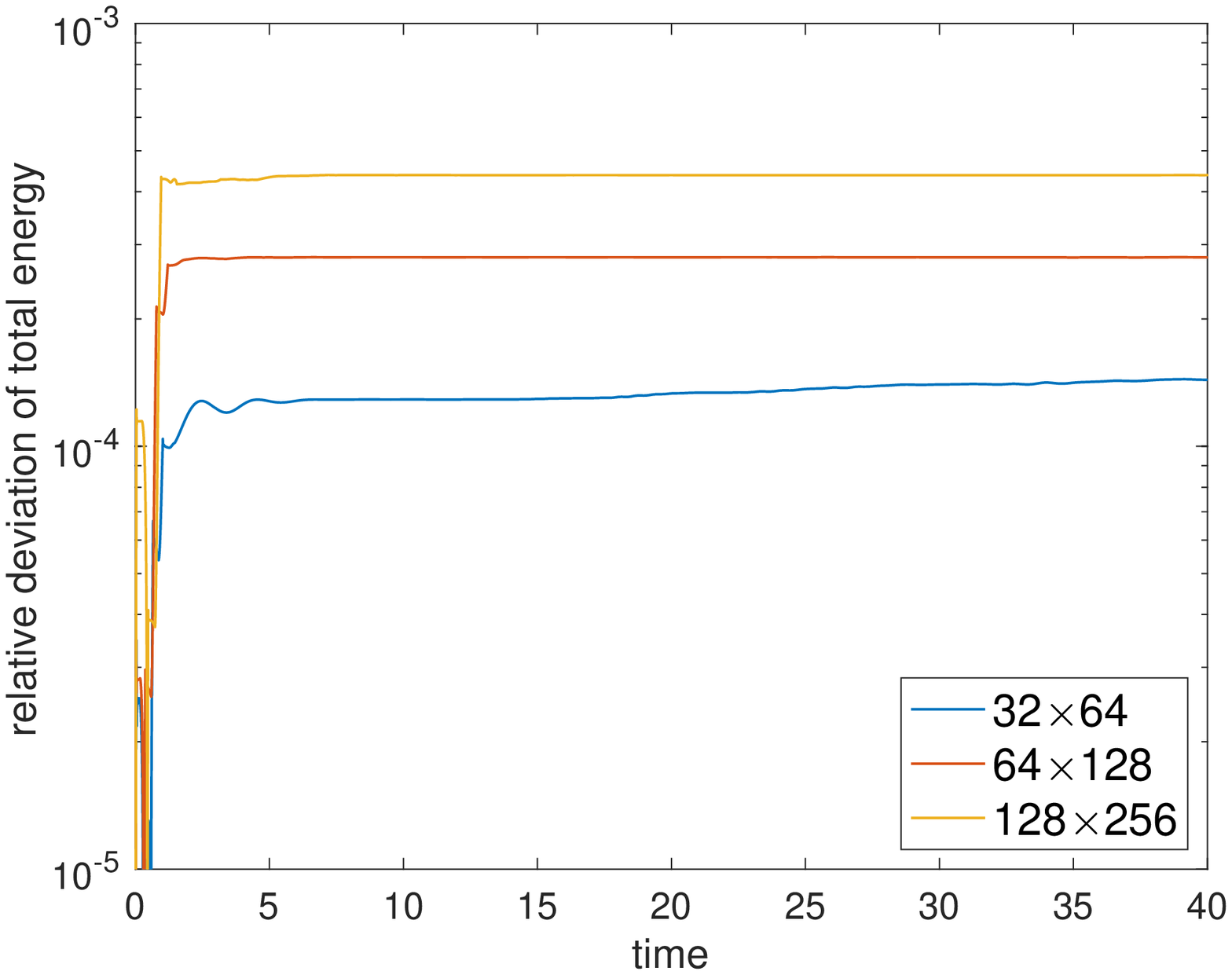}}
	\caption{Example \ref{ex:strong1d}.  The time evolution of relative deviation of total mass (a, d), absolute total momentum (b, e), and relative deviation of total energy (c, f). Conservative method (a, b, c) and non-conservative method (d, e, f).  $\varepsilon=10^{-4}$.}
	\label{fig:strong1d_invar2}
\end{figure}

 \begin{exa}\label{ex:bumpontail} (Bump on tail.) In this example, we simulate the bump-on-tail test with the initial condition
 \begin{equation}
\label{eq:bump1d}
f(x,v,t=0) = \left(1+\alpha  \cos \left(k x\right)\right)\left(n_p\exp\left(-\frac{v^2}{2}\right) +n_b\exp\left(-\frac{(v-u)^2}{2v_{t}}\right) \right),
\end{equation}
where $\alpha=0.04$, $k=0.3$, $n_{p}=\frac{9}{10 \sqrt{2 \pi}}$, $n_{b}=\frac{2}{10 \sqrt{2 \pi}}$, $u=4.5$, $v_{t}=0.5$. 
 \end{exa}
 
In the simulation, we set $\varepsilon=10^{-4}$ for truncation. The weight function $w(v) = \exp(-\frac{v^2}{3})$ is chosen for the conservative method. Note that unlike the previous Landau damping examples, the solution of bump on tail does not have the symmetry. In Figure \ref{fig:bumpontail_elec2}, we report the time evolution of the electric energy as well as the ranks of the solutions. Consistent numerical results are observed. In Figure \ref{fig:bumpontail_invar2}, we plot the time evolution of the relative deviation of the total mass, momentum and energy for both methods. The non-conservative  method is able to conserve the invariants on the scale of the truncation threshold. Note that the exact momentum conservation is not observed as expected. Meanwhile, the conservative method is able to conserve the total mass and momentum up to the machine precision and has smaller conservation errors of total energy compared to the non-conservative method. Note that the conservation error of the total momentum increases after $t=15$ for the conservative method with the coarse mesh size $32\times64$ which is ascribed to the boundary error. In Figure \ref{fig:bumpontail_contour}, we report the contour plots of the solutions by the two methods. The numerical solutions are observed to qualitatively match each other. 
Last, we test the performance of projector $P_{\mathcal{N}}$ with different subspaces $\mathcal{N}$ as discussed in Remark \ref{rem:35}. In particular, we denote by $P_1$, $P_2$ and $P_3$ the orthogonal projectors with the weighted inner product \eqref{eq: inner_prod_d} onto the subspaces $\text{span}\{1_v\}$,  $\text{span}\{1_v,\bv\}$,  $\text{span}\{1_v,\bv,\bv^2\}$, respectively. In Figure \ref{fig:bumpontail_comp}, we plot the time evolution of the electric energy, ranks, total mass, total momentum, and total energy of the solutions from the conservative method with  $P_1$, $P_2$ and $P_3$. The three methods generate consistent results for the time evolution of the electric energy, and the ranks are comparable. By construction all three methods can conserve the total mass, and the methods with $P_2$ and $P_3$ can conserve the total momentum. Further, the method with $P_3$ does the best job in conserving the total energy as $P_3$ preserves the kinetic energy for truncation.

 \begin{figure}[h!]
	\centering
	\subfigure[]{\includegraphics[height=50mm]{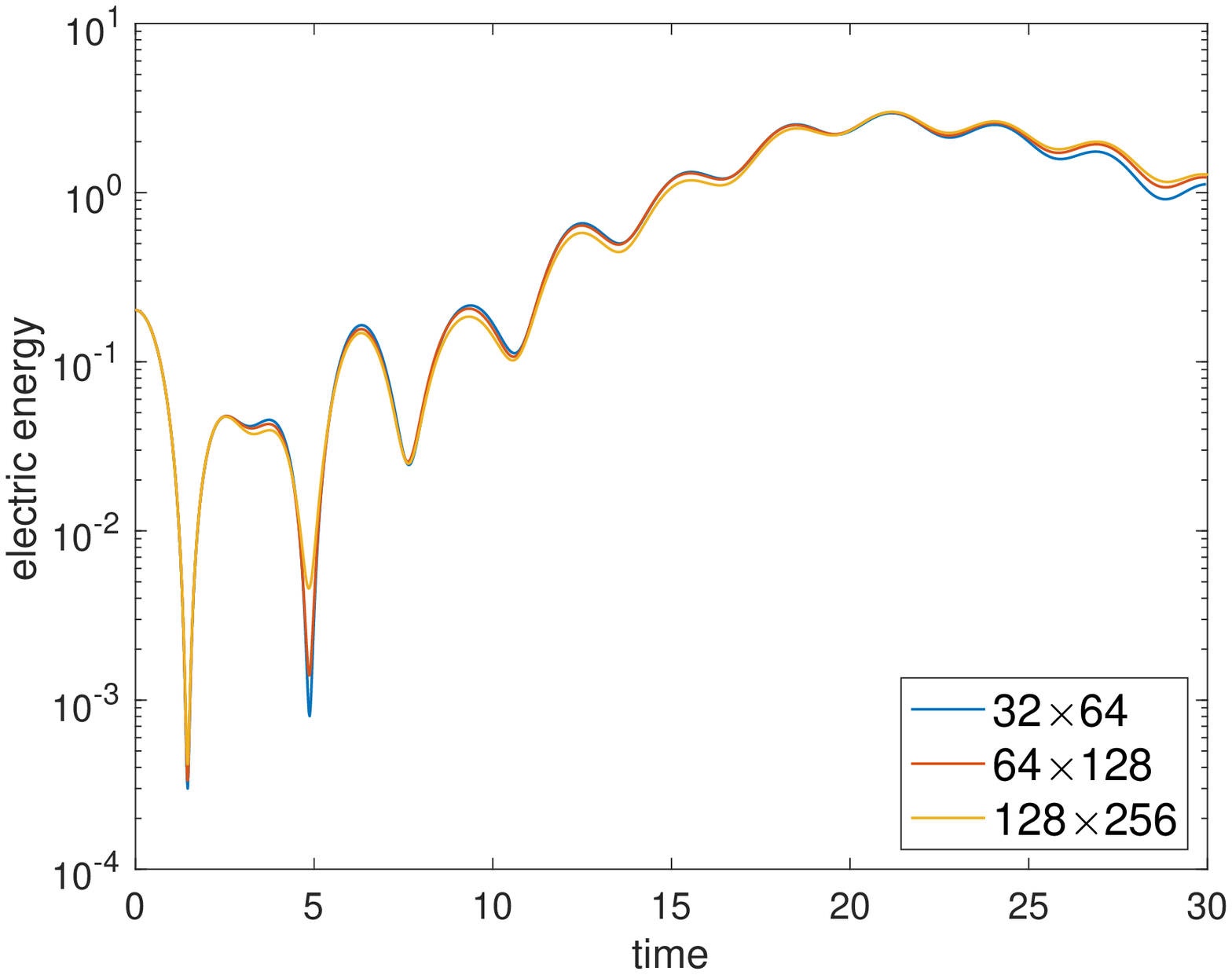}}
		\subfigure[]{\includegraphics[height=50mm]{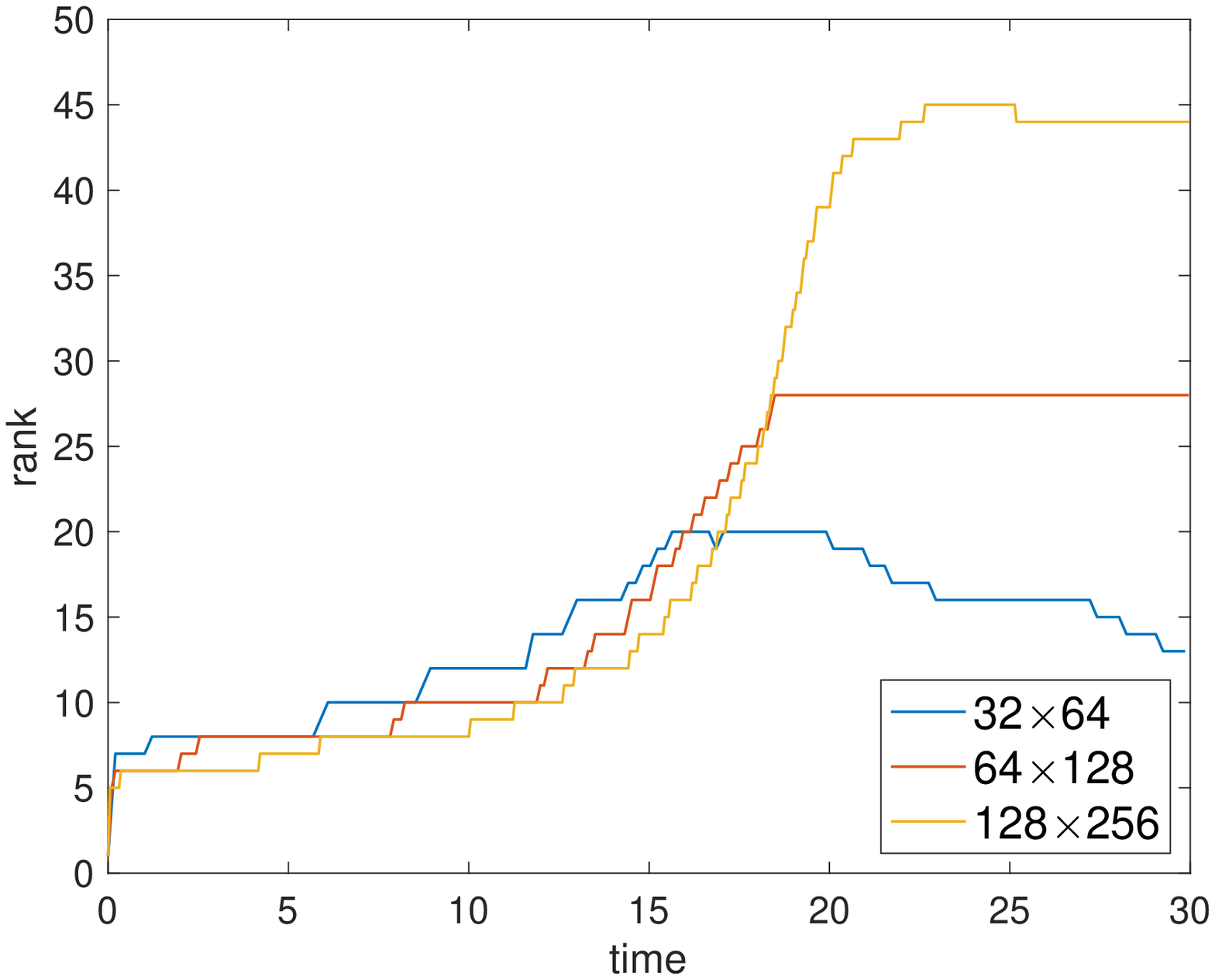}}
			\subfigure[]{\includegraphics[height=50mm]{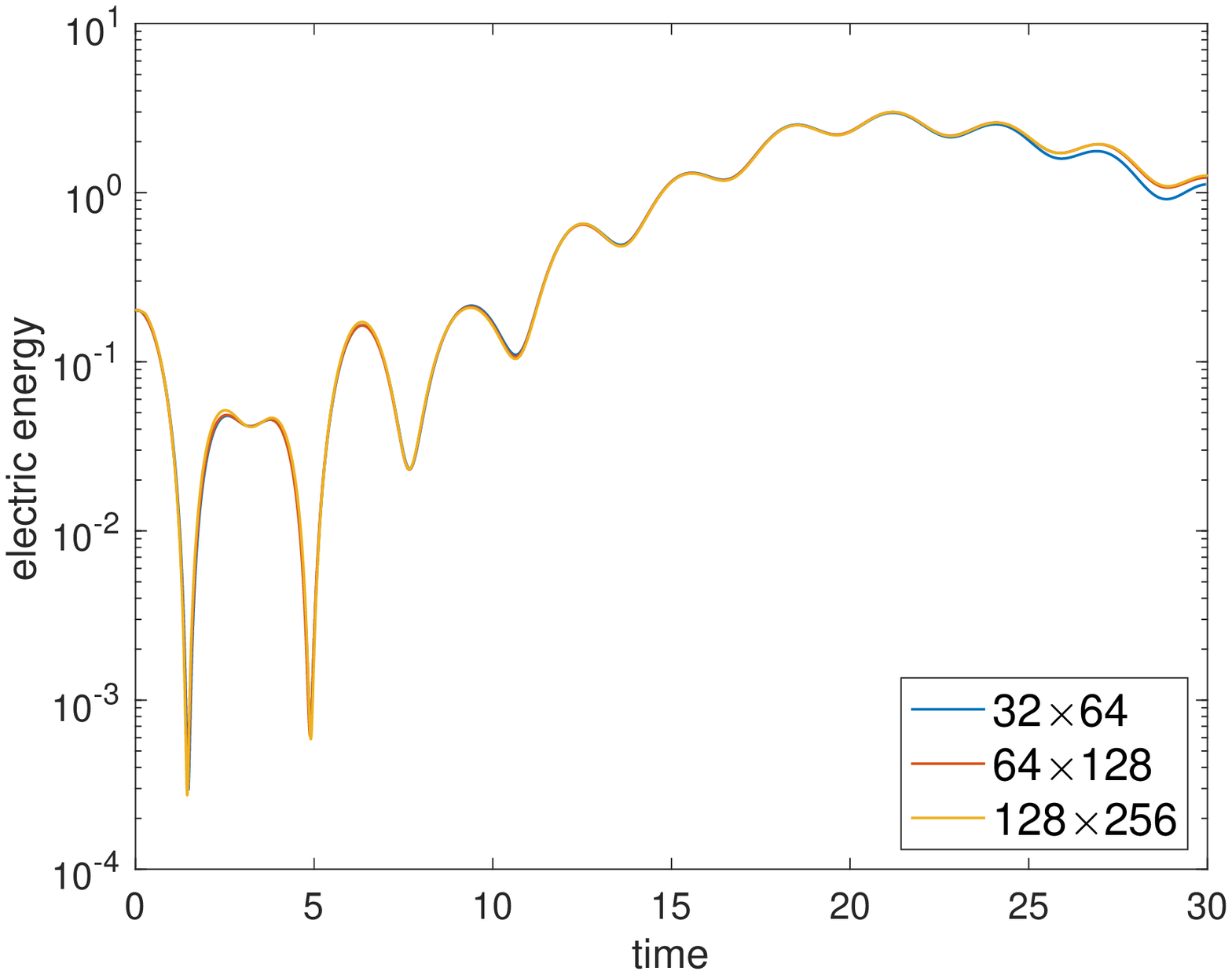}}
		\subfigure[]{\includegraphics[height=50mm]{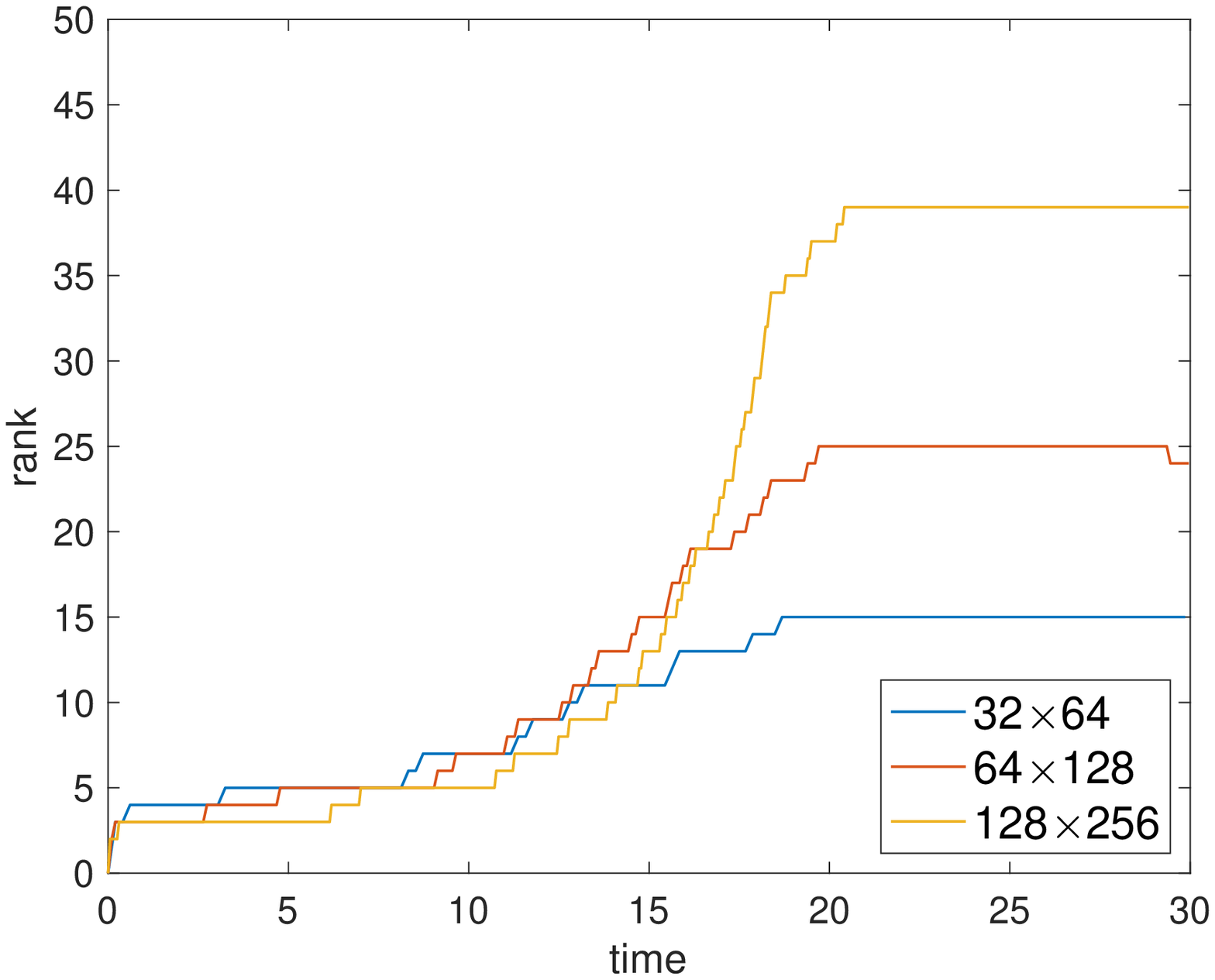}}
	\caption{Example \ref{ex:bumpontail}.  The time evolution of the electric energy (a, c) and the ranks of the numerical solutions (b, d). Conservative method (a, b) and non-conservative method (c, d).  $\varepsilon=10^{-4}$.}
	\label{fig:bumpontail_elec2}
\end{figure}

\begin{figure}[h!]
	\centering
	\subfigure[]{\includegraphics[height=40mm]{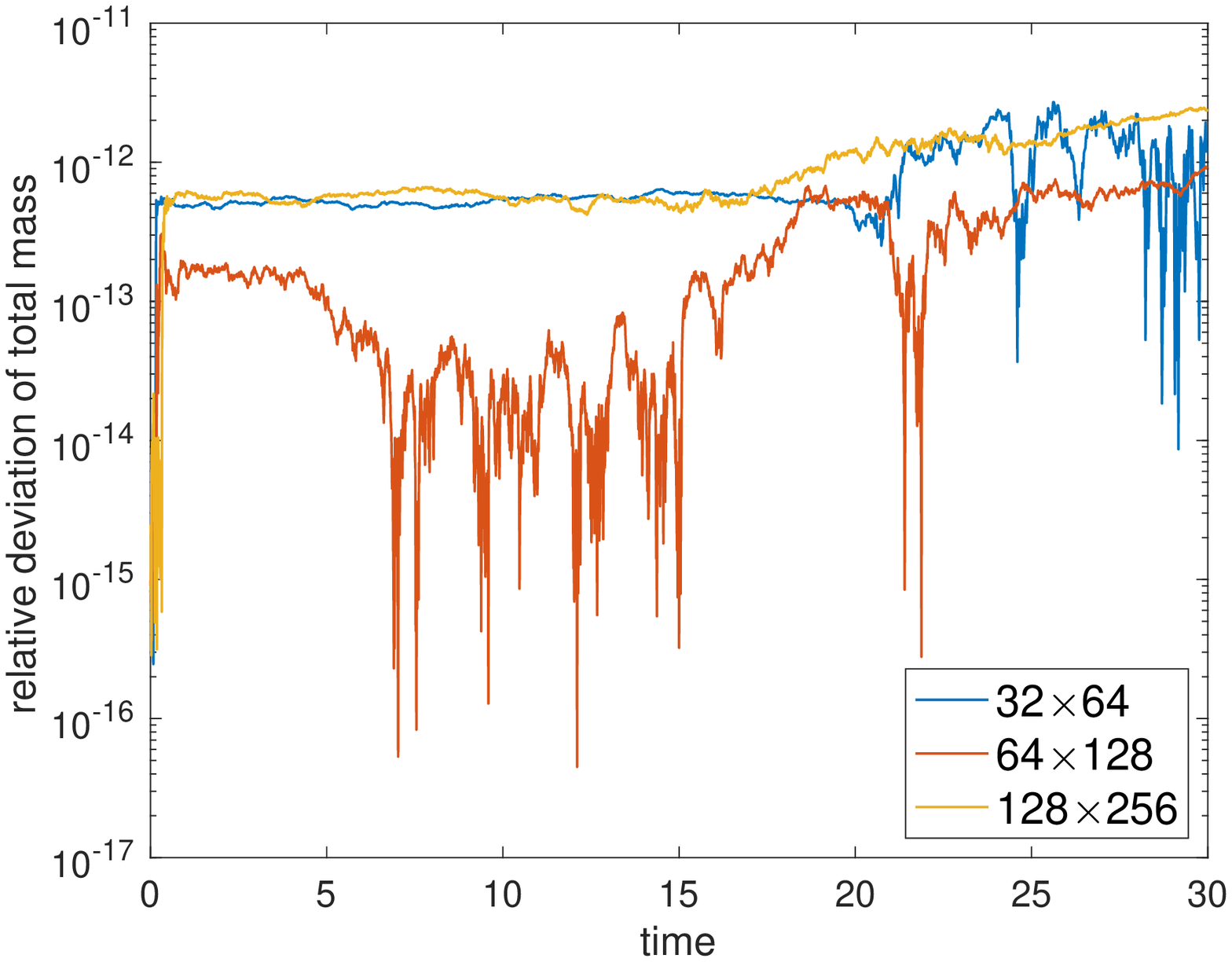}}
		\subfigure[]{\includegraphics[height=40mm]{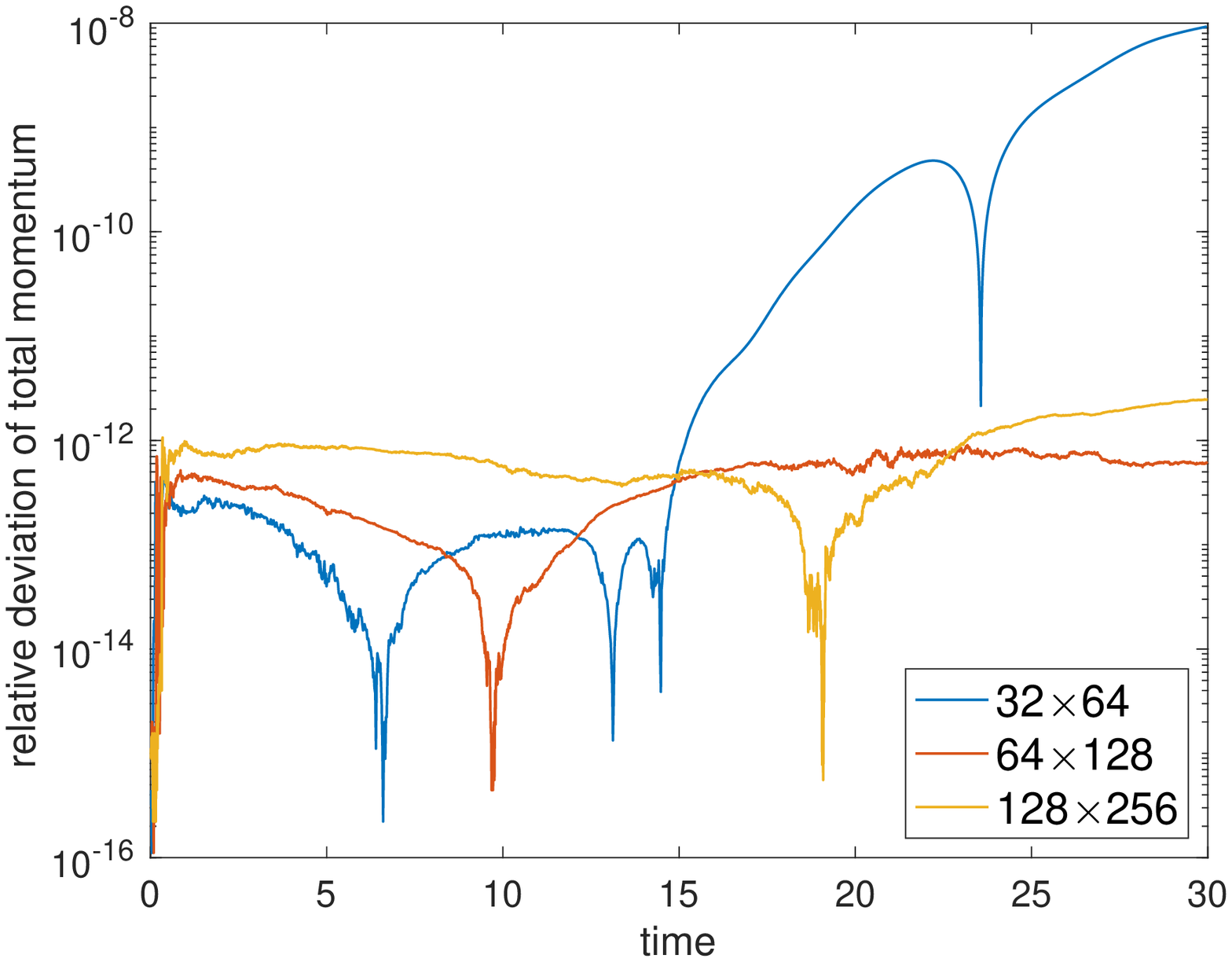}}
		\subfigure[]{\includegraphics[height=40mm]{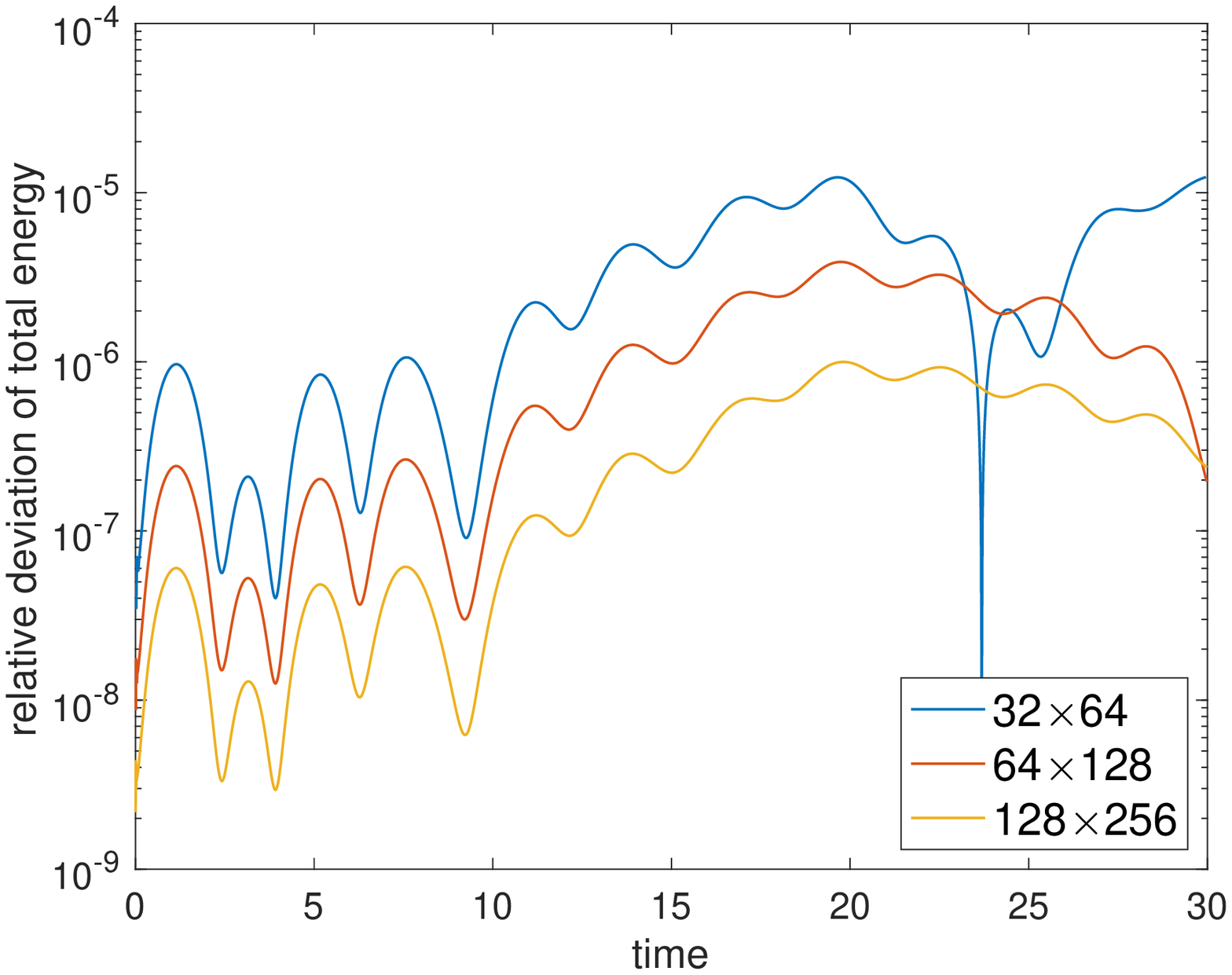}}
			\subfigure[]{\includegraphics[height=40mm]{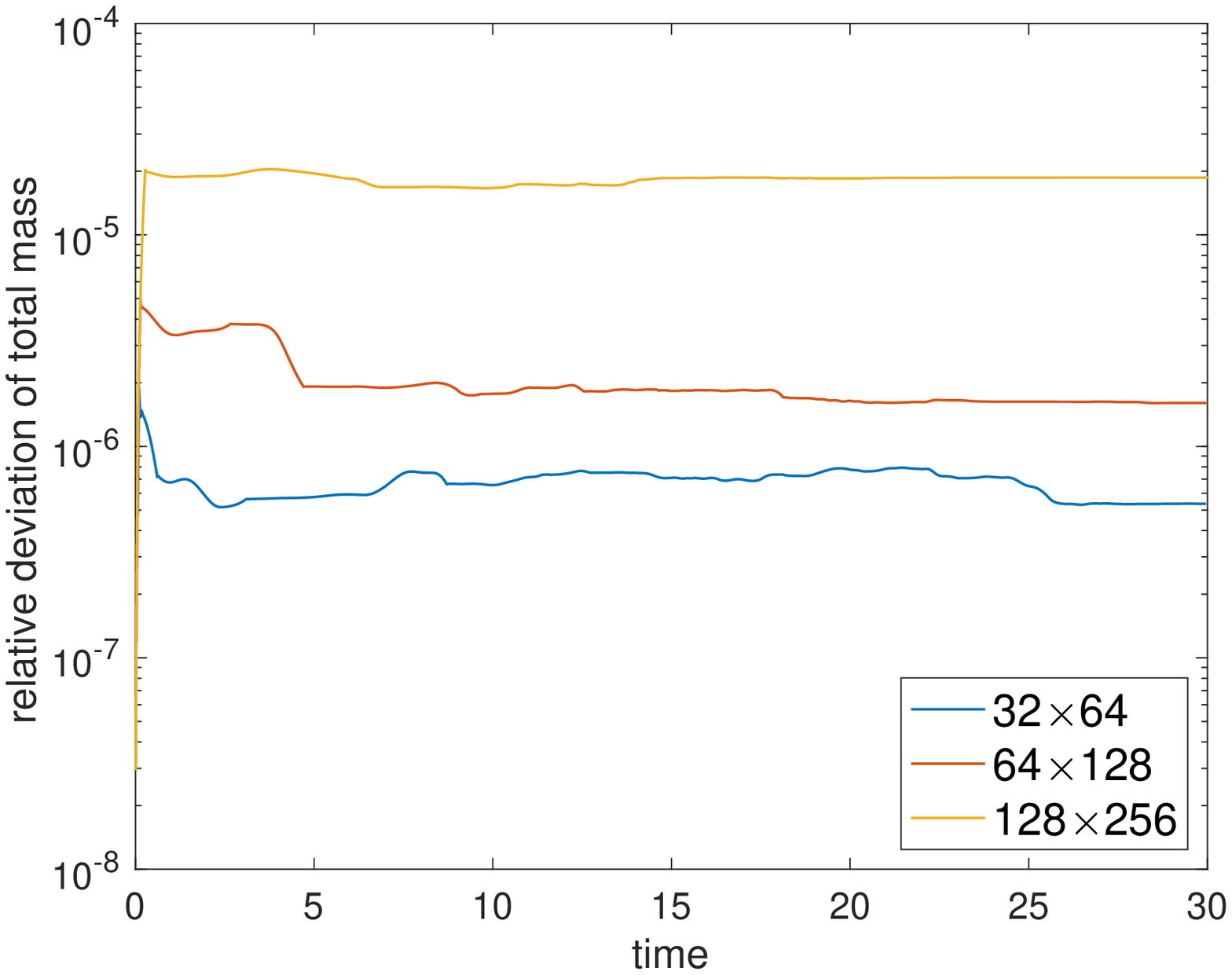}}
		\subfigure[]{\includegraphics[height=40mm]{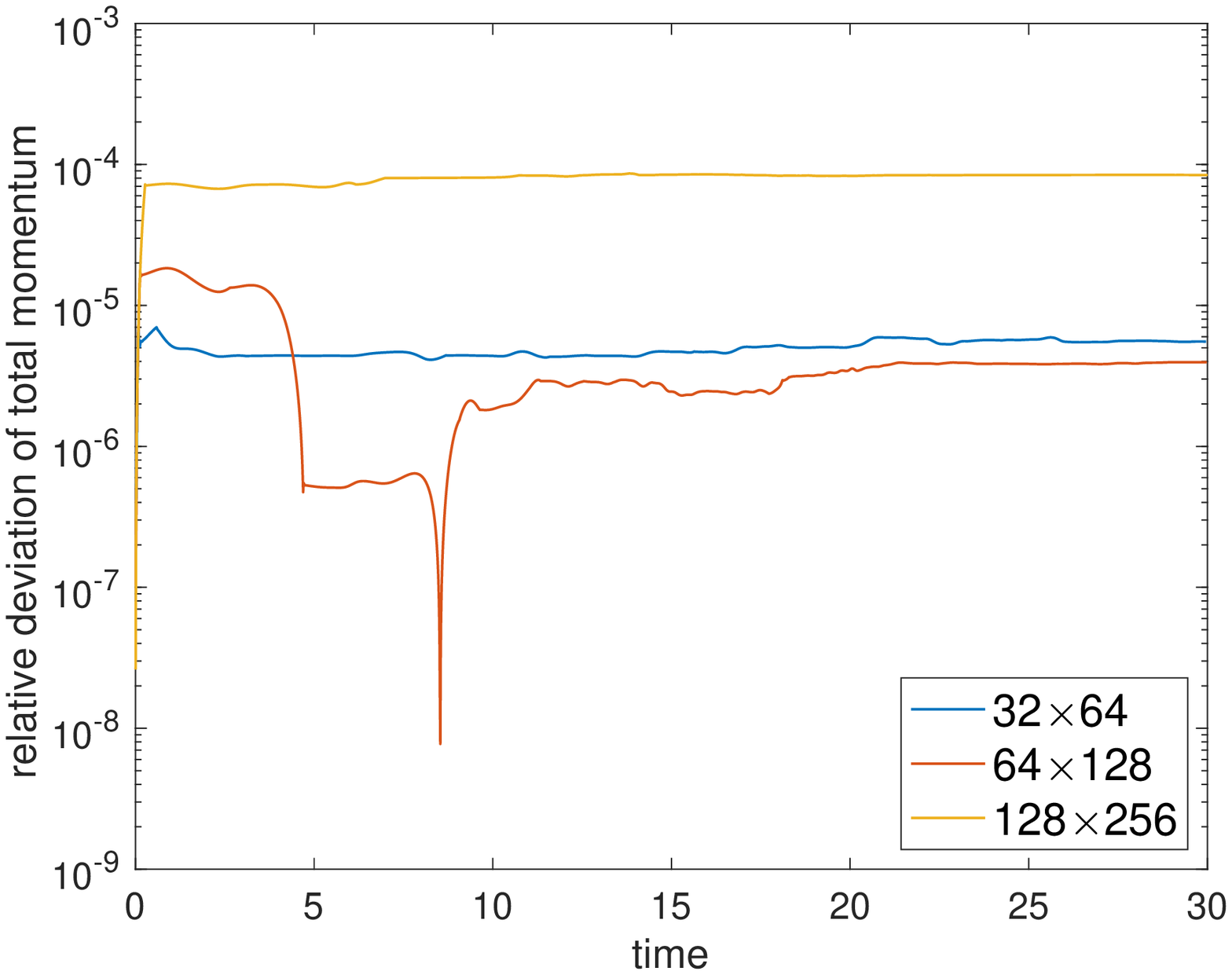}}
		\subfigure[]{\includegraphics[height=40mm]{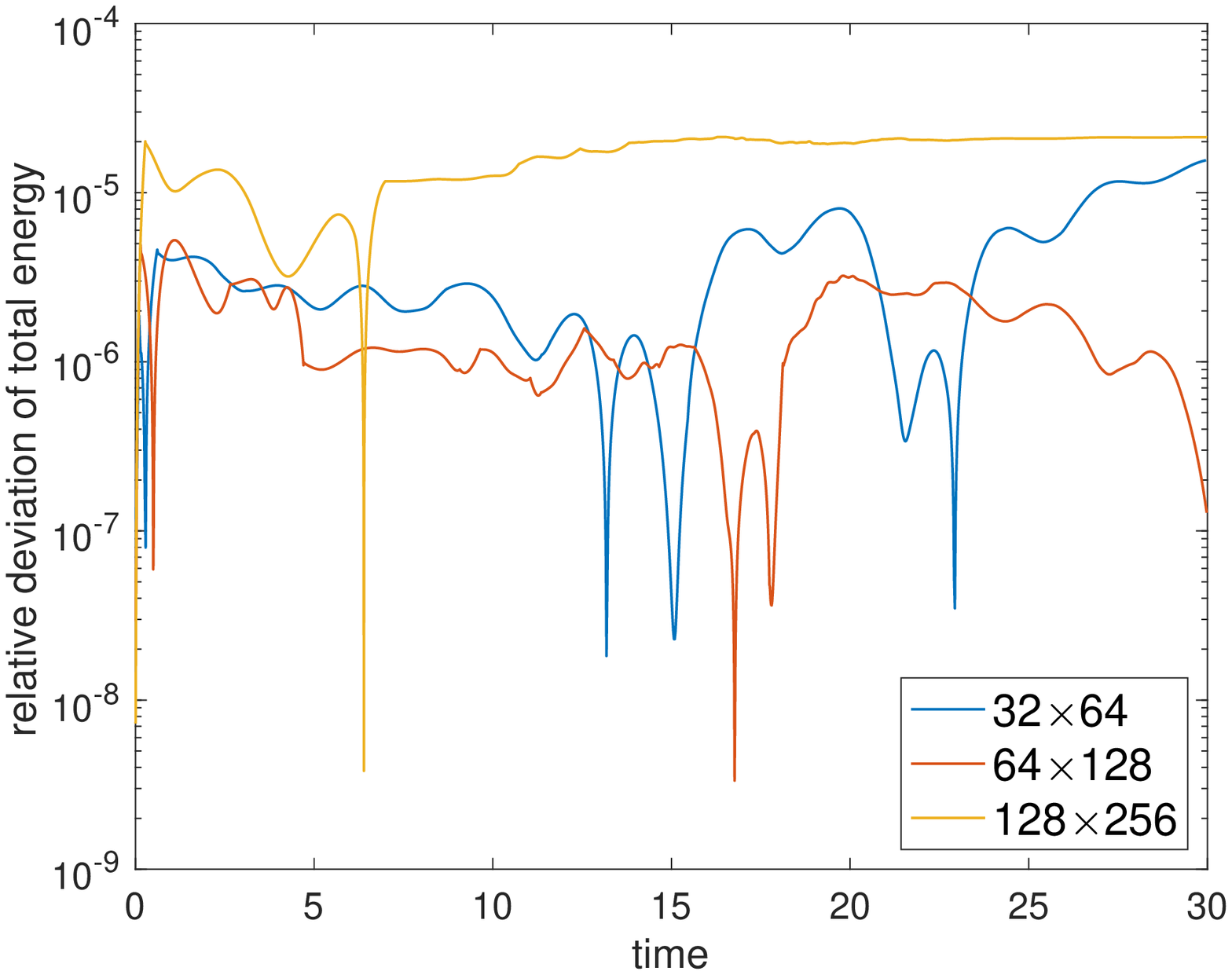}}
	\caption{Example \ref{ex:bumpontail}.  The time evolution of relative deviation of total mass (a, d), relative deviation of total momentum (b, e), and relative deviation of total energy (c, f). Conservative method (a, b, c) and non-conservative method (d, e, f).  $\varepsilon=10^{-4}$.}
	\label{fig:bumpontail_invar2}
\end{figure}

 \begin{figure}[h!]
	\centering
		\subfigure[]{\includegraphics[height=50mm]{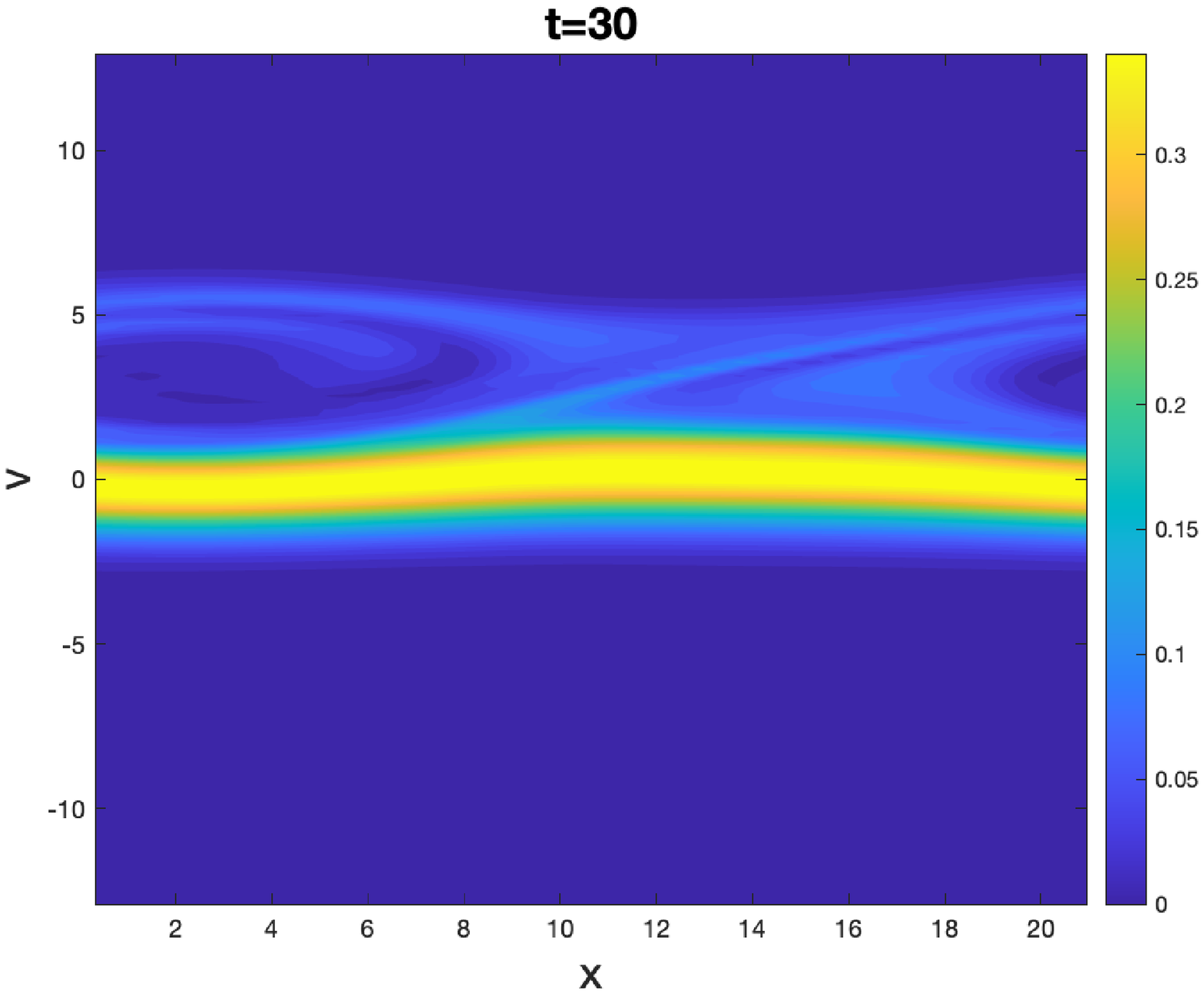}}
		\subfigure[]{\includegraphics[height=50mm]{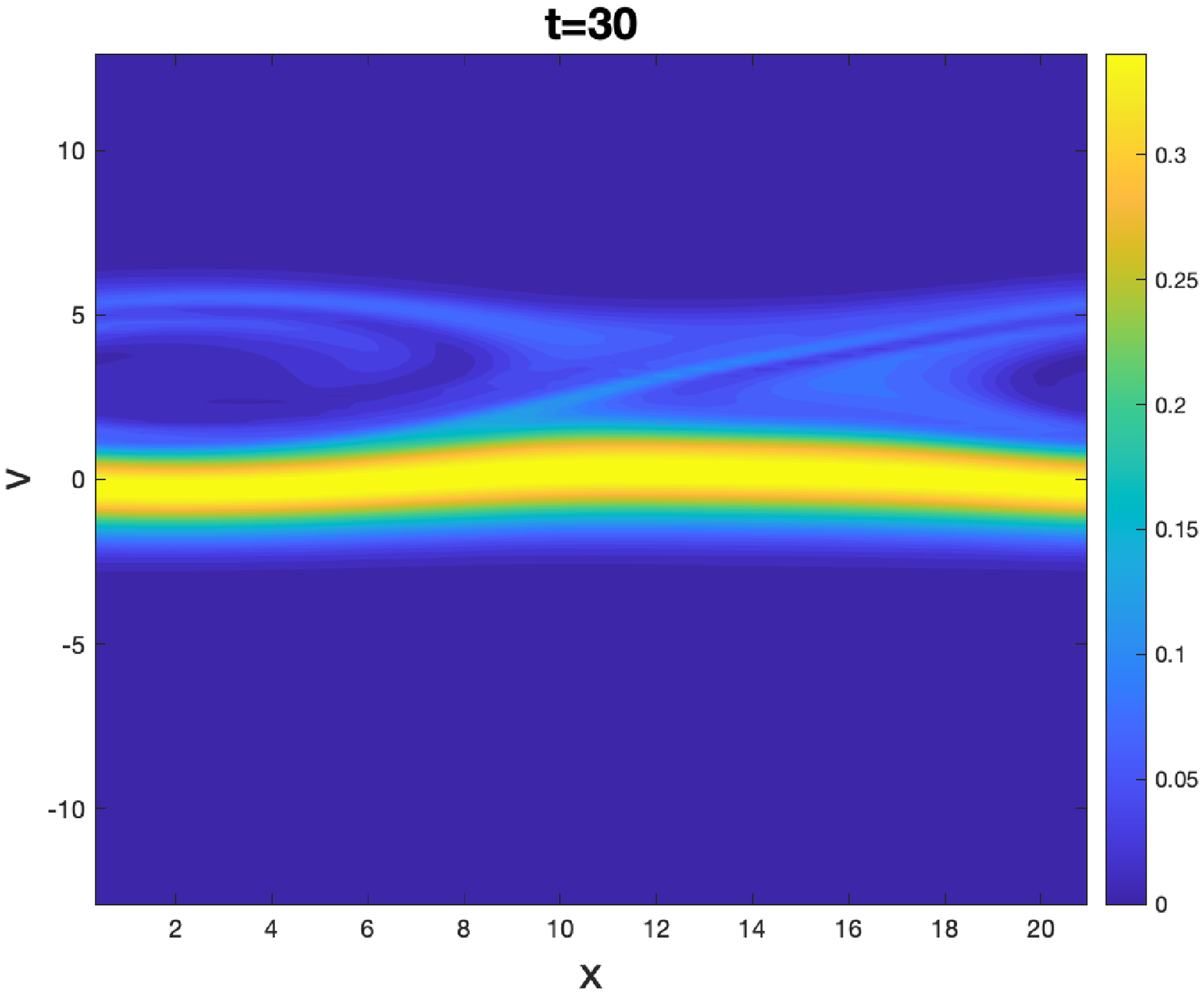}}
	\subfigure[]{\includegraphics[height=50mm]{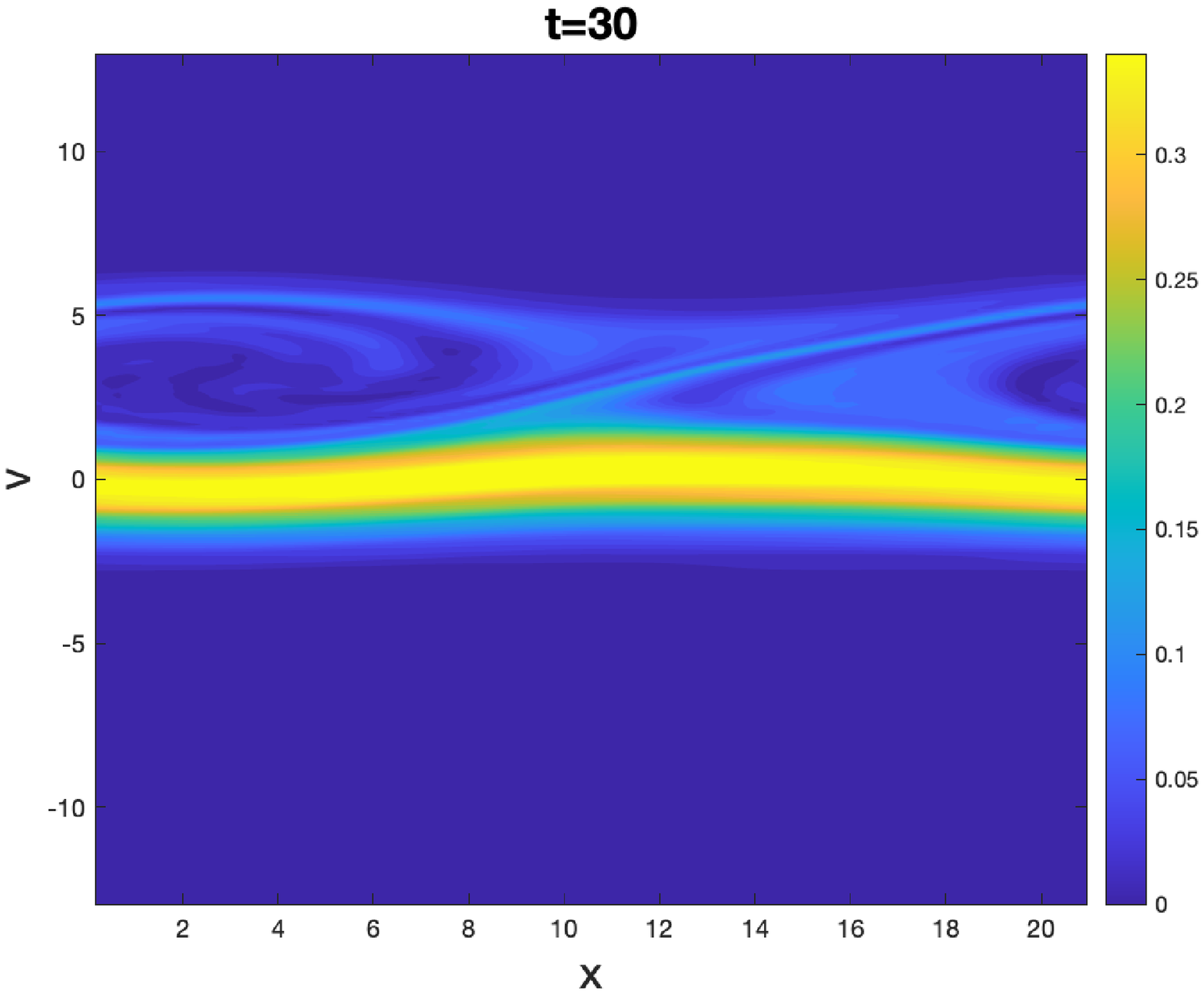}}
		\subfigure[]{\includegraphics[height=50mm]{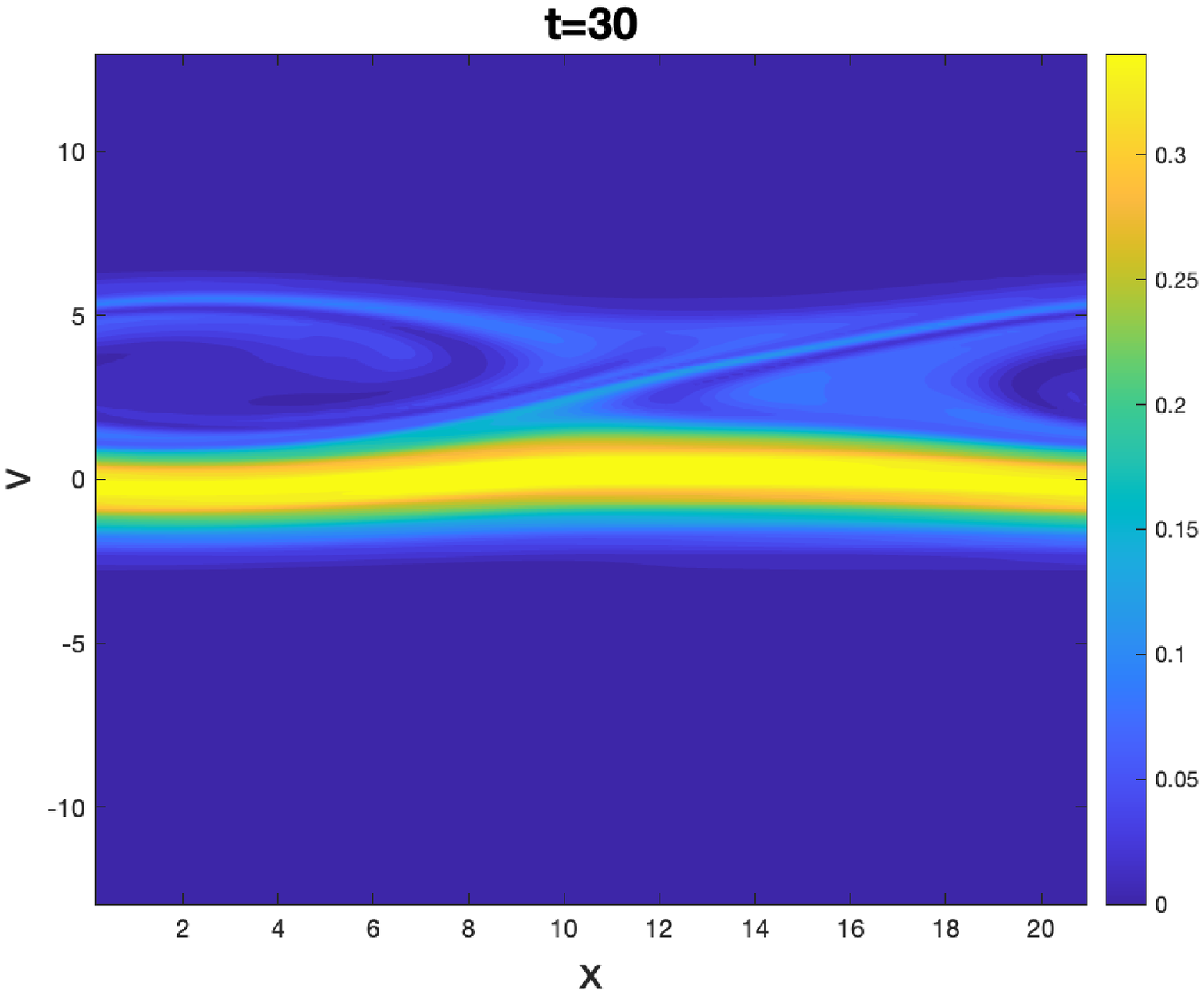}}
			\caption{Example \ref{ex:bumpontail}. Contour plots of the solutions at $t=30$. Conservative method with mesh $N_x\times N_v=64\times128$ (a). Non-conservative method with mesh $N_x\times N_v=64\times128$ (b). Conservative method with mesh $N_x\times N_v=128\times256$ (c). Non-conservative method with mesh $N_x\times N_v=128\times256$ (d). $\varepsilon=10^{-4}$.}

	\label{fig:bumpontail_contour}
\end{figure}

\begin{figure}[h!]
	\centering
	\subfigure[]{\includegraphics[height=40mm]{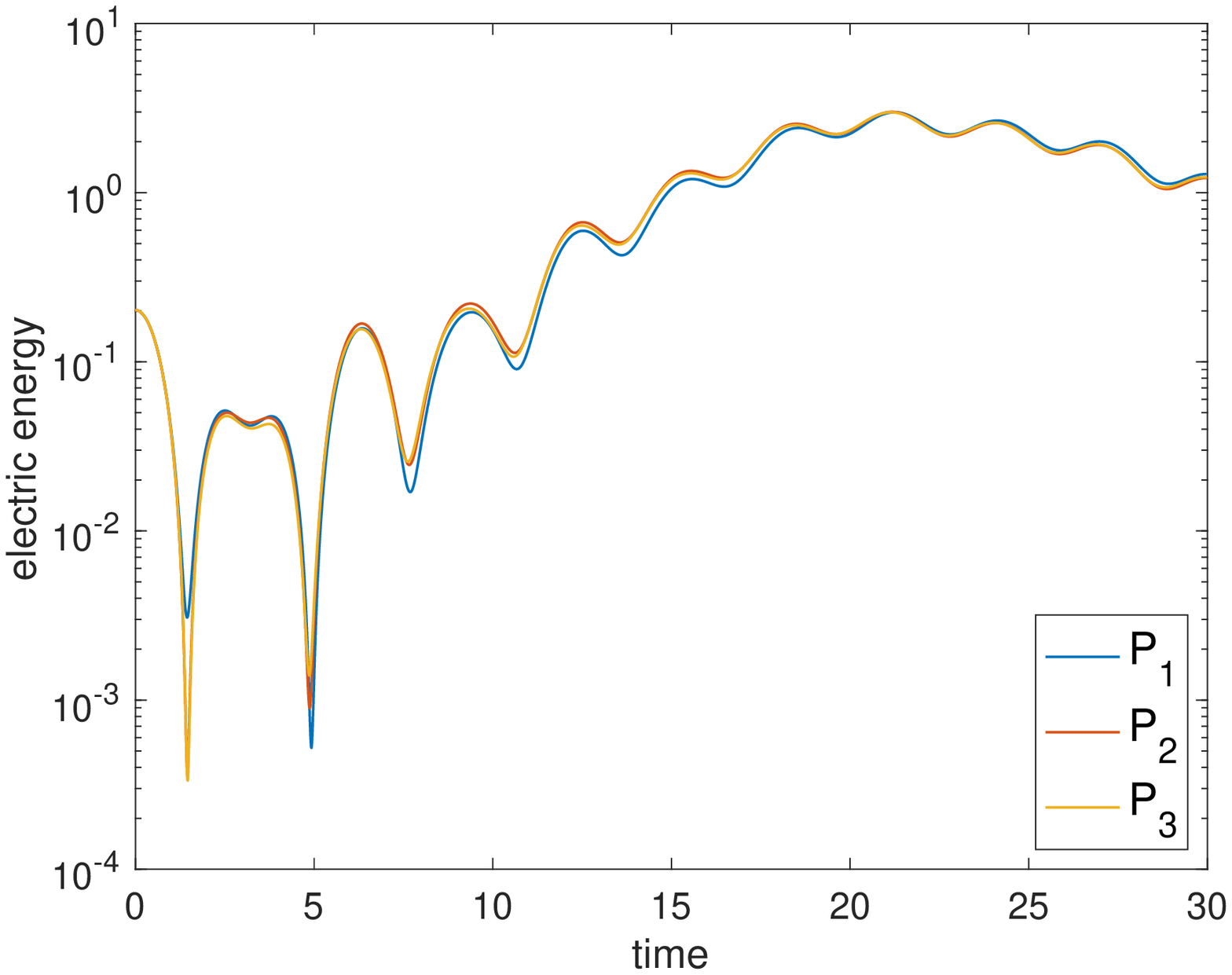}}
		\subfigure[]{\includegraphics[height=40mm]{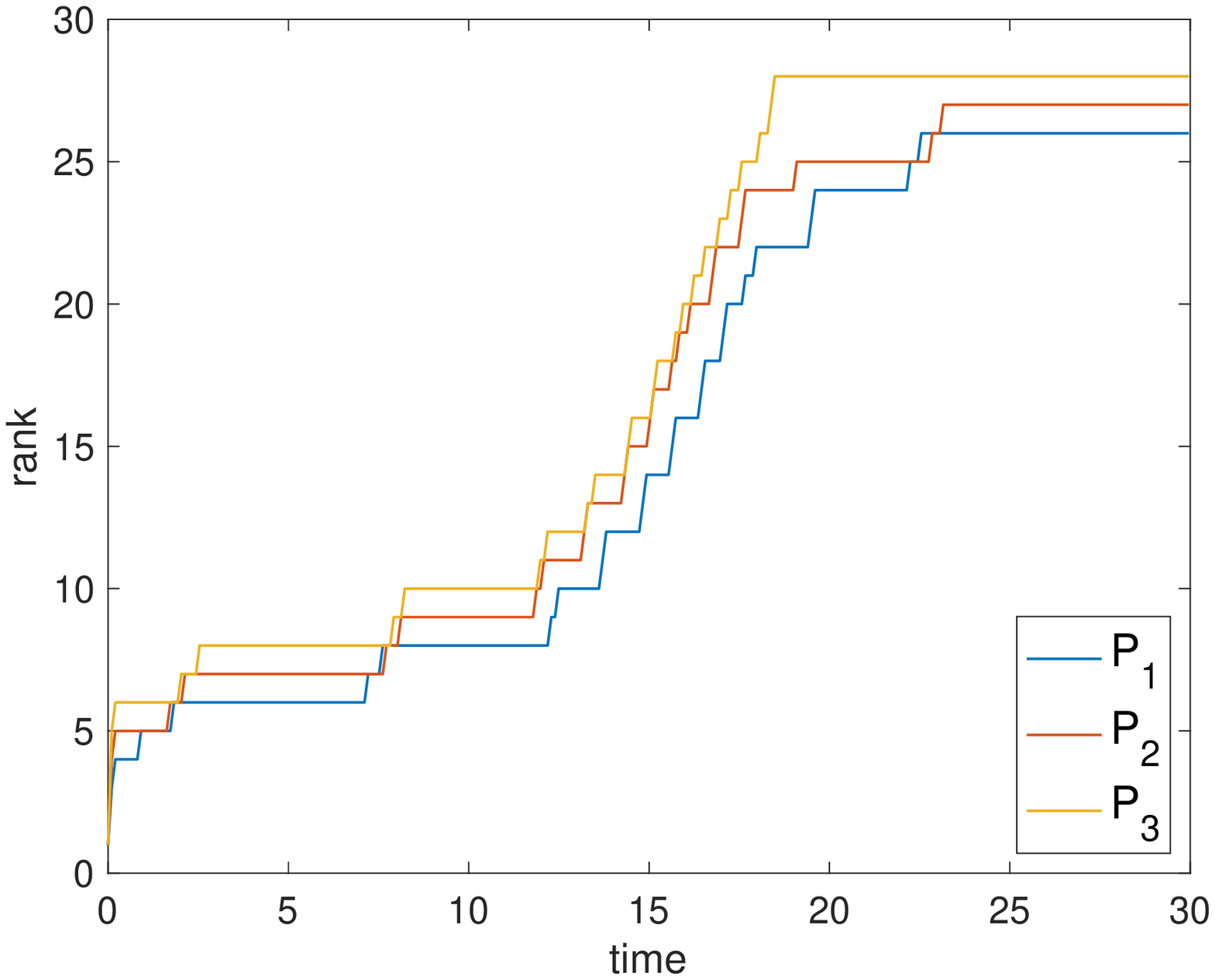}}
		\subfigure[]{\includegraphics[height=40mm]{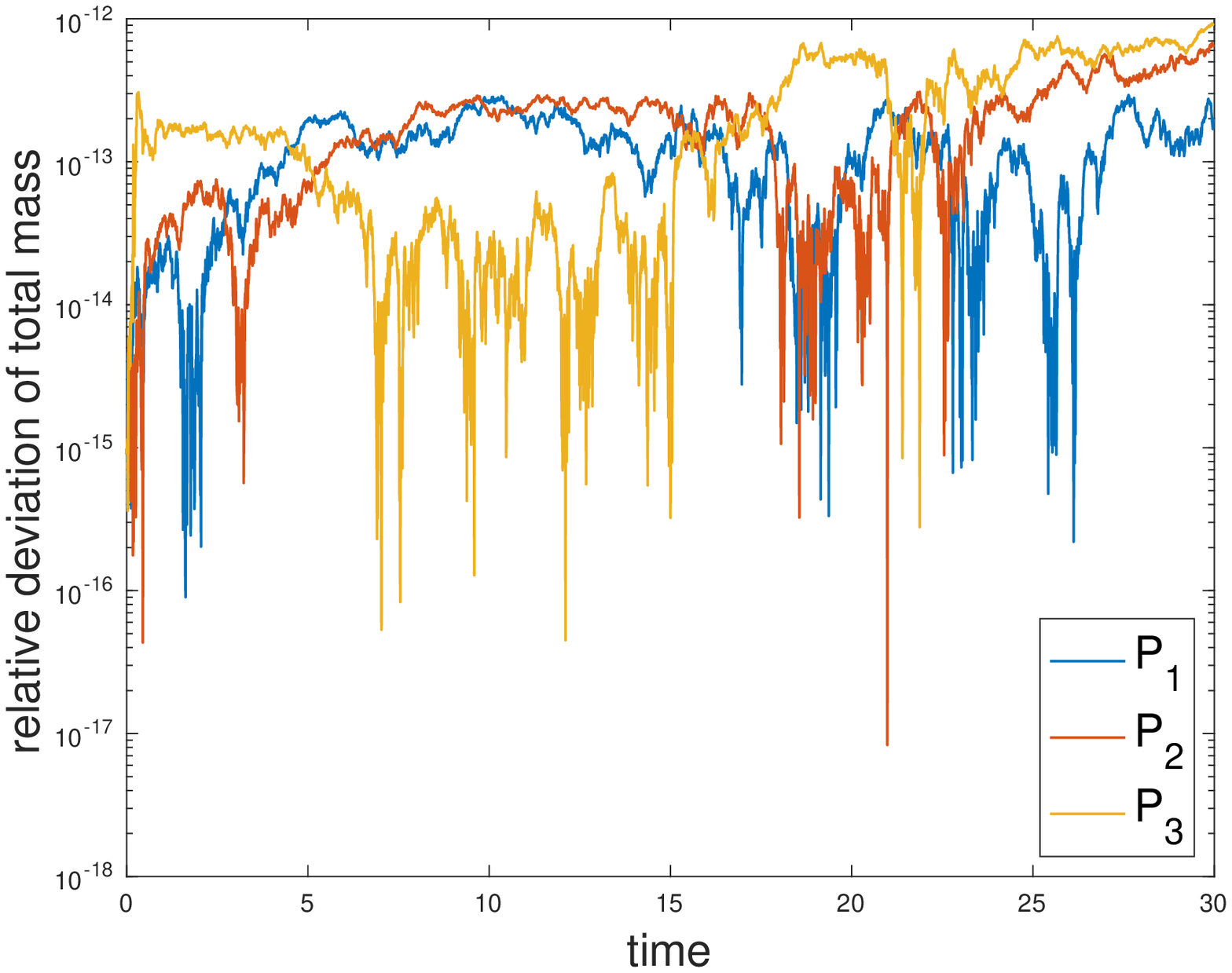}}
			\subfigure[]{\includegraphics[height=40mm]{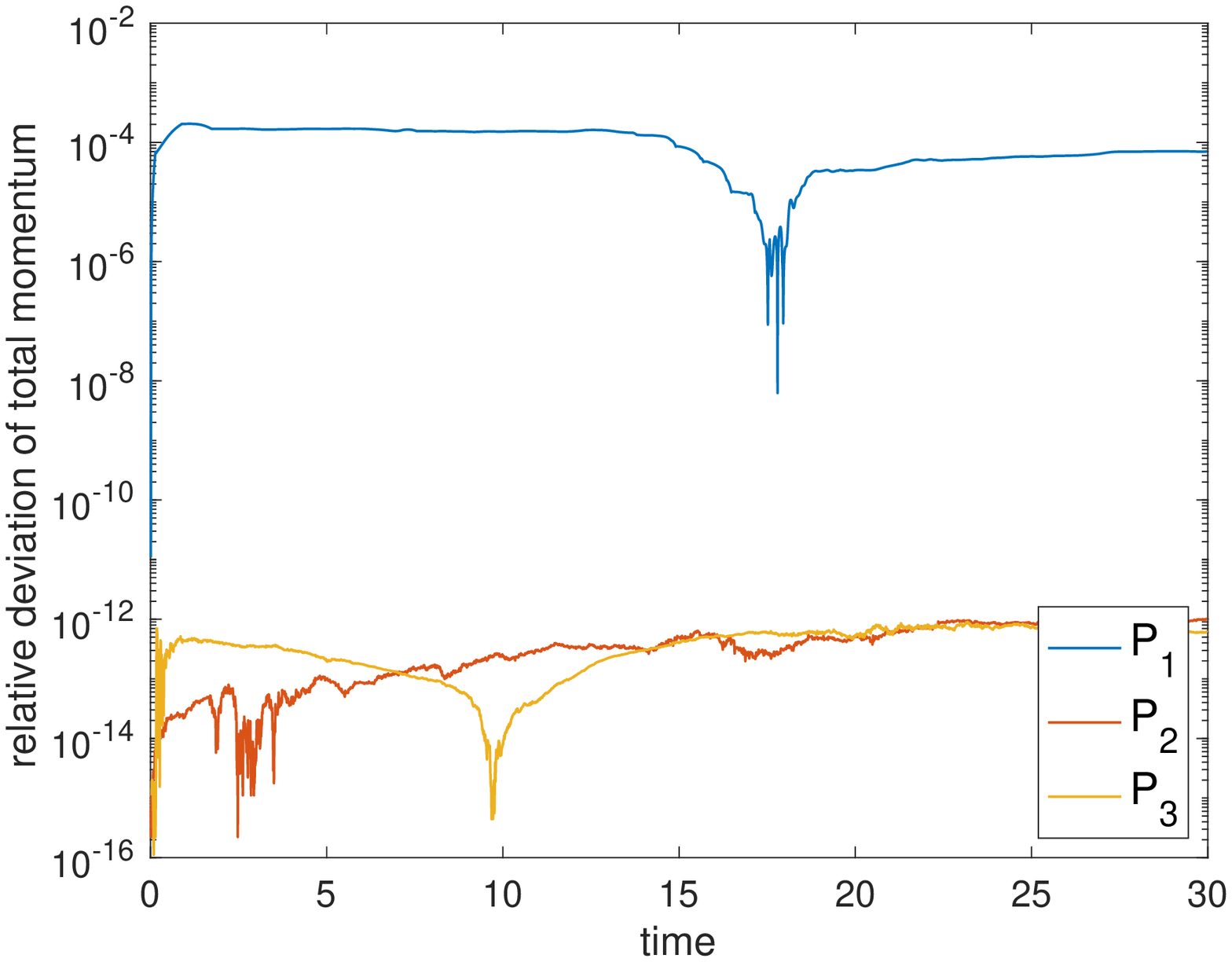}}
		\subfigure[]{\includegraphics[height=40mm]{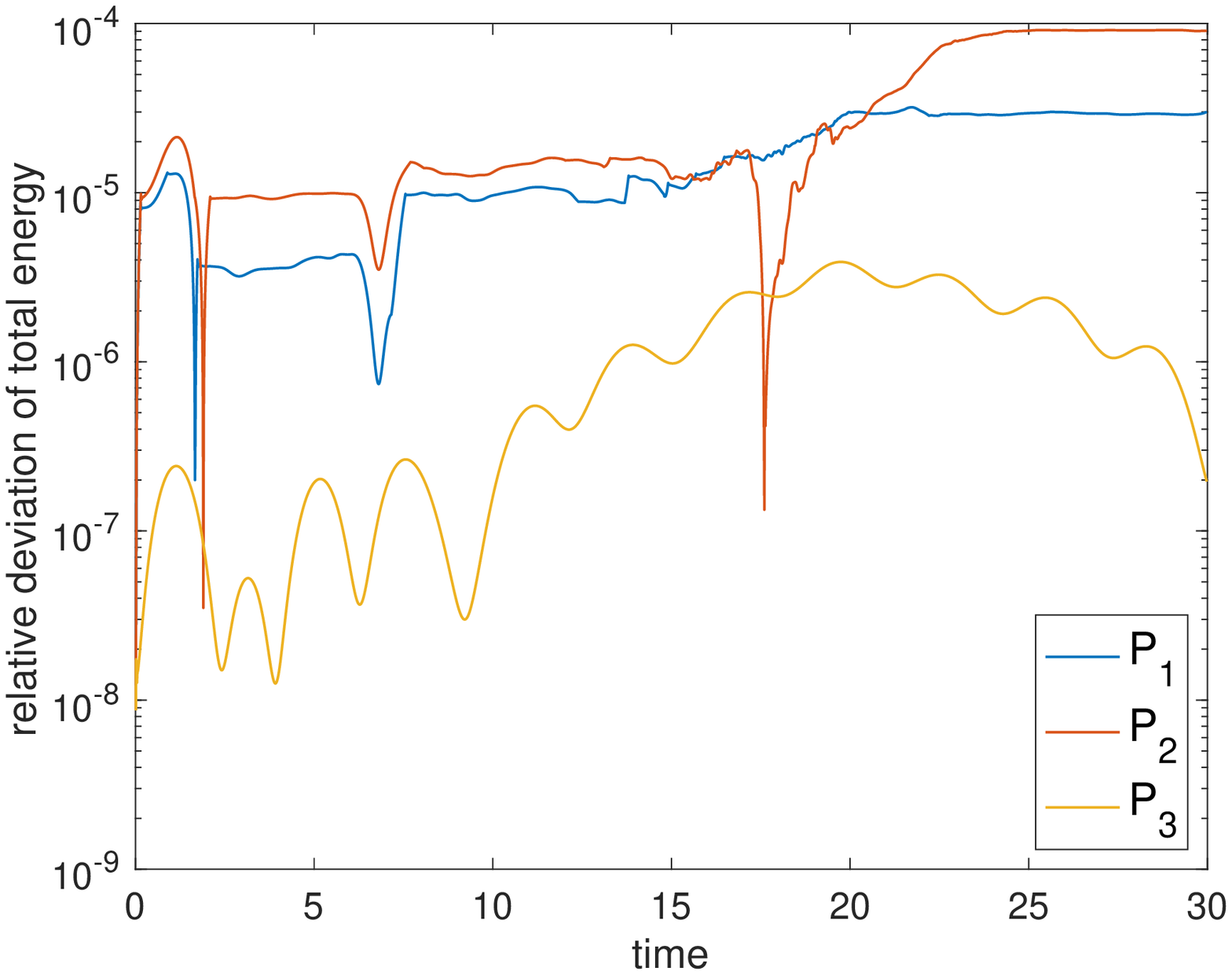}}
	\caption{Example \ref{ex:bumpontail}. Comparison of projections $P_1,\, P_2,\,P_3$ of the proposed conservative low rank method. The time evolution of  electric energy (a), ranks of the numerical solutions (b), relative deviation of total mass (c), relative deviation of total momentum (e), and relative deviation of total energy (f). $N_x\times N_v=64\times128$. $\varepsilon=10^{-4}$.}
	\label{fig:bumpontail_comp}
\end{figure}

 \subsubsection{2D2V Vlasov-Poisson system} 
 \begin{exa} \label{ex:weak2d} (Weak Landau damping.) We consider the 2D2V weak Landau damping, the dynamics of which is similar to the 1D1V case. The initial  
condition is 
\begin{equation}
	\label{eq:weak}
	f(\bx,\bv,t=0) =\frac{1}{(2 \pi)^{d / 2}} \left(1+\alpha \sum_{m=1}^{d} \cos \left(k x_{m}\right)\right)\exp\left(-\frac{|\bv|^2}{2}\right),
\end{equation}
where $d=2$, $\alpha=0.01$, and $k=0.5$. 
 \end{exa}
 
We set the computation domain as $[0,L_x]^2\times[-L_v,L_v]^2$, where $L_x=\frac{2\pi}{k}$ and $L_v=6$, and the truncation threshold $\varepsilon=10^{-5}$. We simulate the problem with both conservative and non-conservative methods, and the solutions are represented in the fourth order HT format. In Figures \ref{fig:weak2d_elec_con}-\ref{fig:weak2d_elec_non}, we report the time evolution of the electric energy, hierarchical ranks of the numerical solution,  relative deviation of total mass and energy together with absolute total momentum $J_1$ and $J_2$. It is observed that both methods are able to predict the damping rate of the electric energy. Furthermore, the conservative method is able to conserve the total mass and momentum $J_1$ and $J_2$ up to the machine precision and enjoys better total energy conservation compared to the non-conservative counterpart. On the other hand, the hierarchical ranks of the solution from the conservative method, especially $r_1$ and $r_2$, are larger than that from the non-conservative method. This is because ${\bf f}_1$ is constructed without compression to guarantee the local conservation (in fact ${\bf f}_1$ is truncated with threshold $10^{-15}$ in the simulation), and then the solution tensor is not compressed in the $\bx$ direction.

\begin{figure}[h!]
	\centering
	\subfigure[]{\includegraphics[height=40mm]{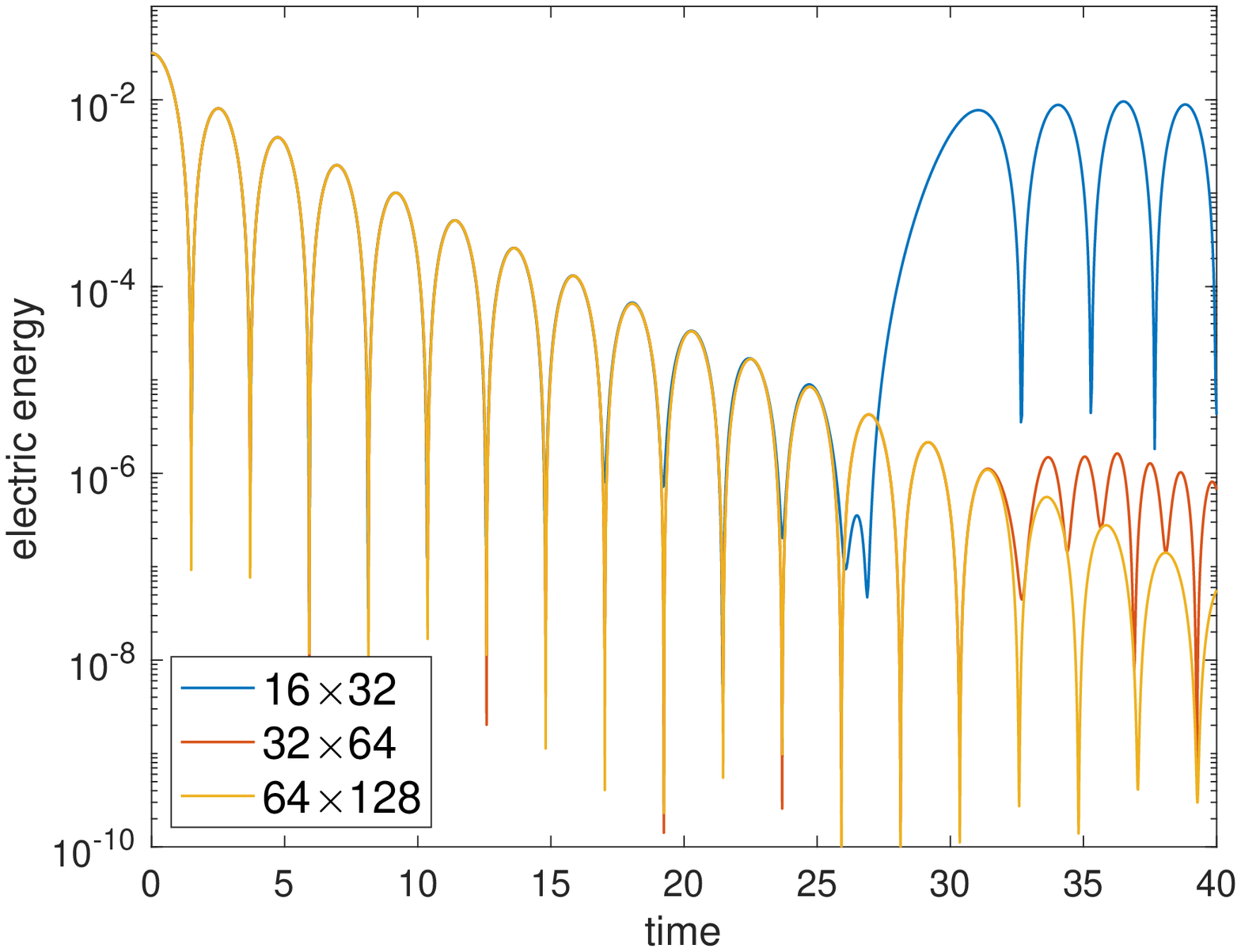}}
		\subfigure[]{\includegraphics[height=40mm]{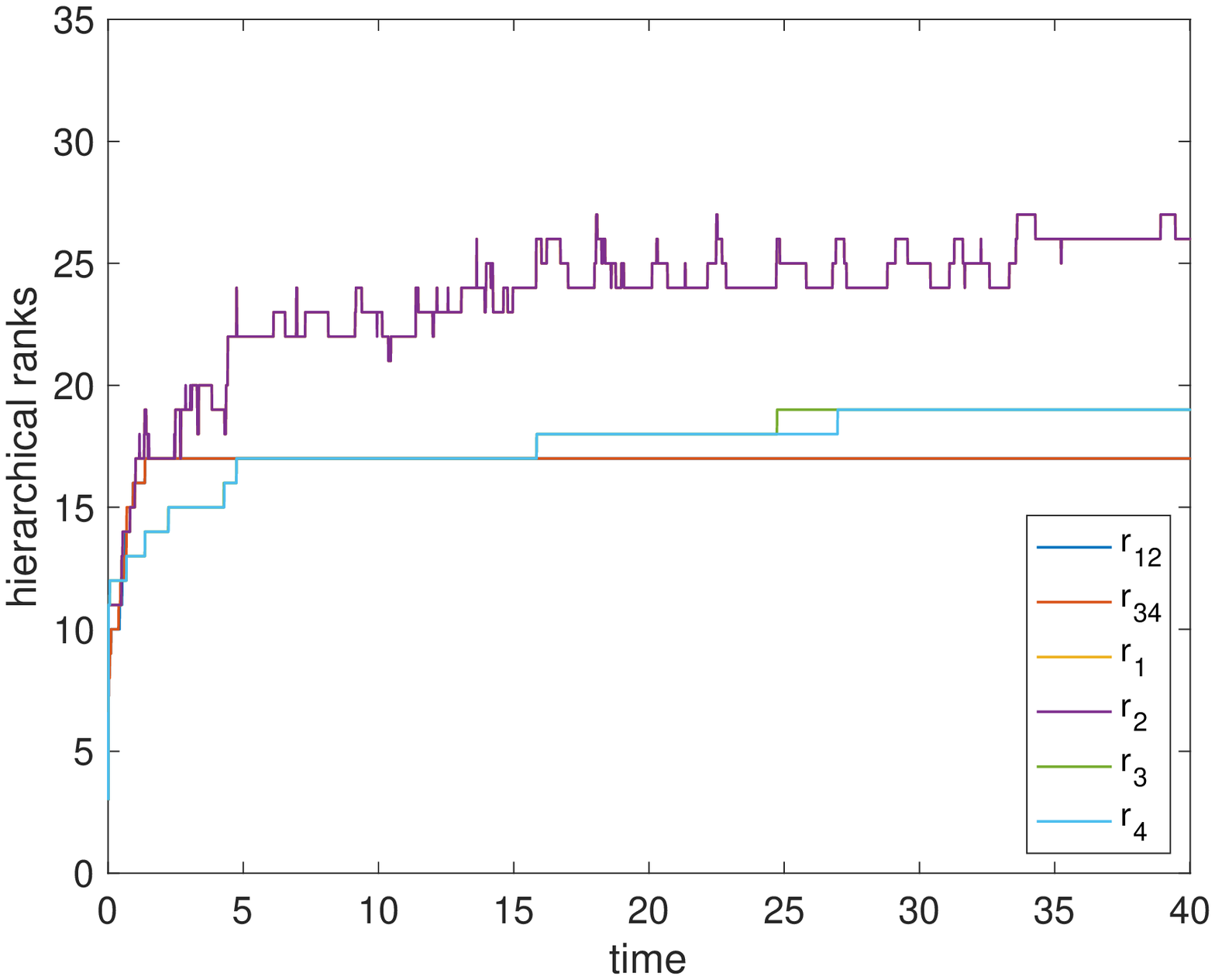}}
		\subfigure[]{\includegraphics[height=40mm]{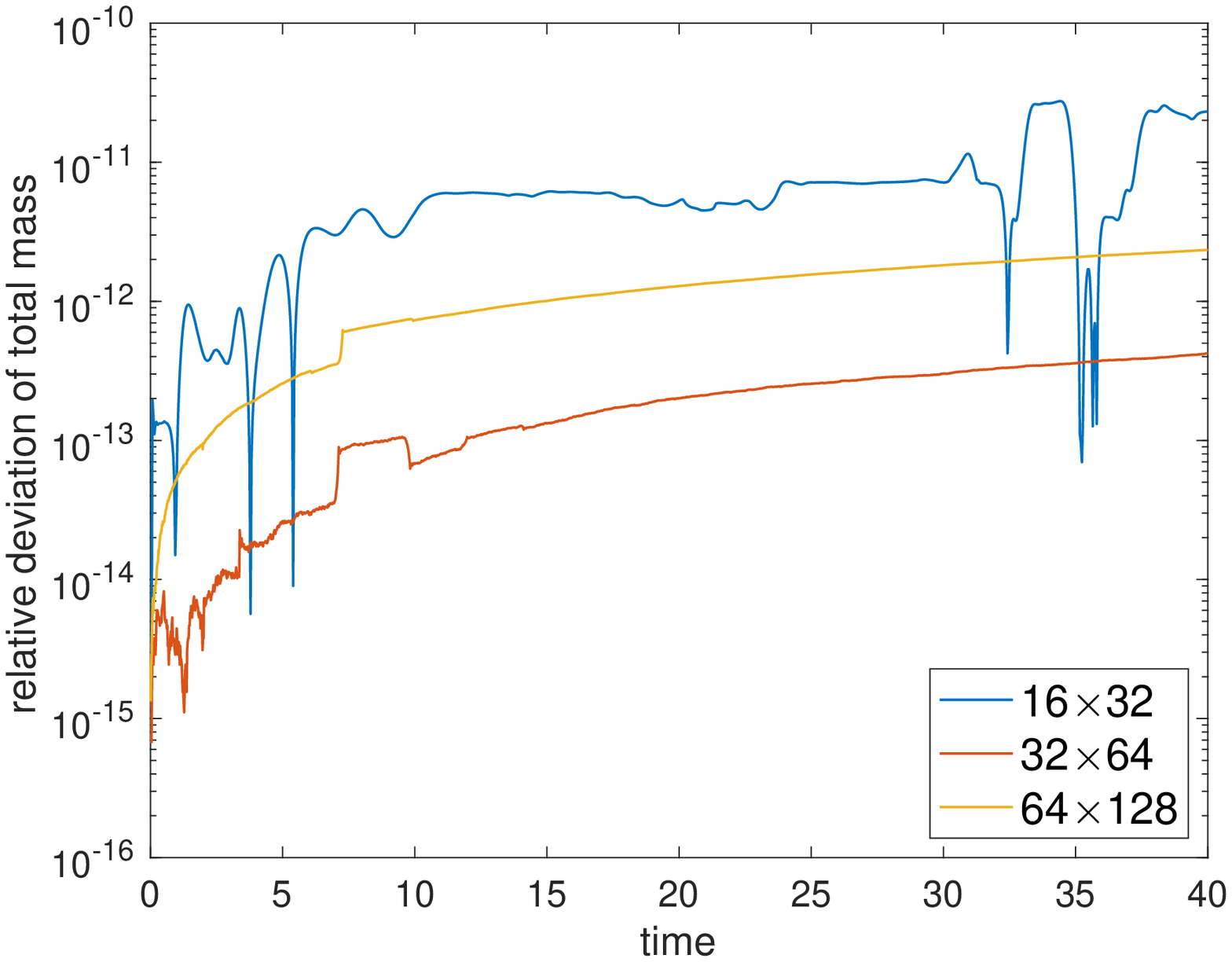}}
		\subfigure[]{\includegraphics[height=40mm]{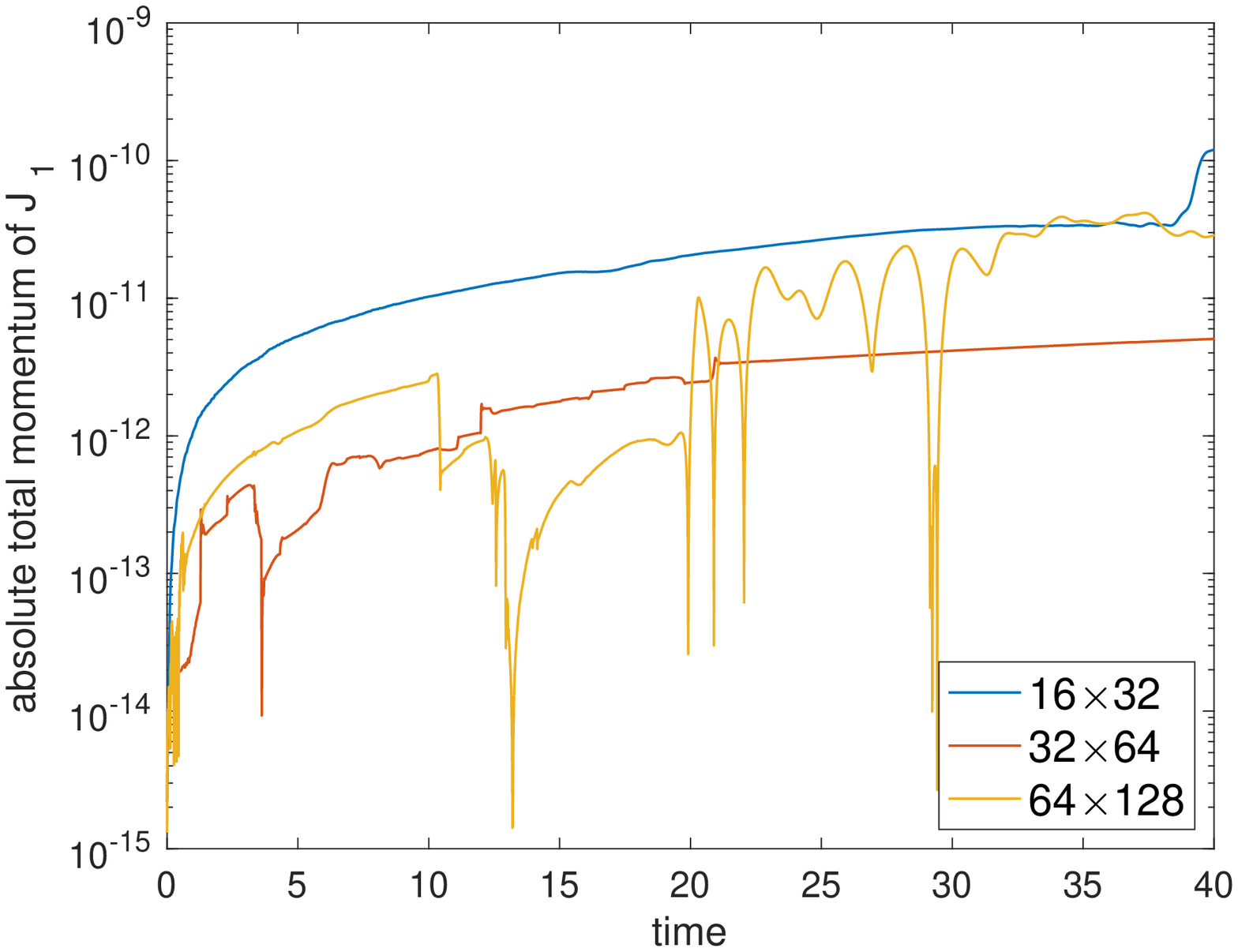}}
			\subfigure[]{\includegraphics[height=40mm]{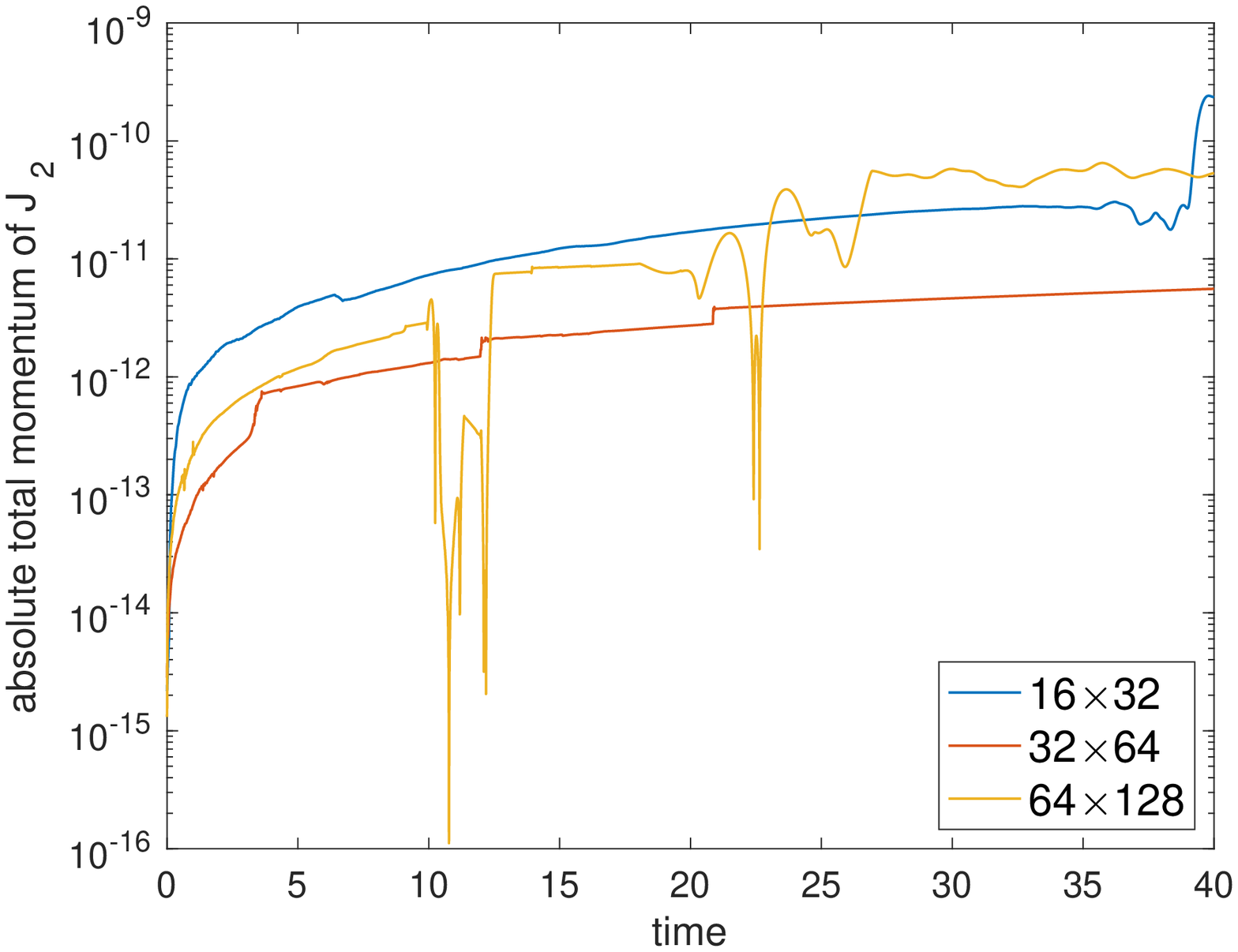}}
		\subfigure[]{\includegraphics[height=40mm]{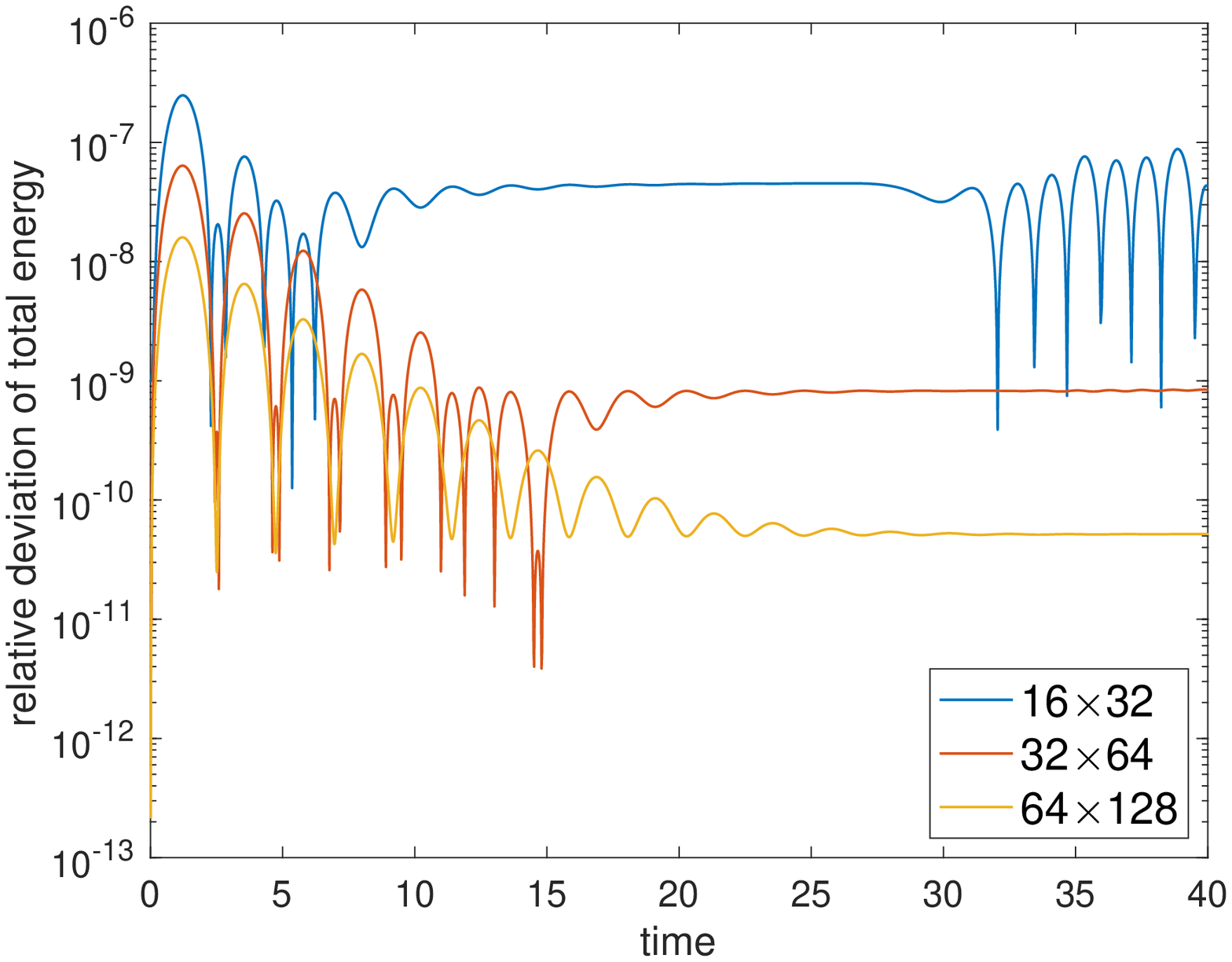}}
	\caption{Example \ref{ex:weak2d}. Conservative low rank method. The time evolution of  electric energy (a), hierarchical ranks of the xnumerical solution of mesh size $N_x\times N_v=64\times128$ (b), relative deviation of total mass (c), absolute total momentum $J_1$ (d), absolute total momentum $J_2$ (e), and relative deviation of total energy (f). $\varepsilon=10^{-5}$. In (b), $r_{12}$ and $r_{34}$ are close, $r_1$ and $r_2$ are close, and $r_3$ and $r_4$ are close.}
	\label{fig:weak2d_elec_con}
\end{figure}

\begin{figure}[h!]
	\centering
	\subfigure[]{\includegraphics[height=40mm]{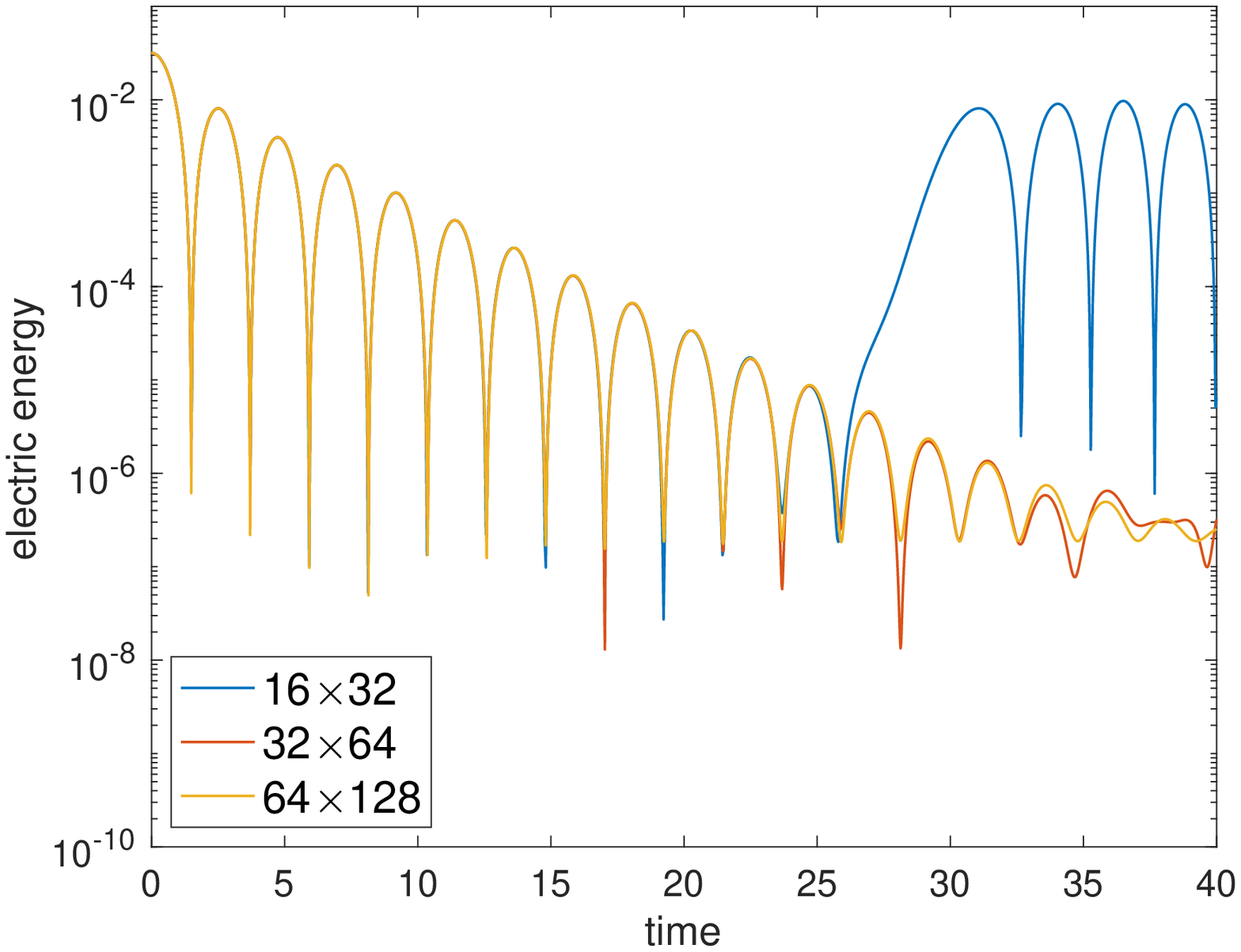}}
		\subfigure[]{\includegraphics[height=40mm]{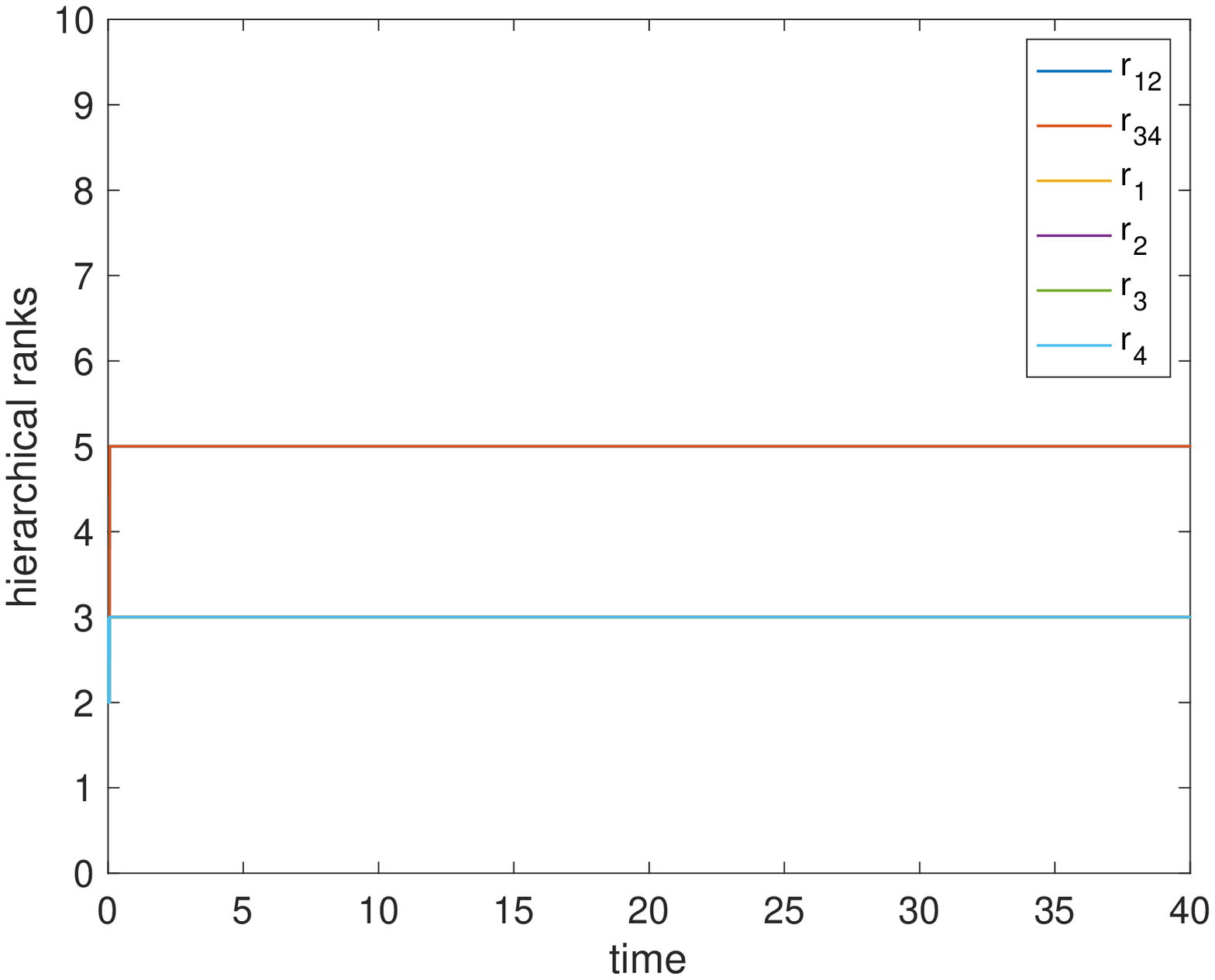}}
		\subfigure[]{\includegraphics[height=40mm]{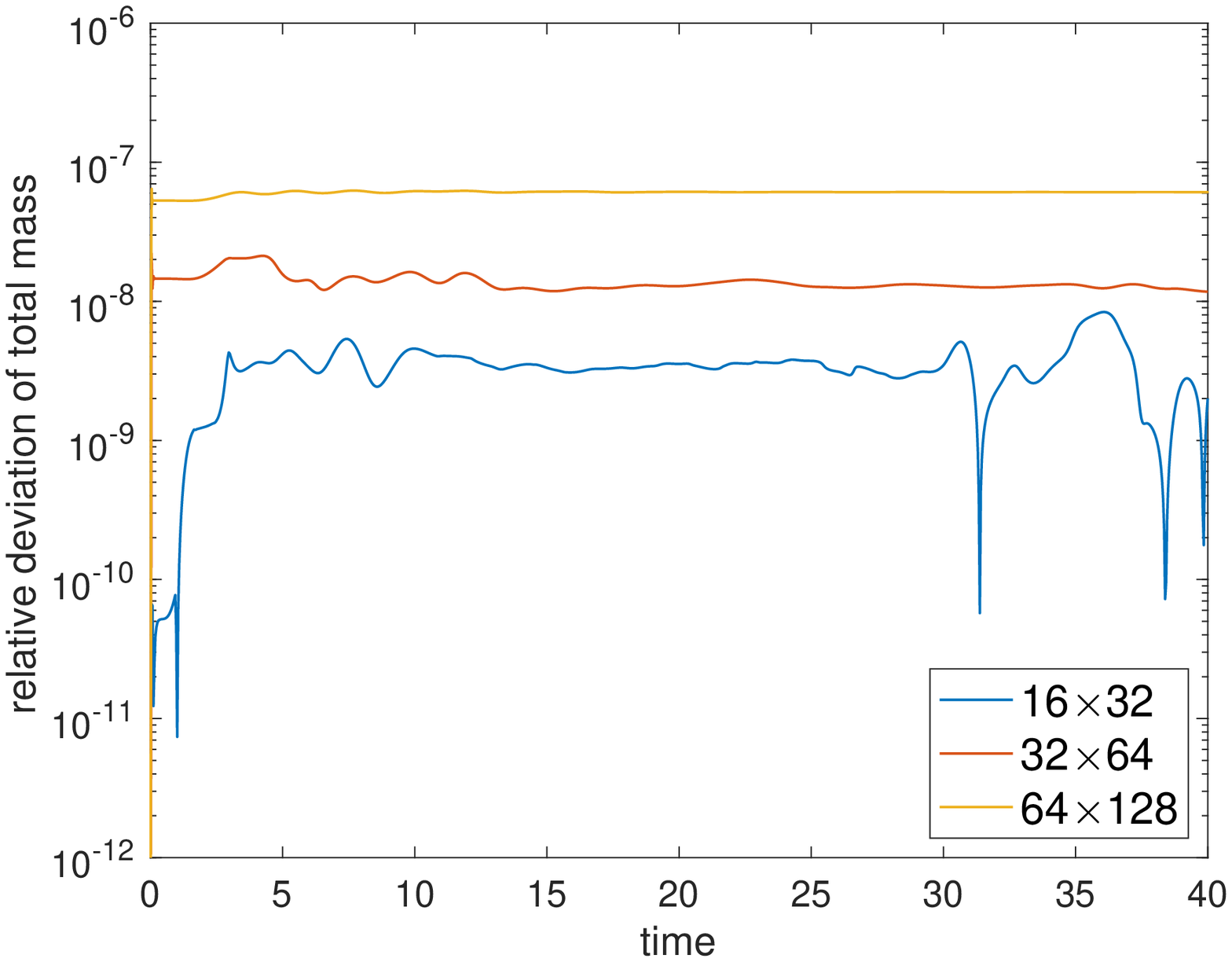}}
		\subfigure[]{\includegraphics[height=40mm]{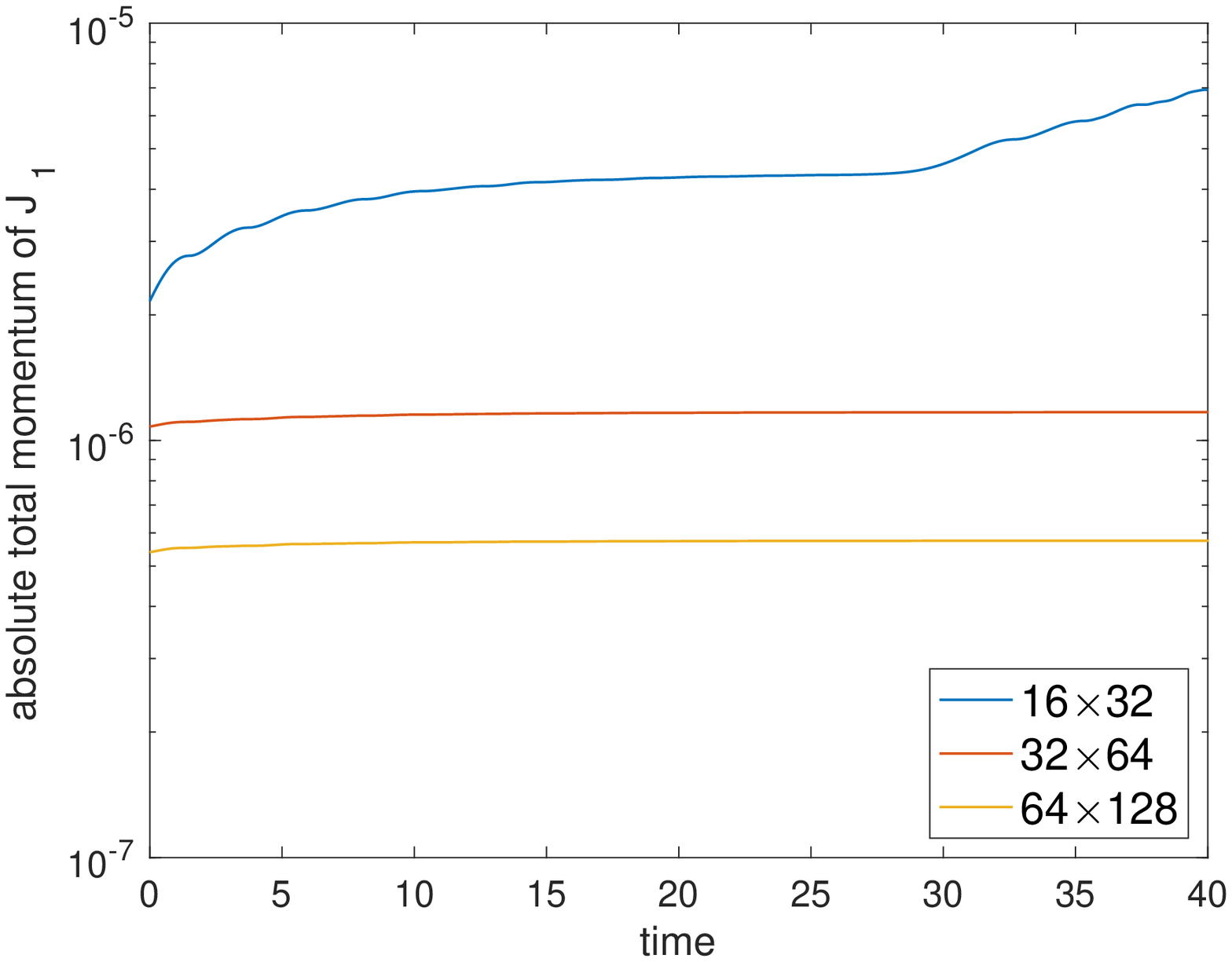}}
			\subfigure[]{\includegraphics[height=40mm]{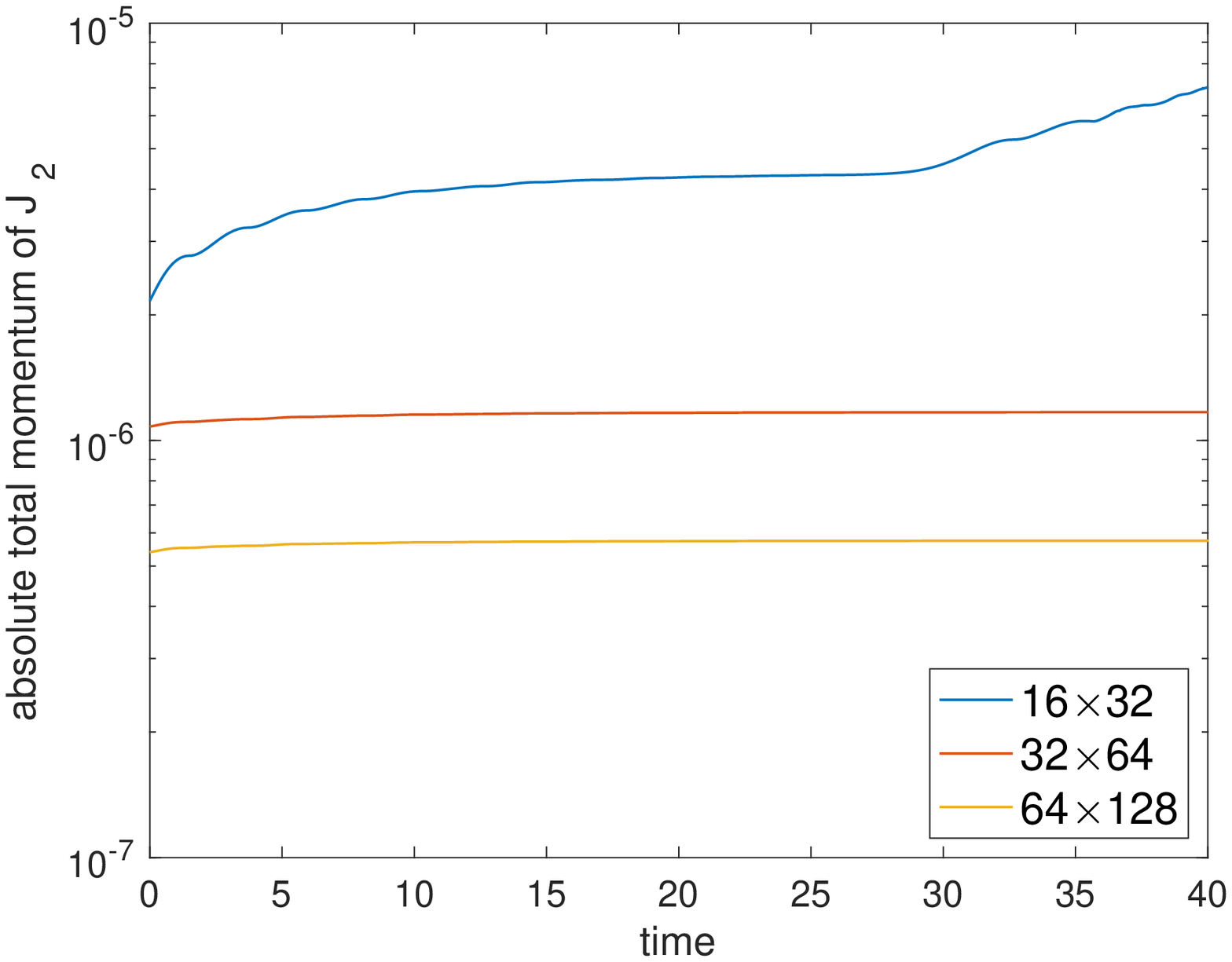}}
		\subfigure[]{\includegraphics[height=40mm]{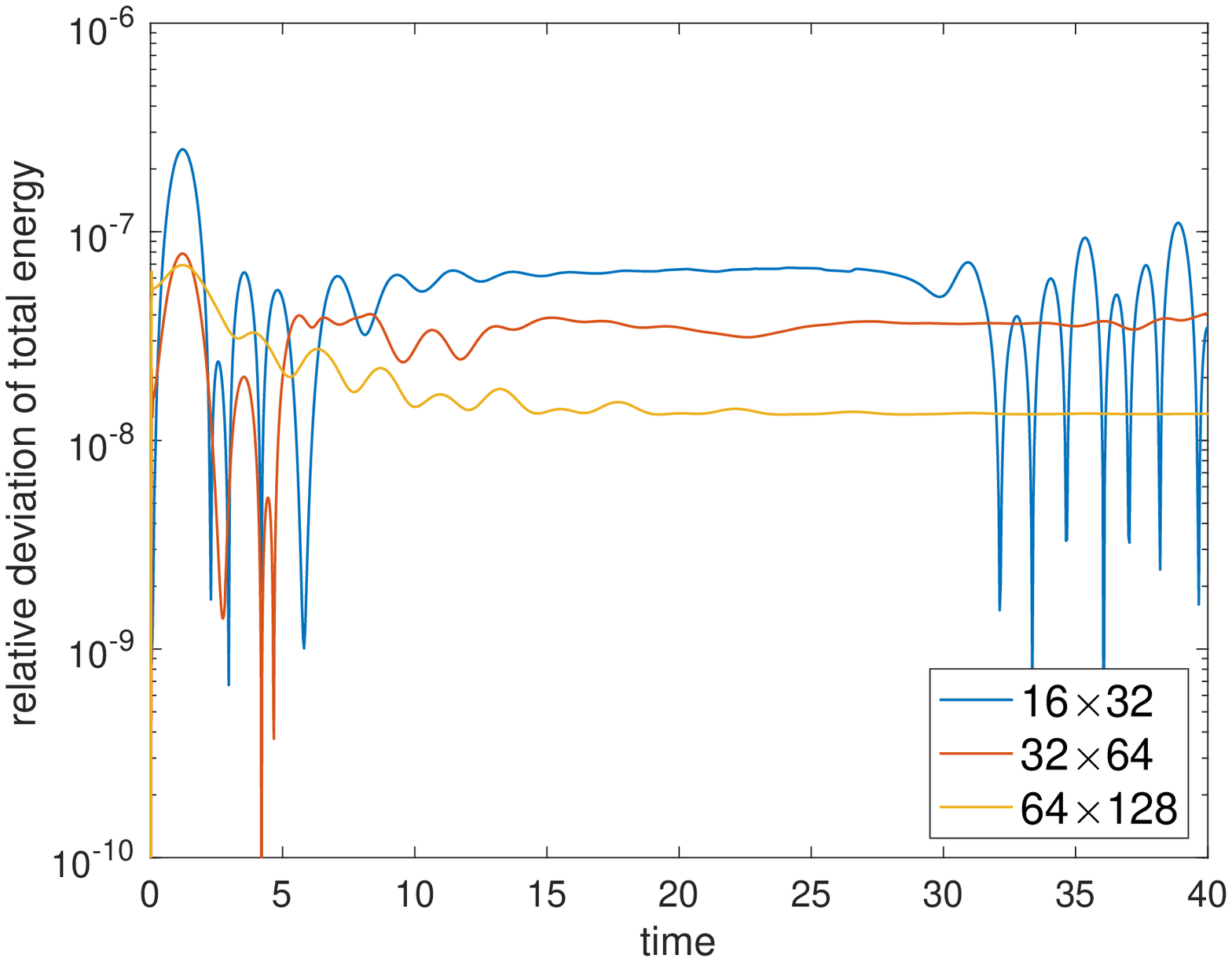}}
	\caption{Example \ref{ex:weak2d}. Non-conservative low rank method. The time evolution of  electric energy (a), hierarchical ranks of the numerical solution  of mesh size $N_x\times N_v=64\times128$ (b), relative deviation of total mass (c), absolute total momentum $J_1$ (d), absolute total momentum $J_2$ (e), and relative deviation of total energy (f). $\varepsilon=10^{-5}$. In (b), $r_{12}$ and $r_{34}$ are close. $r_1$, $r_2$, $r_3$, and $r_4$ are close.}
	\label{fig:weak2d_elec_non}
\end{figure}

  \begin{exa} \label{ex:two2d}We consider the 2D2V two-stream instability with initial condition
\begin{equation}
	\label{eq:two2d}
	f(\bx,\bv,t=0) =\frac{1}{2^d(2 \pi)^{d / 2}} \left(1+\alpha \sum_{m=1}^{d} \cos \left(k x_{m}\right)\right)\prod_{m=1}^d\left(\exp\left(-\frac{(v_m-v_0)^2}{2}\right) + \exp\left( -\frac{(v_m+v_0)^2}{2}\right)\right),
\end{equation}
where $d=2$, $\alpha=0.001$, $v_0=2.4$, and $k=0.2$.  
\end{exa}

The computation domain is set as $[0,L_x]^2\times[-L_v,L_v]^2$, where $L_x=\frac{2\pi}{k}$ and $L_v=8$, and the truncation threshold is set as $\varepsilon=10^{-5}$.  In Figures \ref{fig:two2d_elec_con}-\ref{fig:two2d_elec_non}, respectively for conservative and non-conservative methods, we report the time evolution of the electric energy, hierarchical ranks of the numerical solution,  relative deviation of total mass and energy together with absolute total momentum $J_1$ and $J_2$. The observation is similar to the previous example that the proposed conservative method is able to conserve the total mass and momentum, and meanwhile, the hierarchical ranks of the solution tensor from the conservative method are larger than that from the non-conservative counterpart.

\begin{figure}[h!]
	\centering
	\subfigure[]{\includegraphics[height=40mm]{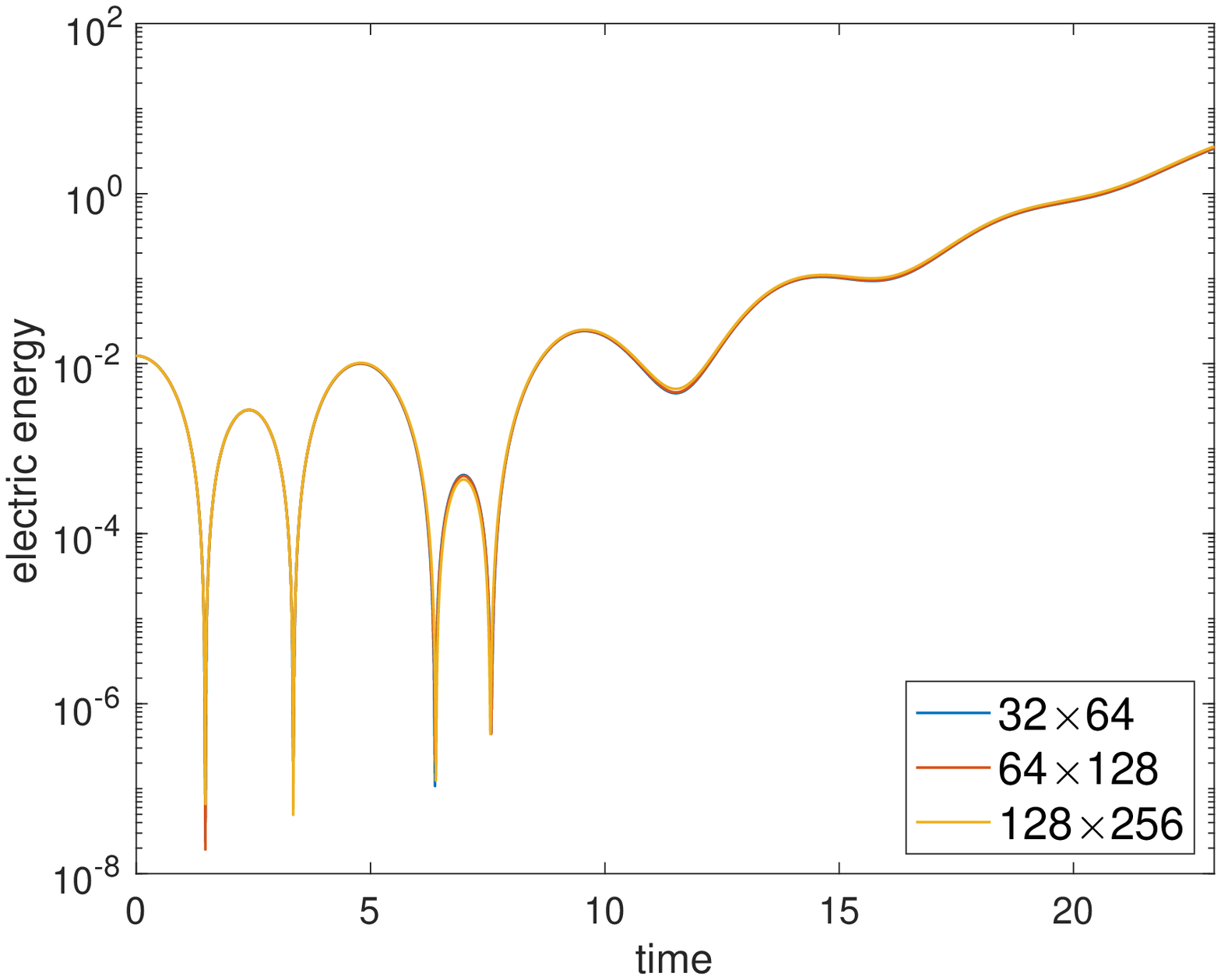}}
		\subfigure[]{\includegraphics[height=40mm]{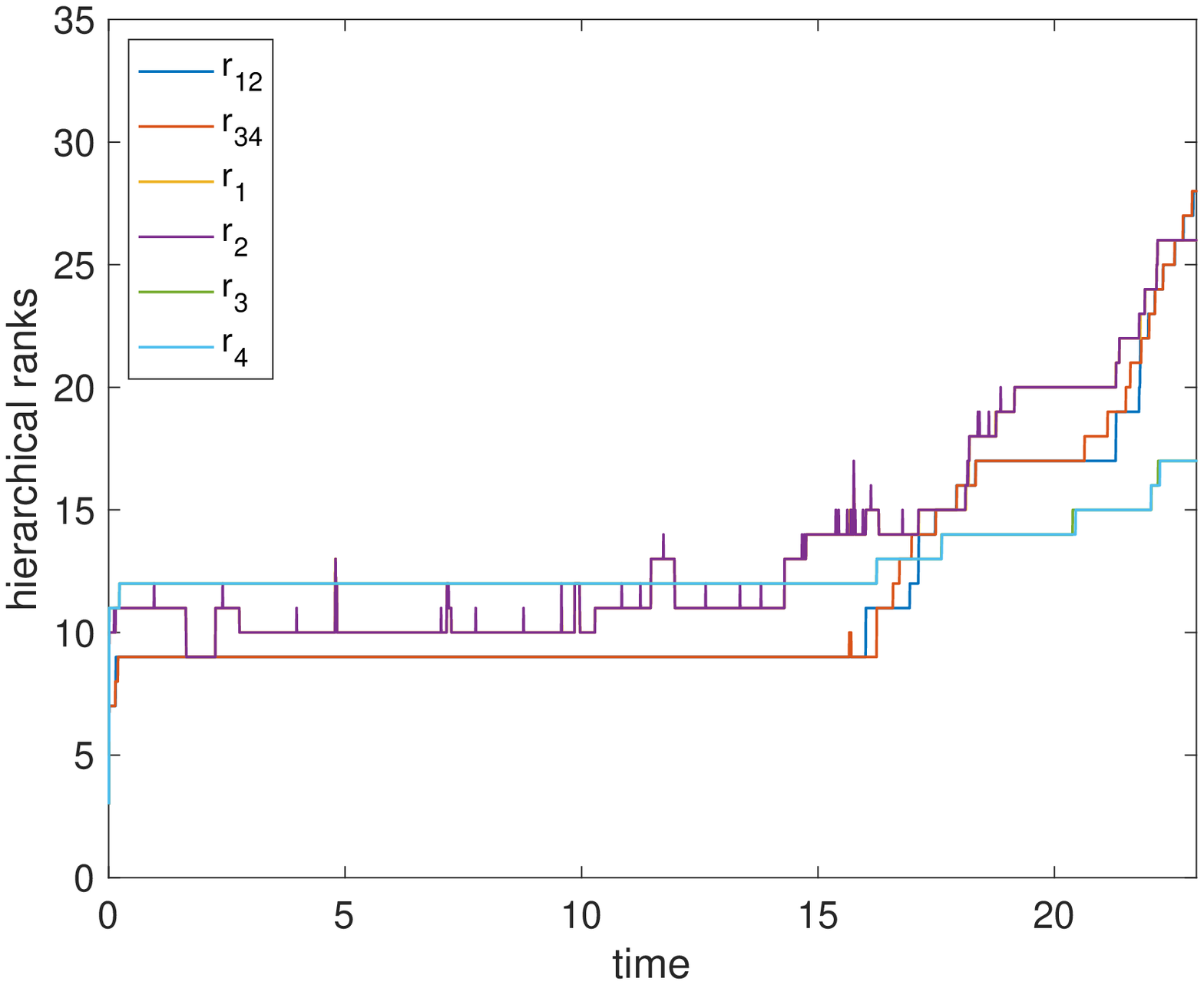}}
		\subfigure[]{\includegraphics[height=40mm]{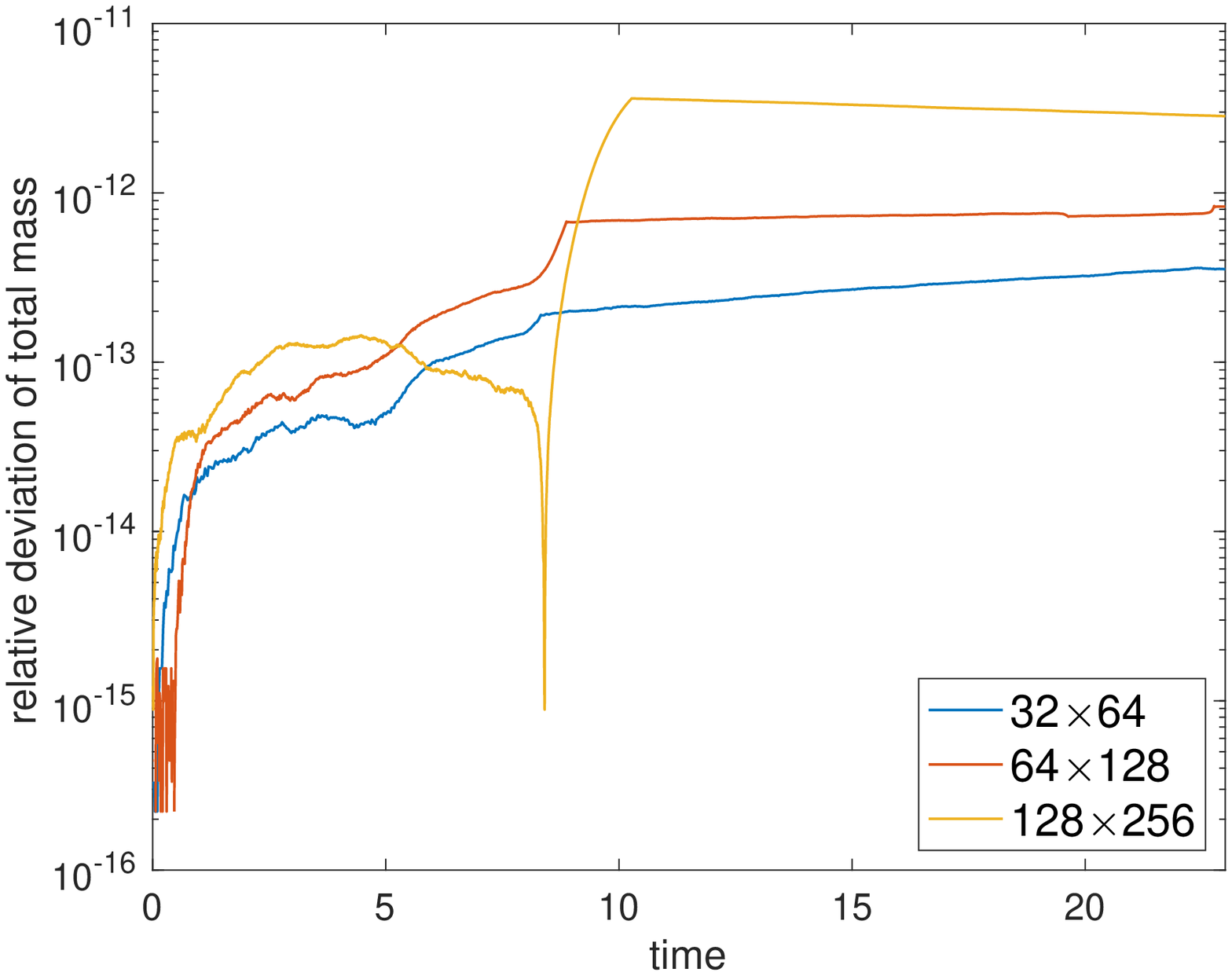}}
		\subfigure[]{\includegraphics[height=40mm]{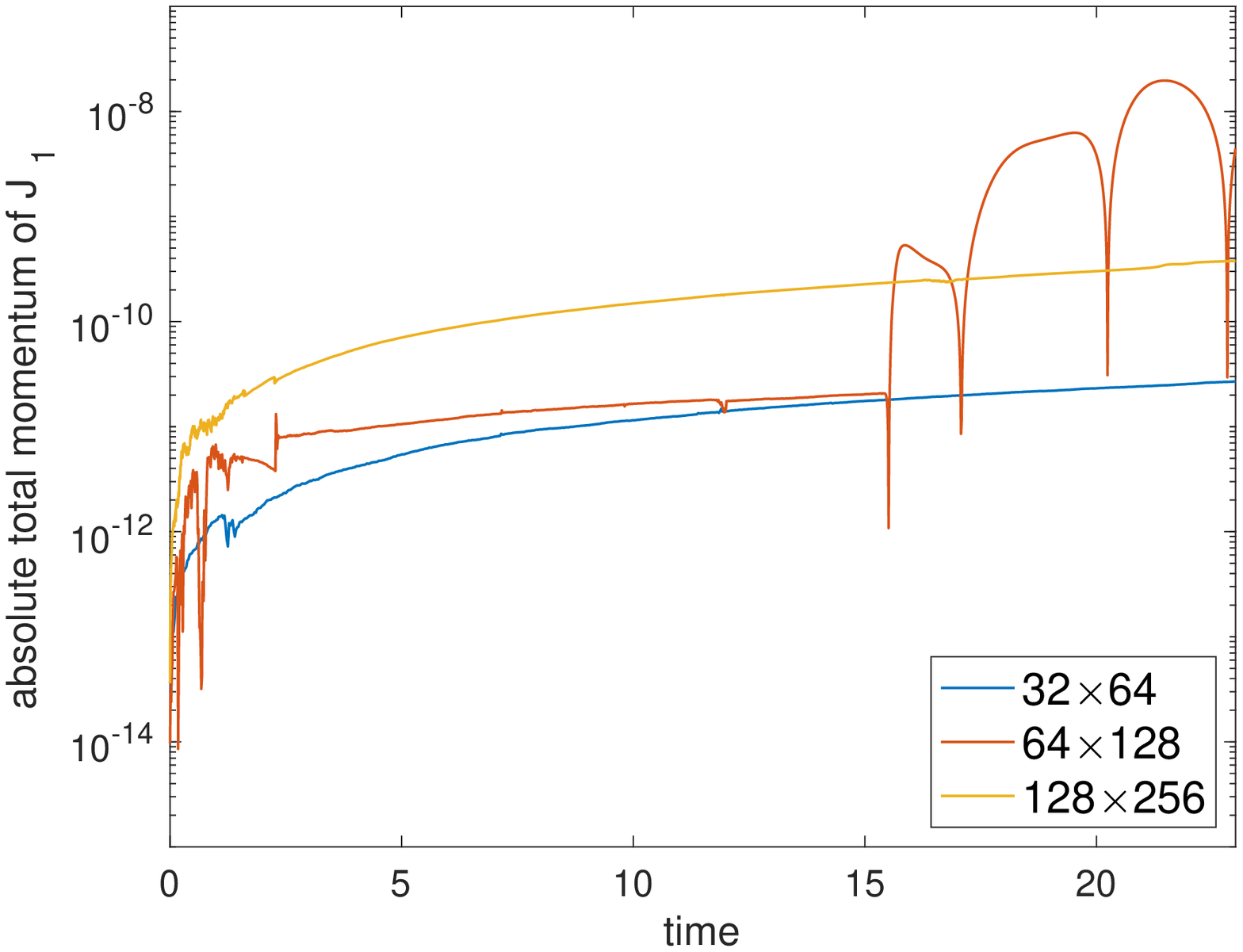}}
			\subfigure[]{\includegraphics[height=40mm]{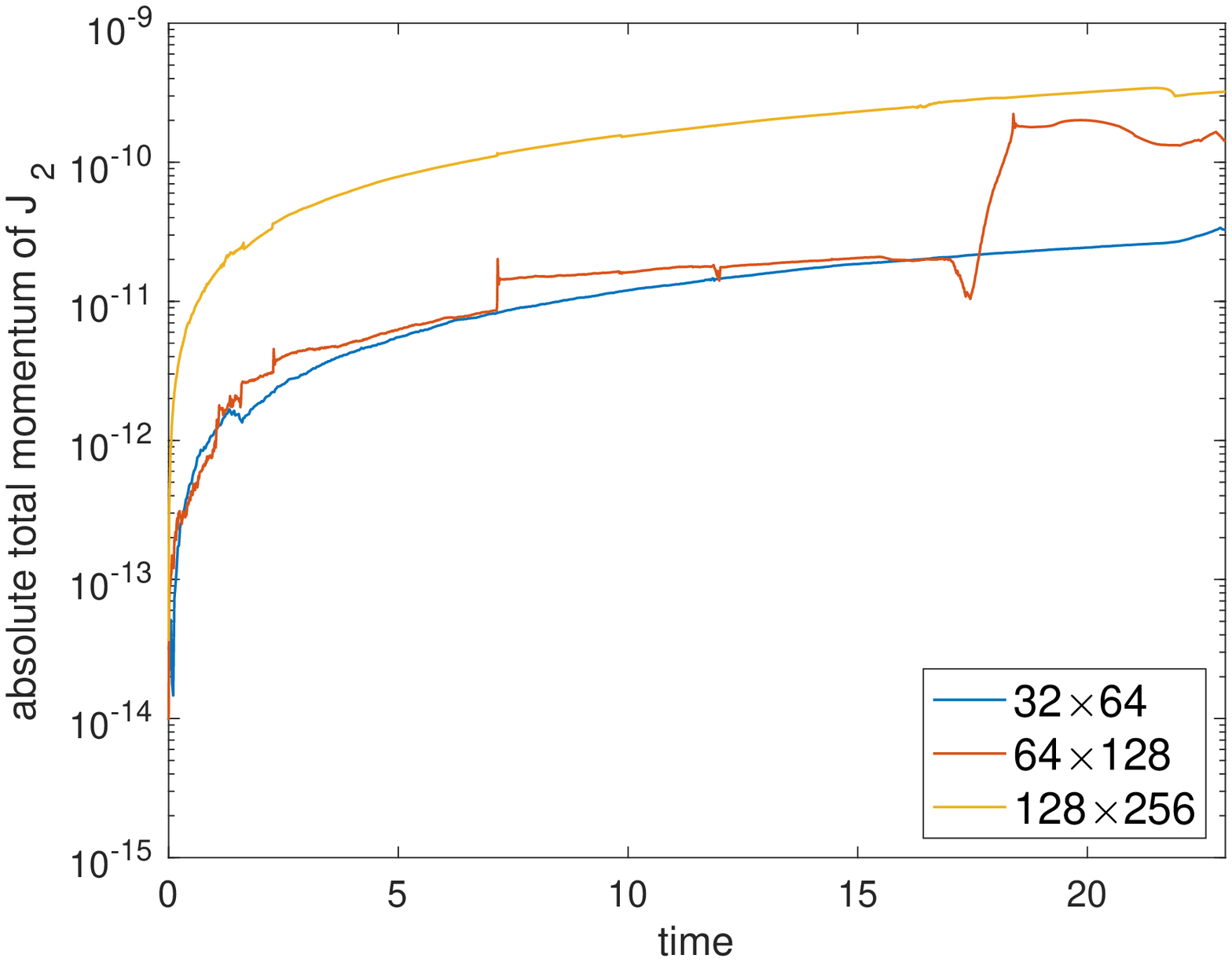}}
		\subfigure[]{\includegraphics[height=40mm]{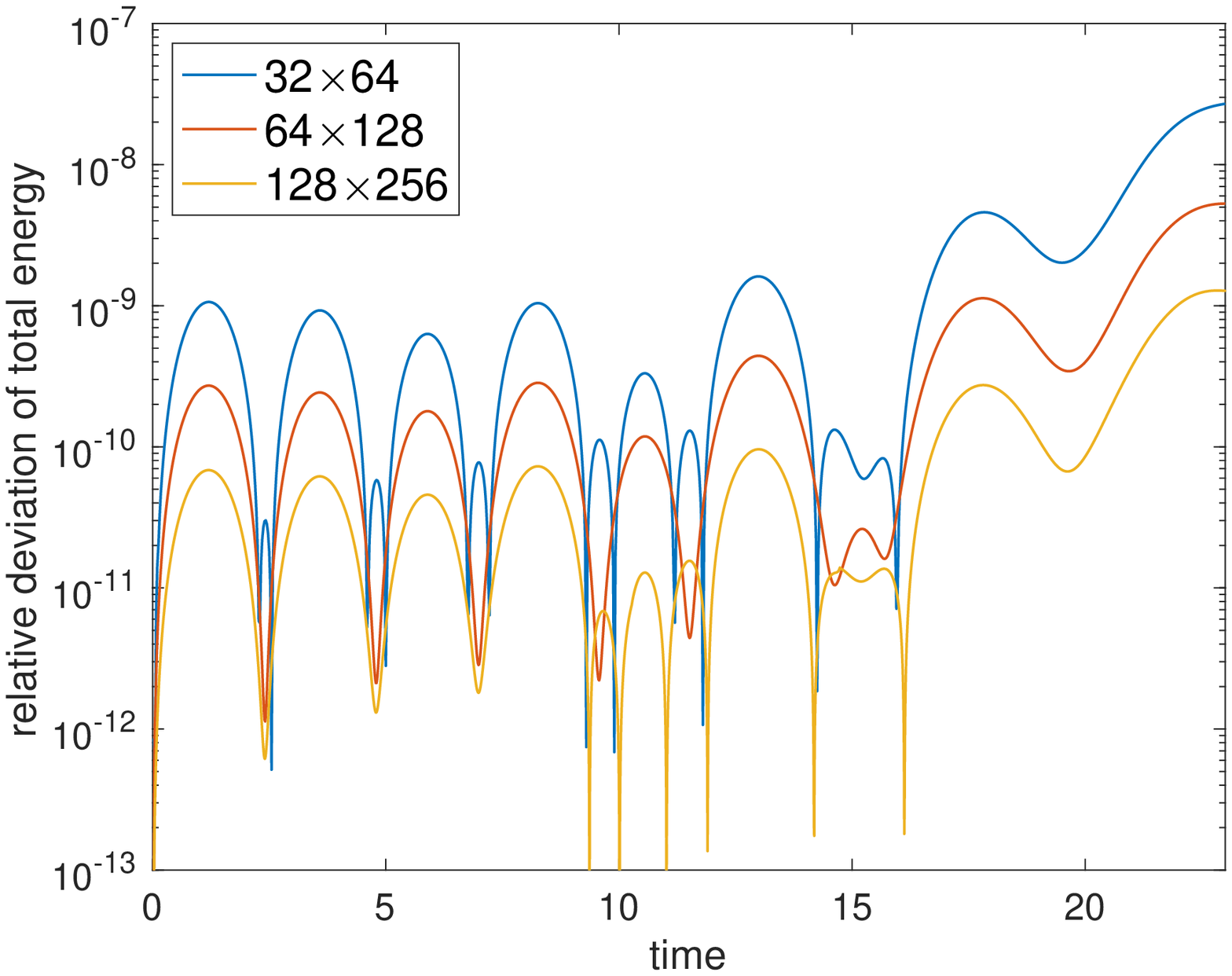}}
	\caption{Example \ref{ex:two2d}. Conservative low rank method. The time evolution of the  electric energy (a), hierarchical ranks of the numerical solutions (b), relative deviation of total mass (c), absolute total momentum $J_1$ (d), absolute total momentum $J_2$ (e), and relative deviation of total energy (f). $\varepsilon=10^{-5}$.}
	\label{fig:two2d_elec_con}
\end{figure}

\begin{figure}[h!]
	\centering
	\subfigure[]{\includegraphics[height=40mm]{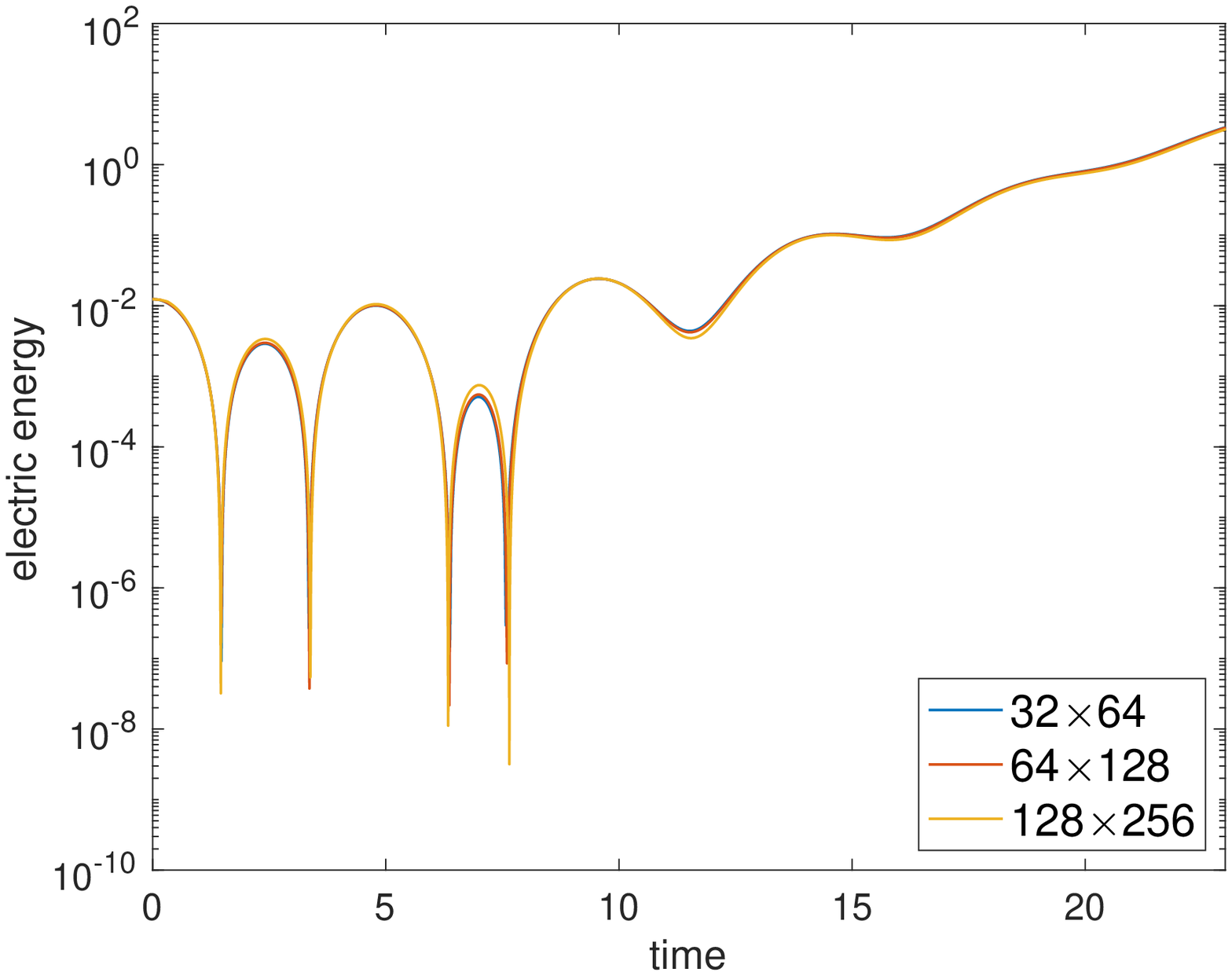}}
		\subfigure[]{\includegraphics[height=40mm]{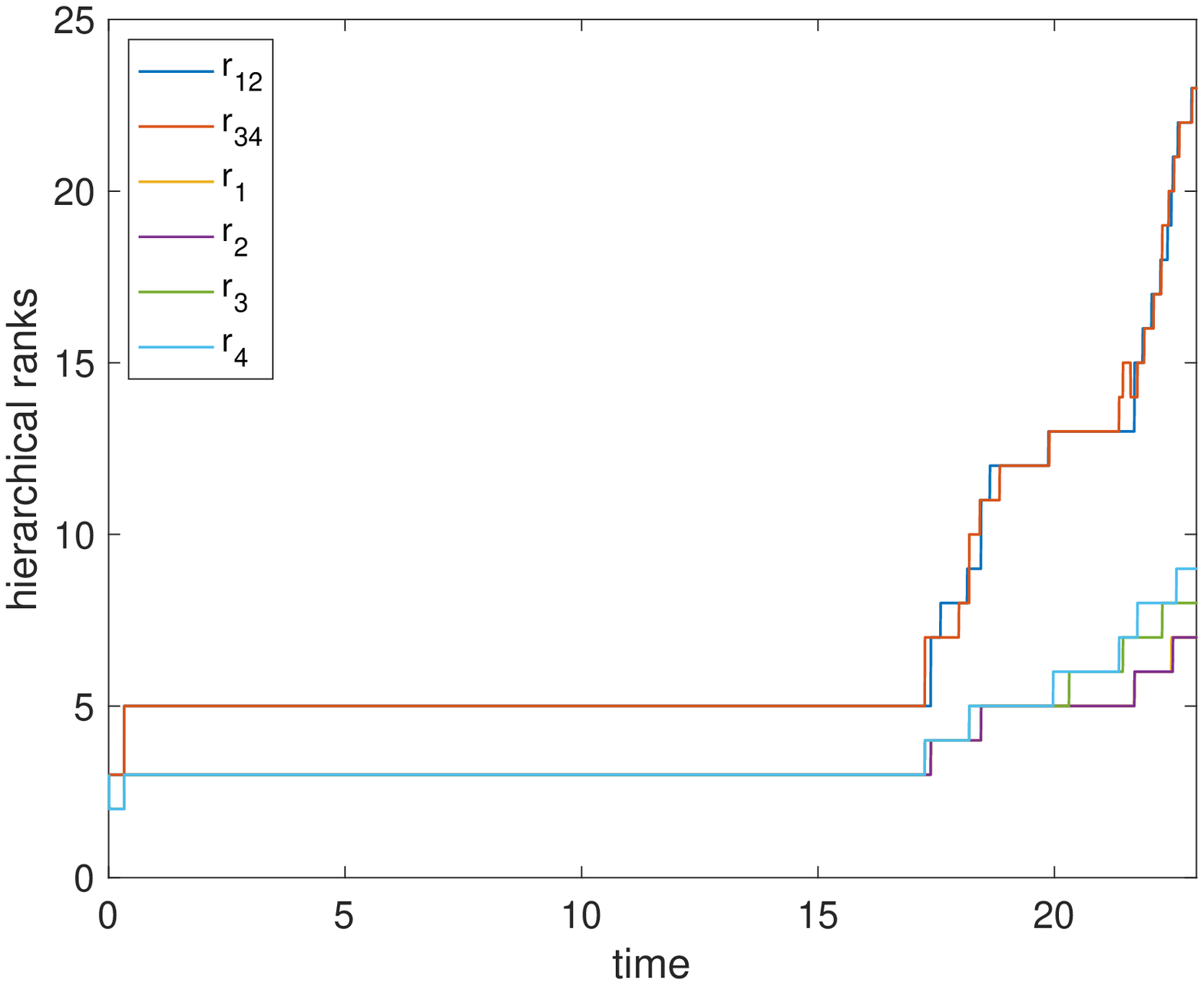}}
		\subfigure[]{\includegraphics[height=40mm]{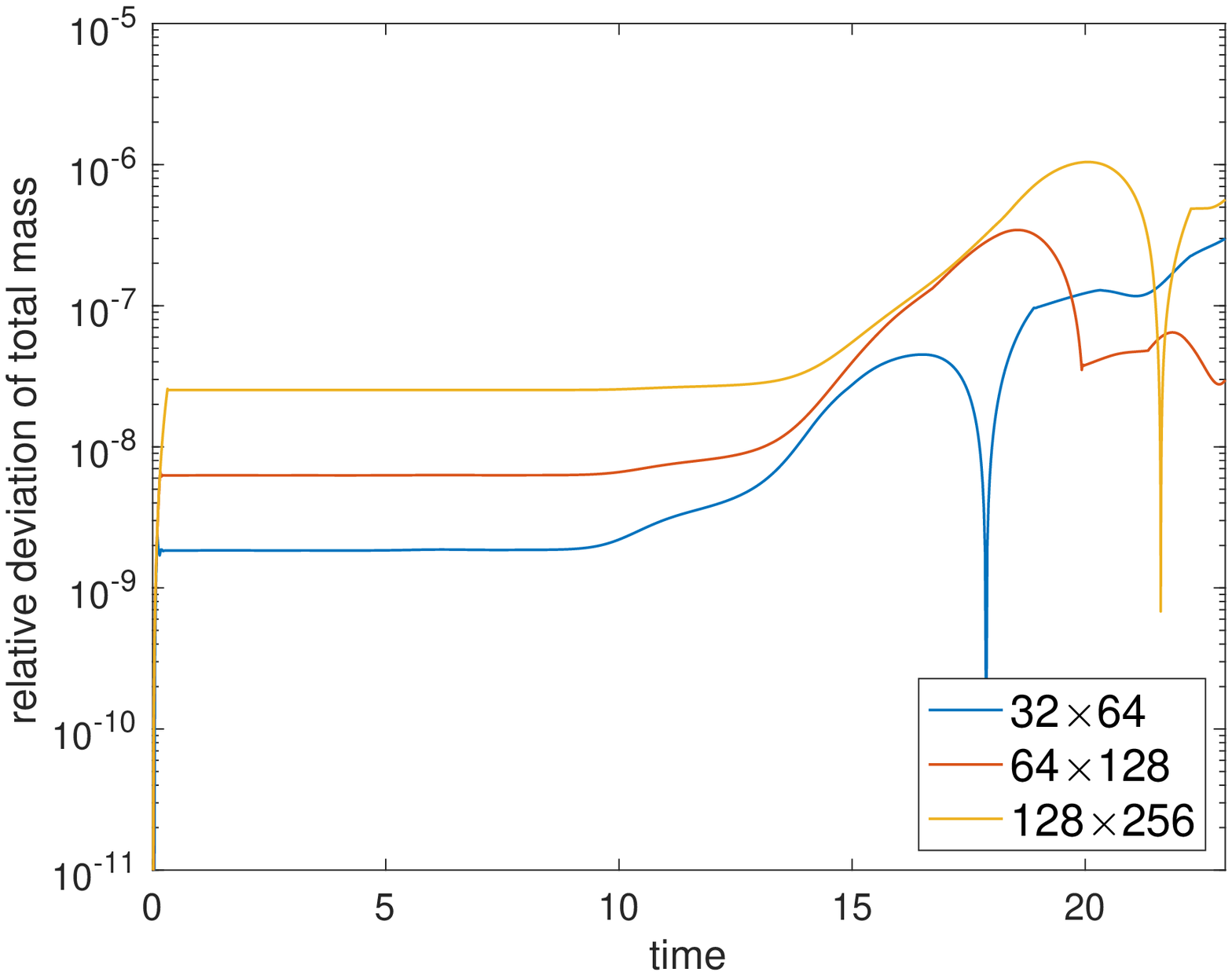}}
		\subfigure[]{\includegraphics[height=40mm]{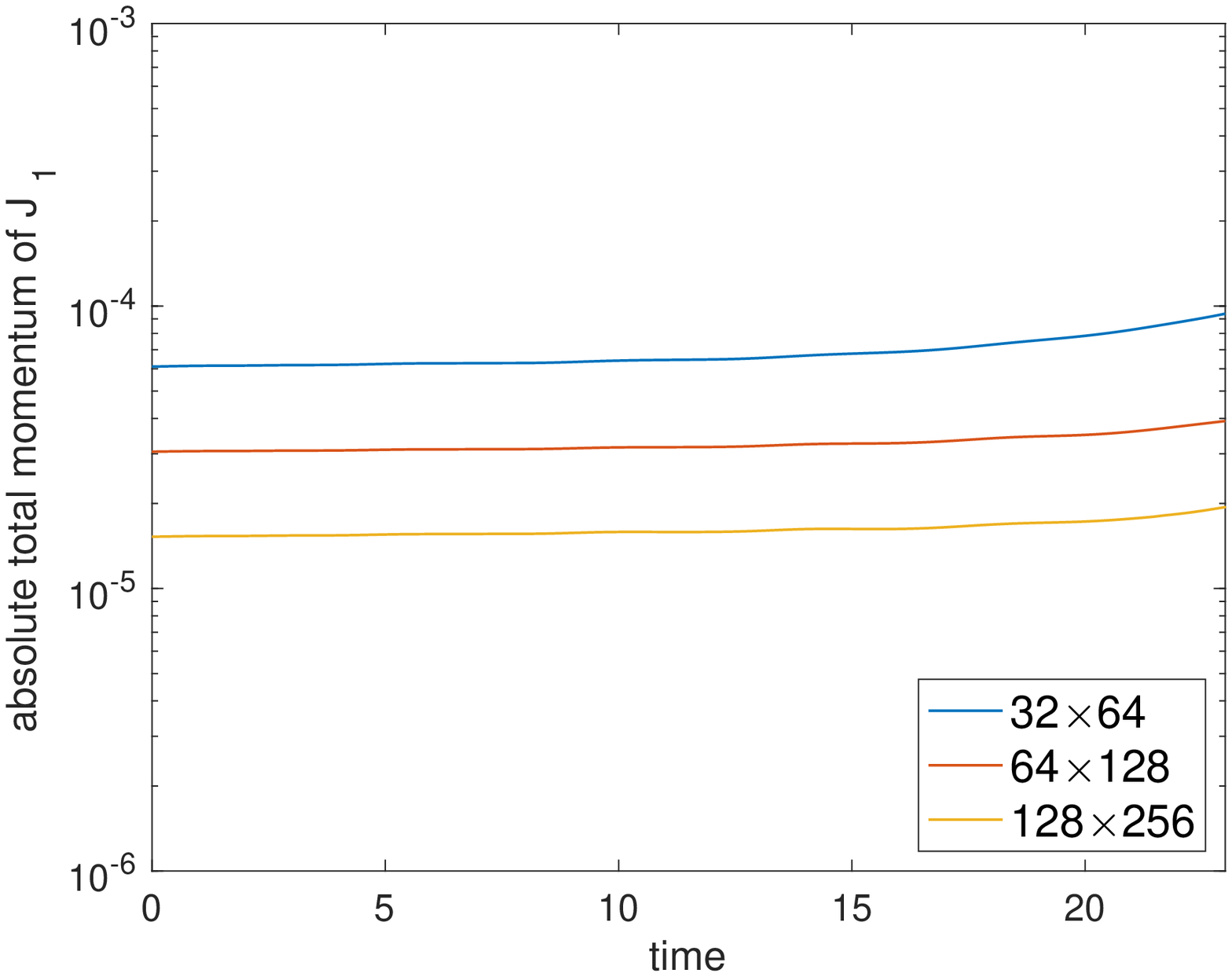}}
			\subfigure[]{\includegraphics[height=40mm]{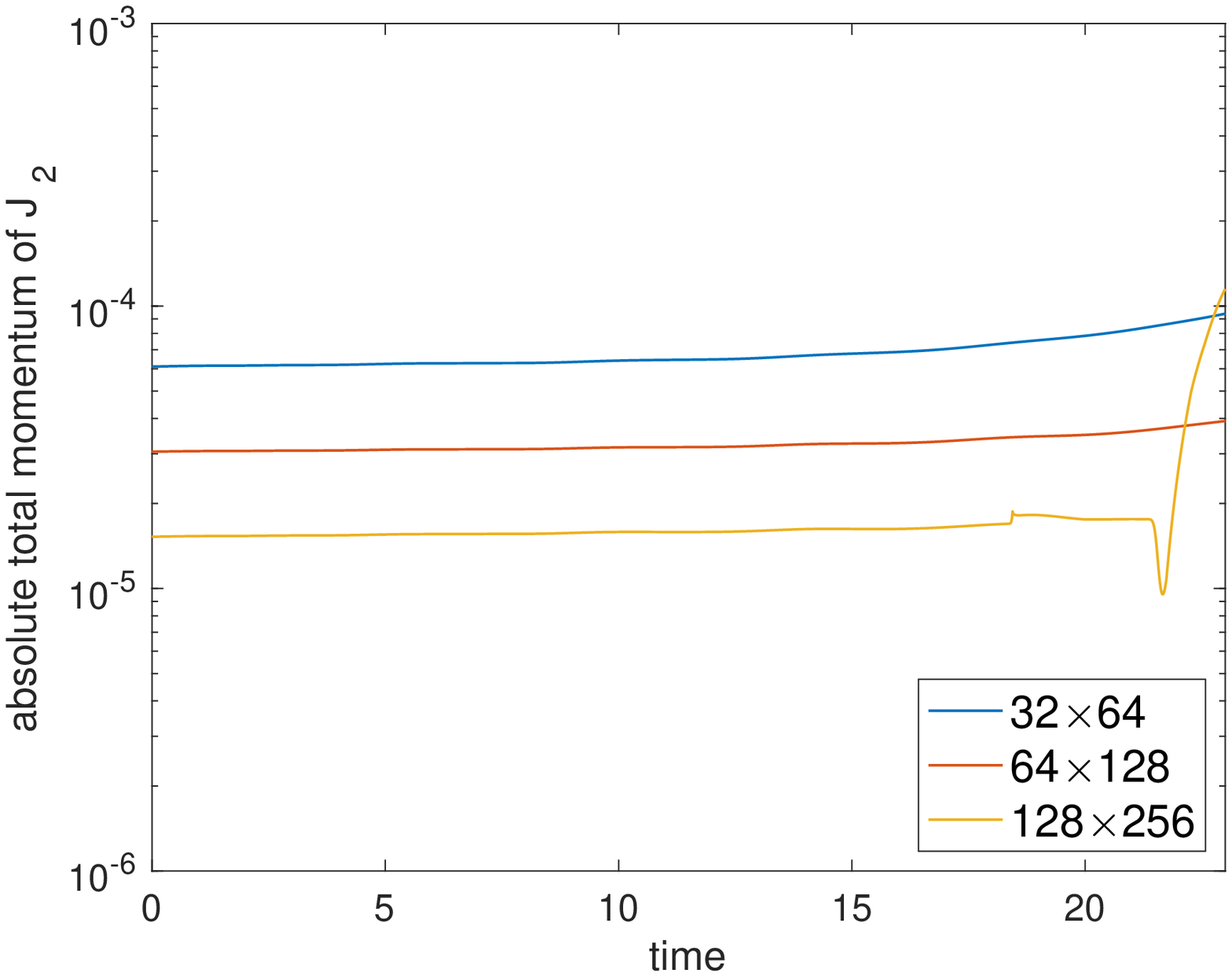}}
		\subfigure[]{\includegraphics[height=40mm]{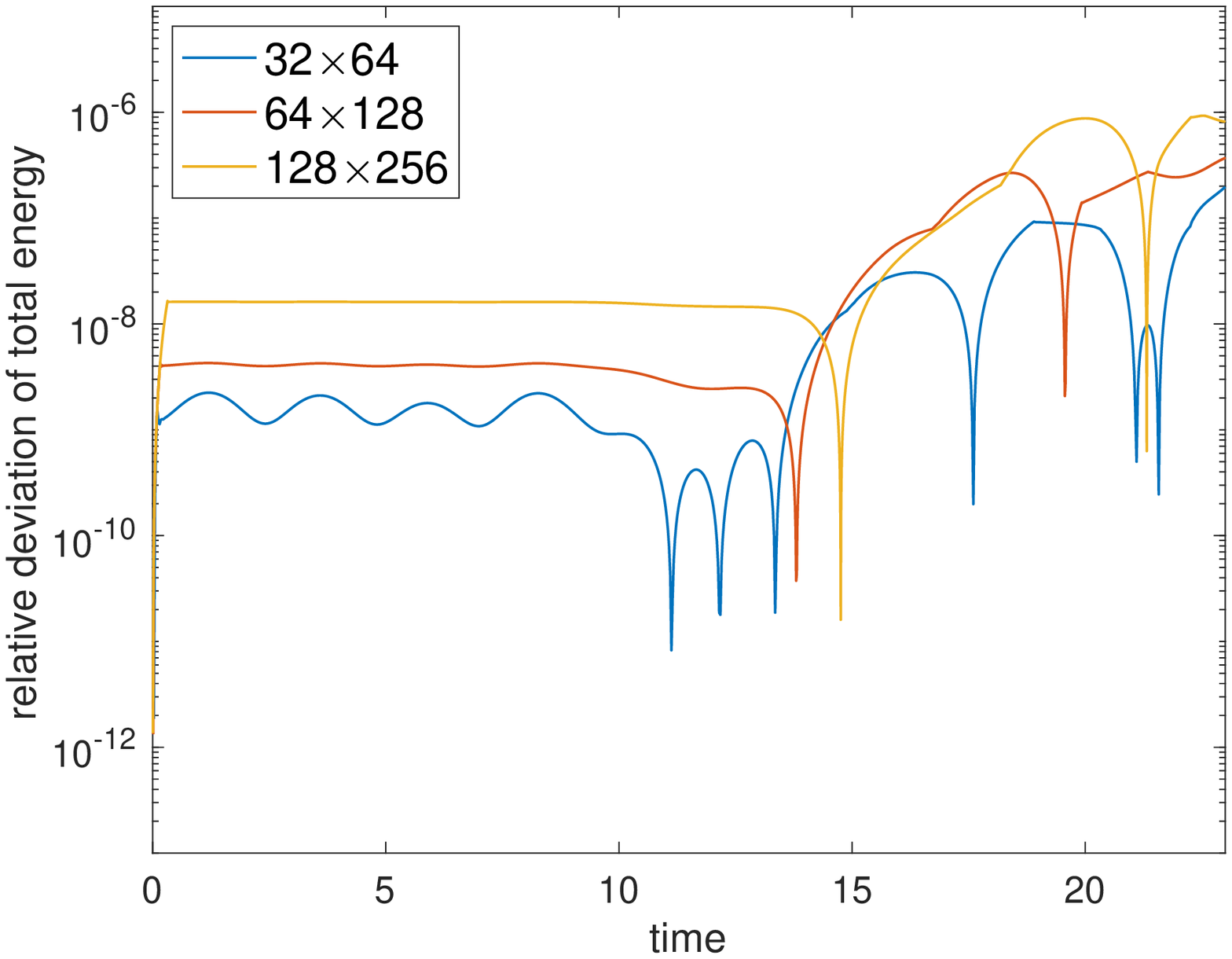}}
	\caption{Example \ref{ex:two2d}. Non-conservative low rank method. The time evolution of  electric energy (a), hierarchical ranks of the numerical solutions (b), relative deviation of total mass (c), absolute total momentum $J_1$ (d), absolute total momentum $J_2$ (e), and relative deviation of total energy (f). $\varepsilon=10^{-5}$.}
	\label{fig:two2d_elec_non}
\end{figure}

%% file: conclusion.tex
\section{Conclusion}
\setcounter{equation}{0}
\setcounter{figure}{0}
\setcounter{table}{0}

In this paper, we proposed a conservative truncation procedure for a low-rank tensor approach for performing a grid-based Vlasov simulations. The basic idea is initialized in the 1D1V setting, and is further developed to the 2D2V setting with the HT tensor decompositions. The newly developed conservative low rank tensor algorithm is theoretically proved to be a locally conservative scheme to the macroscopic equations for charge and current densities, and is numerically verified to globally conserve the total charge and current. Further development of the low rank tensor algorithm with local energy conservation is subject to our future work.